\documentclass[10pt]{amsart}
\usepackage{latexsym, amsmath, amssymb, longtable, booktabs,amscd, graphicx}

\usepackage[english]{babel}
\usepackage[utf8]{inputenc}

\usepackage{hyperref}
\usepackage{enumerate}

\usepackage[all]{xy}

\numberwithin{equation}{section}
\newcommand{\St}{\operatorname{St}}
\newcommand{\matr}[4]{\left(\begin{smallmatrix}#1 & #2 \\ #3 & #4\end{smallmatrix}\right)}
\newcommand{\eps}{\varepsilon}

\renewcommand{\P}{\mathbb{P}}
\newcommand{\ind}{\operatorname{ind}}
\renewcommand{\O}{\mathcal{O}}
\newcommand{\m}{\mathfrak{m}}

\newcommand{\Reptors}{\operatorname{Rep}_{\text{tors}}}
\newcommand{\RepOE}{\operatorname{Rep}_{\O_E}}
\newcommand{\RepE}{\operatorname{Rep}_{E}}
\newcommand{\JH}{\operatorname{JH}}

\newcommand\Id{\mathrm{Id}}

\newcommand\GL{{\mathrm{GL}}}

\newcommand\Qbar{{\bar{\Q}}}

\newcommand{\Q}{\mathbb{Q}}
\newcommand{\Z}{\mathbb{Z}}
\newcommand{\Zbar}{{\bar{\Z}}}

\newcommand{\F}{\mathbb{F}}

\newcommand\Hom{{\mathrm {Hom}}}

\newcommand\Sym{{\mathrm {Sym}}}

\newcommand\iso{{\> \simeq \>}}
 
\newcommand\linee{{
        {\begin{tiny}
        \begin{xymatrix}{
         \bullet  \ar@{-}[d] \\ 
         \bullet  \ar@{-}[d] \\
         \bullet} 
       \end{xymatrix}
       \end{tiny}} }}  
\newcommand\diamondd{{
        \begin{tiny}
        \begin{xymatrix}{
         & \bullet  \ar@{-}[dl] \ar@{-}[dr] &  \\ 
         \bullet  \ar@{-}[dr] &   &   \bullet  \ar@{-}[dl] \\
         & \bullet & } 
       \end{xymatrix}
       \end{tiny} }}

\newtheorem{thm}{Theorem}[section]
\newtheorem{theorem}[thm]{Theorem}
\newtheorem{cor}[thm]{Corollary}

\newtheorem{prop}[thm]{Proposition}
\newtheorem{lemma}[thm]{Lemma}
\theoremstyle{definition}
\newtheorem{definition}[thm]{Definition}
\theoremstyle{remark}
\newtheorem{remark}[thm]{Remark}

\theoremstyle{remark}

\begin{document}

\title[Reductions of Galois Representations of slope 1]{Reductions of Galois representations of slope 1}

\author[S. Bhattacharya]{Shalini Bhattacharya}
\address{Department of Mathematics, Bar-Ilan University, Ramat Gan-5290002, Israel}
\email{shaliniwork16@gmail.com}

\author[E. Ghate]{Eknath Ghate} 
\address{School of Mathematics, Tata Institute of Fundamental Research, Homi Bhabha Road, Mumbai-5, India}
\email{eghate@math.tifr.res.in}

\author[S. Rozensztajn]{Sandra Rozensztajn}
\address{UMPA, \'ENS de Lyon, UMR 5669 du CNRS, 46, all\'ee d’Italie, 69364 Lyon Cedex 07, France}
\email{sandra.rozensztajn@ens-lyon.fr}

\begin{abstract}
We compute the reductions of irreducible crystalline two-dimensional
representations of $G_{\Q_p}$ of slope 1, for primes $p \geq 5$, and all
weights. We describe the semisimplification of the reductions completely.
In particular, we show that the reduction is often reducible.
We also investigate whether the extension obtained is peu or tr\`es
ramifi\'ee, in the relevant reducible non-semisimple 
cases. The proof
uses the compatibility between the $p$-adic and mod $p$ Local Langlands
Correspondences, and involves a detailed study of the reductions of both
the standard and non-standard 
lattices in certain $p$-adic Banach spaces. 
\end{abstract}

\maketitle



\section{Introduction}

Let $p$ be an odd prime. This paper is concerned with computing the
reductions of certain crystalline two-dimensional representations of the
local Galois group $G_{\Q_p}$.  The first computations of the reduction
in positive slope, after Edixhoven \cite{Edixhoven92}, were carried out by Breuil
in \cite{Br03}, for weights at most $2p+1$.  The reductions are also
known for slopes which are large compared to the weight by Berger-Li-Zhu
\cite{BLZ}; see also \cite{YY} for results using similar techniques. In
the other direction, the reductions have also recently been computed for
small fractional slopes, namely, for slopes in $(0,1)$ by Buzzard-Gee
\cite{BG09}, \cite{BG13}, and for slopes in $(1,2)$ in
\cite{GG15}, \cite{Bhattacharya-Ghate}, under a mild hypothesis.  In this paper we
compute the reduction in the important missing case of integral slope 1.
As far as we know, the shape of the reduction is not known for all
weights, for any other (positive) integral slope. 

Let us introduce some notation.
Let $E$ be a finite extension field of $\Q_p$ and let $v$ be the
valuation of $\bar\Q_p$ normalized so that $v(p) = 1$. Let $a_p \in E$
with $v(a_p) > 0$ and let $k \geq 2$.  Let $V_{k,a_p}$ be the irreducible
crystalline representation of $G_{\Q_p}$ with Hodge-Tate weights
$(0,k-1)$ and slope $v(a_p) > 0$ such that $D_\mathrm{cris}(V_{k,a_p}^*)
= D_{k,a_p}$, where $D_{k,a_p}  = E e_1 \oplus E e_2$ is the filtered
$\varphi$-module as defined in \cite[\S 2.3]{Berger11}.  The
semisimplification $\bar{V}_{k,a_p}^{ss}$ of the reduction
$\bar{V}_{k,a_p}$ with respect to a lattice in ${V}_{k,a_p}$ is
independent of the choice of the lattice.  Let $\omega = \omega_1$ and
$\omega_2$ denote the fundamental characters of level 1 and 2
respectively. Let $\mathrm{ind}(\omega_2^{a})$ denote the  representation
of $G_{\Q_p}$ obtained by inducing the character $\omega_2^a$, for $a \in
\Z$, from $G_{\Q_{p^2}}$ to  $G_{\Q_p}$; it is irreducible if $p+1 \nmid
a$.  Finally, let $\mu_x$ be the unramified character of $G_{\Q_p}$
taking (geometric) Frobenius at $p$ to $x \in \bar\F_p^\times$. 

The following theorem describes the reduction $\bar{V}_{k,a_p}^{ss}$ when
the slope $v(a_p)$ is equal to $1$, for all primes $p > 3$.

\begin{theorem}
 \label{maintheoslopeone}
  Let $p>3$, let $k\geq 2p+2$ and let $r = k-2 \equiv b \mod (p-1)$, with $ 2 \leq b\leq p$. Suppose that the slope $v(a_p)=1$.  
  Then $\bar{V}_{k, a_p}^{ss}$ is as follows:
\begin{eqnarray*}
b = 2 & \implies & 
\begin{cases}
    \mathrm{ind}(\omega_2^{b+1}),  & \text{if }  v\left(\dfrac{a_p}{p} -\dbinom{r}{2} \dfrac{p}{a_p} \right)< v(r-2) \vspace{.2cm} \\
	\mu_{\lambda} \cdot \omega^b\,\oplus\,\mu_{\lambda^{-1}} \cdot \omega, &
\text{if }  v\left(\dfrac{a_p}{p} -\dbinom{r}{2}\dfrac{p}{a_p}\right) = v(r-2),\> \text{ with }   
\lambda = \overline{\dfrac{2}{2-r}\left(\dfrac{a_p}{p} -\dbinom{r}{2}\dfrac{p}{a_p}\right)} \vspace{.2cm} \\ 
	\mathrm{ind}(\omega_2^{b+p}), & \text{if } v\left(\dfrac{a_p}{p} -\dbinom{r}{2} \dfrac{p}{a_p} \right)>v(r-2),       
\end{cases} \\
3\leq b\leq p-1 & \implies & 
\begin{cases}
	\mu_{\lambda} \cdot \omega^b\oplus \mu_{\lambda^{-1}} \cdot
\omega, &\text{if } p\nmid r-b, \text{ with } \lambda
=\overline{\dfrac{b}{b-r} \cdot
\dfrac{a_p}{p}}\in\bar{\mathbb{F}}_p^\times\\
	\mathrm{ind}(\omega_2^{b+1}), &\text{if } p\mid r-b,
\end{cases} \\
b=p & \implies & 
\begin{cases}
	\mathrm{ind}(\omega_2^{b+p}), &\text{if } p\nmid r-b\\
	\mu_{\lambda} \cdot \omega\,\oplus\, \mu_{\lambda^{-1}} \cdot \omega, &\text{if } p\mid r-b, 
\text{ with } \lambda + \dfrac{1}{\lambda} = \overline{\dfrac{a_p}{p}-\dfrac{r-p}{a_p}}\in\bar{\mathbb{F}}_p.
\end{cases}     
\end{eqnarray*}
\end{theorem}

\begin{remark}
We have:
\begin{itemize}
   \item Theorem~\ref{maintheoslopeone} describes the reduction $\bar{V}_{k,a_p}^{ss}$ completely for slope $v(a_p) = 1$ and primes $p \geq 5$.
The results in the theorem match with all previously known results for weights $k \leq 2p+1$ 
summarized in \cite{Berger11}, 
except for $k = 2$, where our theorem does not give the correct result, 
and for the cases $(k, a_p)=(4,\pm p)$, where the theorem does not
directly apply as the quantity $\lambda$ is undefined.\footnote{Recent work by Arsovski \cite{Arsovski}  includes
a study of the reduction $\bar{V}_{k,a_p}^{ss}$ when the slope is 1, for
$b \neq 3$, $p$, though the reduction is not  uniquely specified there.} So, the theorem could have been stated 
for all weights excluding these cases.%
\item It is a classical fact that the reduction is 
reducible in the ordinary case, that is, when the slope is $0$.
The theorem shows that the reduction
$\bar{V}_{k,a_p}^{ss}$ is often reducible when the slope is
$1$. For example, the reduction is reducible
for the congruence classes $3 \leq b \leq p-1$ if $p \nmid r - b$.
This behaviour at the integral slopes $\{0\} \cup \{1\}$
contrasts with the behaviour of the reduction for fractional
slopes $v(a_p) \in (0,1) \cup (1,2)$, where $\bar{V}_{k,a_p}$ is mostly irreducible.
\item However, when $v(a_p) = 1$, the theorem shows that in each congruence class of
weights mod $(p-1)$ there are further congruence classes of weights mod
$p$ where $\bar{V}_{k,a_p}$ is irreducible.
 \item A surprising trichotomy occurs for $b = 2$. This 
           seems to be a more complicated manifestation of the dichotomy 
          that occurs for weights $r \equiv 1 \mod (p-1)$ when $v(a_p) = \frac{1}{2}$ in \cite{BG13}.
\end{itemize}
\end{remark}

The proof  of Theorem~\ref{maintheoslopeone} uses the compatibility of the $p$-adic and mod
$p$ Local Langlands Correspondences, with respect to the process of reduction \cite{B10}. This compatibility allows one to 
reduce the reduction problem to one on the `automorphic side', namely, to computing
the reduction of a lattice in a certain Banach space. 
Let $G =  \mathrm{GL}_2(\Q_p)$ and let $B(V_{k,a_p})$ be the unitary $G$-Banach space associated
to $V_{k,a_p}$ by the $p$-adic Local Langlands Correspondence. 
The reduction $\overline{B(V_{k,a_p})}^{ss}$ of a lattice 
in this Banach space coincides with the image of $\bar{V}_{k,a_p}^{ss}$
under the (semisimple) mod $p$ Local Langlands Correspondence defined in \cite{Br03}. 
Since the mod $p$ correspondence is by definition injective, it suffices to compute the reduction 
$\overline{B(V_{k,a_p})}^{ss}$.

Let $K = \mathrm{GL}_2(\Z_p)$, a maximal compact open subgroup of $G =
\mathrm{GL}_2(\Q_p)$, and let $Z= \Q_p^\times$ be the center of $G$.  Let $X =
KZ \backslash G$ be the (vertices of the) Bruhat-Tits tree associated to
$G$.  The module $\mathrm{Sym}^r \bar\Q_p^2$, for $r = k-2$, carries a
natural action of $KZ$, and the projection

$$
KZ \backslash ( G \times   {\mathrm {Sym}}^{k-2} \bar\Q_p^2 ) \rightarrow KZ \backslash G = X
$$
defines a  local system on $X$. The smooth representation of $G$ 
\begin{eqnarray*}  
   {\mathrm{ind}}_{KZ}^{G} \> {\mathrm {Sym}}^{k-2} \bar\Q_p^2
\end{eqnarray*}
consists of all sections $f : G \rightarrow \mathrm{Sym}^{k-2}
\bar\Q_p^2$ of this local system which are compactly supported mod $KZ$.
There is a $G$-equivariant Hecke operator $T$ acting on this space of
sections. Let $\Pi_{k,a_p}$ be the locally algebraic representation of
$G$ defined by taking the cokernel of $T-a_p$ acting on the above space.
Let $\Theta_{k,a_p}$ be the image of the integral sections
${\mathrm{ind}}_{KZ}^{G} \> {\mathrm {Sym}}^{k-2} \bar\Z_p^2$ in
$\Pi_{k,a_p}$. 
Then $B(V_{k,a_p})$ is the completion $\hat{\Pi}_{k,a_p}$ of
$\Pi_{k,a_p}$ with respect to the lattice $\Theta_{k,a_p}$.
The completion $\hat{\Theta}_{k,a_p}$, and sometimes by
abuse of notation $\Theta_{k,a_p}$ itself, is called the standard lattice
in $B(V_{k,a_p})$.  We have $\overline{B(V_{k,a_p})}^{ss} \cong
\bar{\hat{\Theta}}_{k,a_p}^{ss} \cong \bar\Theta_{k,a_p}^{ss}$.

Thus, to compute $\bar{V}^{ss}_{k,a_p}$, it suffices to compute the reduction 
$\bar\Theta_{k,a_p}^{ss}$ of $\Theta_{k,a_p}$.
There are two main steps. First, we study a 
quotient $P$ of $\mathrm{Sym}^{k-2} \bar\F_p^2$ whose Jordan-H\"older factors  provide an upper bound on the possible Jordan-H\"older factors that can contribute to $\bar\Theta_{k,a_p}^{ss}$. This is fairly routine and is done in Section~\ref{sectionP}. It turns out that $P$ has $3$ (and sometimes 2) Jordan-H\"older factors.

Second, we  determine the eigenspaces of the Hecke operator $T$ through which each of 
these Jordan-H\"older factors contribute to $\bar\Theta_{k,a_p}^{ss}$, if at all. This involves
some delicate spectral analysis of the Hecke operator $T$ and is the hardest
part of the argument. 
The thrust of this analysis can be summarized as follows: we 
either `eliminate' all but one Jordan-H\"older factor of $P$, or eliminate all but two 
Jordan-H\"older factors of $P$ that `pair up' nicely under the mod $p$ Local Langlands Correspondence.
In the former case, $\bar\Theta_{k,a_p}$ is generically a supersingular representation and $\bar{V}_{k,a_p}$ 
is irreducible. In the latter case,
$\bar\Theta_{k,a_p}^{ss}$ is generically a direct sum of principal series representations and $\bar{V}_{k,a_p}^{ss}$ 
is reducible. There are additional complications. For instance, if the dimension 
of the sole surviving  Jordan-H\"older factor of $P$ in the former case is equal to $p-1$, it is also possible 
for $\bar\Theta_{k,a_p}^{ss}$ to be a direct sum of principal series, in which case 
$\bar{V}_{k,a_p}^{ss}$ is reducible with the inertia subgroup $I_p$ acting by scalars.

In order to carry out this spectral analysis, we must hunt for certain explicit rational sections $f$, 
respectively $f^\pm$,  of the cover above, and use the explicit formula for the Hecke 
operator $T$ to show that $(T-a_p)f$,
respectively $(T-a_p)f^\pm$, is integral with non-zero reduction of a relatively simple form.
In particular, we require that either the image of $\overline{(T-a_p)f}$ 
generates ${\mathrm{ind}}_{KZ}^{G} \> J$, for some Jordan-H\"older factor $J$ of $P$, 
or that $\overline{(T-a_p)f^\pm}$ generates $(T-\lambda^{\pm 1}) ({\mathrm{ind}}_{KZ}^{G} \> J^{\pm})$, for some $\lambda \in \bar\F_p^\times$ and some pair of Jordan-H\"older factors $J^\pm$ of $P$.
 In the former case, an easy argument shows that $J$ does not contribute to $\bar{\Theta}_{k,a_p}$, so is 
`eliminated', and in the latter case the factors $J^{\pm}$ `pair up' nicely.
 Note that neither the Jordan-Holder factors $J$, $J^\pm$ nor the `eigenvalue' $\lambda$ 
are specified in advance, which makes even starting this analysis rather difficult. 

The spectral analysis is carried out for the congruence classes $a \in \{1, 2, \ldots, p-1\} $ of $r$ mod $(p-1)$ in three batches. The generic case $3 \leq a= b \leq p-1$ is treated first in Section~\ref{case3ap-1}. 
The cases $a = 1$ (so $b = p$) and $a = 2$ (so $b = 2$) are more 
complicated and require special treatment (see
Sections~\ref{casea=1}, \ref{casea2}).
Generally speaking, the explicit functions $f$ we use are 
supported mod $KZ$ on a ball of radius at most 2 centered
at the origin of the tree.
But to establish the trichotomy in the case $b = 2$, we work with functions $f$ supported mod $KZ$ 
on a ball of radius $v(r-2) + 1$ centered at the origin, which may be arbitrarily large depending
on the valuation of $r-2$.  Moreover, the trichotomy we establish when $b = 2$ shows that each 
of the three possibilities for  $\bar{V}^{ss}_{k,a_p}|_{I_p}$ allowed by the structure of $P$ in this case and the mod $p$ Local Langlands Correspondence actually does occur. 

As already remarked, Theorem~\ref{maintheoslopeone} shows that $\bar{V}_{k,a_p}$ is often
reducible. This gives rise to a couple of subtler questions concerning the reduction, which do not seem to have been addressed for the crystalline representations $V_{k,a_p}$ treated so far in the literature, 
perhaps primarily because the reduction $\bar{V}_{k,a_p}$ in these cases
was found to be generically irreducible.  
We describe these two questions now.


First, if $\bar{V}_{k,a_p}^{ss}$ is isomorphic to $\omega \oplus 1$, up to a twist, then 
by a well-known result of Ribet, there is a lattice which is unique up to homothety 
inside $V_{k,a_p}$ that reduces to a non-split 
extension of $1$ by $\omega$, up to the same twist. One may  ask whether this extension 
is ``peu ramifi\'ee" or ``tr\`es ramifi\'ee" in the sense of Serre \cite{Serre87}.
The answer does not depend on the choice of the lattice or the choice of the basis.
In the context of slope $v(a_p) = 1$ and Theorem~\ref{maintheoslopeone}, this question 
arises in exactly two cases:
\begin{enumerate}
  \item $b=2$, $r\geq p+1$, $v\left(\dfrac{a_p}{p} -\dbinom{r}{2}\dfrac{p}{a_p}\right) = v(r-2)$ and 
  $\lambda = \overline{\dfrac{2}{2-r}\left(\dfrac{a_p}{p} -\dbinom{r}{2}\dfrac{p}{a_p}\right)} 
         = \pm \overline{1}$, 
so $$\overline{\dfrac{a_p}{p}} = \pm \overline{\dfrac{r}{2}} \quad \text{or} \quad \overline{\dfrac{a_p}{p}} = \pm \overline{(1-r)},$$
\item $b=p-1$, $r\geq 2p-2$, $p \nmid r- b$ and $\lambda = \overline{\dfrac{1}{r+1} \cdot \dfrac{a_p}{p}} = \pm \overline{1}$,  so
$$\overline{\dfrac{a_p}{p}} = \pm \overline{(r+1)}.$$
\end{enumerate}

Section~\ref{peu-tres} of this paper is primarily concerned with providing an answer to
this question. The following theorem summarizes our results. To state it, note that
the two roots $\overline{\frac{a_p}{p}}$ of the quadratic equation $\lambda = \pm \overline{1}$ in case (1) above are distinct if and only if $r \not \equiv 2/3 \mod p$.  We settle this question for most cases when $r \not \equiv 2/3 \mod p$ in case (1),  and for all $r$ in case (2), as follows.

\begin{theorem}
 \label{maintheopeuvstres}
Let $p>3$, let $k\geq p+3$ and let $r = k-2 \equiv b \mod (p-1)$, with $b = 2$
or $p-1$.  Suppose that $v(a_p) = 1$ and that $\bar{V}_{k,a_p}$, the
reduction of a lattice in $V_{k,a_p}$, is a non-split extension of $1$ by
$\omega$, up to a twist. 
  \begin{enumerate}
  \item If $b = 2$, $r \not\equiv 2/3 \mod p$, and 
\begin{eqnarray*}
  u := \dfrac{2}{2-r} \left( \dfrac{a_p}{p} -\binom{r}{2}\dfrac{p}{a_p} \right)
\> \text{is a $p$-adic unit with } \>
\overline{u} = \overline \eps, \text{ for } \eps \in \{\pm 1\}, 
\end{eqnarray*}
and 
  \begin{enumerate}
    \item $\overline{\dfrac{a_p}{p}} = \eps \overline{\dfrac{r}{2}}$, then
    $\bar{V}_{k,a_p}$ is ``peu ramifi\'ee'',
    \item $\overline{\dfrac{a_p}{p}} = \eps\overline{(1-r)}$, and if in addition, when $r\equiv 2\mod p$, 
     either \\ ${ } \qquad \Q_p(a_p)$ is an unramified extension of $\Q_p$ or $u-\eps$ is a uniformizer of
 $\Q_p(a_p)$, \\ 
then
 $\bar{V}_{k,a_p}$ is ``peu ramifi\'ee'' 
    if and only if $$v(u - \eps) < 1.$$
          Moreover, 
          as $a_p$ varies 
          with $v(u -\eps) \geq 1$,
   the reduction $\bar{V}_{k,a_p}$ varies through all ``tr\`es ramifi\'ee" extensions. 
 \end{enumerate}
  \item If $b = p-1$, $p \nmid r-b$ and $\overline{\dfrac{a_p}{p}} = \pm \overline{(r+1)}$, 
        then $\bar{V}_{k,a_p}$ is ``peu ramifi\'ee". 
 \end{enumerate}
\end{theorem}

\begin{remark} We have:
\begin{itemize}
  \item Theorem~\ref{maintheopeuvstres} distinguishes between the
 ``peu'' and ``tr\`es ramifi\'ee'' possibilities completely for primes $p \geq 5$, 
            when $r \not\equiv 2$ or $2/3 \mod p$.  The extra assumptions on $\Q_p(a_p)$ in part (1) (b), 
           when $r \equiv 2 \mod p$, 
could probably be removed. 

\item As far as we are aware, part (1) (b) of the theorem provides, for a fixed prime $p$, the first
explicit infinite family of crystalline representations 
with ``tr\`es ramifi\'ee" reduction.  

\item The inequalities for $u$ that occur in
part (1) (b) of the theorem are reminiscent of similar inequalities in the semistable
(non-crystalline) case \cite{Breuil-Mezard02}, separating out the ``peu''
and ``tr\`es ramifi\'ee'' cases in terms of the ${\mathcal L}$-invariant. 
\end{itemize}
\end{remark}

The proof of Theorem~\ref{maintheopeuvstres} relies on the work of Colmez \cite{Col}, 
where the Langlands correspondences are studied functorially. One difficulty that arises is that one cannot 
usually use the standard lattice above to study subtler properties of the reduction $\bar{V}_{k,a_p}$.
Colmez's functors establish correspondences between  all stable lattices on the
`Galois side' and certain stable lattices on the `automorphic side'. Moreover, these functors are compatible with 
the process of taking reduction, again allowing us to work on the automorphic side. 
The key is to find a suitable $\GL_2(\Q_p)$-stable lattice lying in the image of Colmez's functor 
whose reduction admits a particular composition series, which makes it possible to recognize whether
the corresponding mod $p$ Galois representation is ``peu'' or ``tr\`es ramifi\'ee''.

We now provide some more details about the proof, referring to the text for some  notation.
Let $\mathrm{St}$ be the mod $p$ Steinberg representation of $G$ and
let $1$ be the trivial mod $p$ representation of $G$.
The isomorphism classes of (non-split) extensions $E_\tau$ of $1$ by $\mathrm{St}$ 
are parameterized by 
classes $$[\tau] := (\tau(1+p) :  \tau(p)) \in \mathbb{P}^1(\bar\F_p),$$ 
of (non-zero) maps  $\tau \in \mathrm{Hom}_\mathrm{cont}(\Q_p^\times, \bar\F_p)$. 
For instance, if $E_\tau$ is the non-split extension of $1$ by ${\mathrm{St}}$ given 
by $\pi(0,1,1)$, the cokernel of $T-1$ acting on compactly
supported mod $p$ functions on the tree $X$, then $[\tau] = (0:1) \in  \mathbb{P}^1(\bar\F_p)$. 
Here $\pi(r,\lambda,\eta)$ is the basic mod $p$ representation of $G$, defined using
compact induction, for $0 \leq r \leq p-1$, $\lambda \in \bar\F_p$
and $\eta$ a smooth mod $p$ character of $\Q_p^\times$ (cf. Section~\ref{subsectionLLC}).

Let $\Pi_\tau$  be the unique non-split extension of the principal series representation 
$\pi(p-3,1,\omega)$ by $E_\tau$, for $\tau \neq 0$, defined in \cite{Col}.  If $V$ is Colmez's functor from
the `automorphic side' to the `Galois side', then it is known that $V(\Pi_\tau \otimes \omega^m)$ is a non-split extension of $\omega^m$ by $\omega^{m+1}$, for $m \in \Z$. 
Moreover, the extension $V(\Pi_\tau \otimes \omega^m)$ is known to be ``peu ramifi\'ee'' if and only if 
$E_\tau \cong \pi(0,1,1)$, that is, $[\tau] = (0:1)$. 
Thus if one can construct a possibly non-standard complete stable lattice $\Theta'$ in $B(V_{k,a_p})$ 
with $\bar{\Theta}' \cong \Pi_\tau \otimes \omega^m$, in and of itself not easy to do,
to check whether $V(\bar{\Theta}')$ is ``peu ramifi\'ee'', one `only' needs to compute $[\tau]$. 
To this end, there is a
linear form $\mu : \mathrm{St} \otimes \omega^m \rightarrow \bar\F_p(\omega^m)$
which has the property that if $e \in (E_\tau \otimes \omega^m) \setminus (\mathrm{St} \otimes \omega^m$), then 
$$[\tau] = \left( \mu \left( \left( \begin{smallmatrix} 1+p & 0 \\ 0 & 1 \end{smallmatrix} \right) \cdot e - e \right) : 
                            \mu \left( \left( \begin{smallmatrix} p & 0 \\ 0 & 1 \end{smallmatrix} \right) \cdot e - e \right) \right) \in \P^1(\bar\F_p). $$

For instance, when $b = 2$ and so $m = 1$, under the conditions of  part (1) (b), the most difficult part of Theorem~\ref{maintheopeuvstres} to prove, if  we take $\eps = 1$ for simplicity, then we construct  a non-standard lattice $\Theta'$ such that $\bar{\Theta}' \cong \Pi_\tau \otimes \omega$  
and show that
\begin{itemize}
  \item $v(u - \eps) < 1 \implies [\tau] = (0:1) \in   \P^1(\bar\F_p)$,  and,  
  \item $v(u - \eps) \geq 1 \implies [\tau] = (\overline{3r-2}: y) \in \P^1(\bar\F_p)$, for some (all) $y \in \bar\F_p$,
\end{itemize}
proving this part of the theorem.

We now briefly turn to the second question alluded to above.
If $\bar{V}_{k,a_p}^{ss}$ is the trivial representation up to a twist, then again by Ribet, there is a lattice inside
$V_{k,a_p}$ that reduces to a non-split extension of the trivial representation $1$ by $1$, up to the same twist.
One may ask whether this reduction is unramified or ramified. 
In the context of Theorem~\ref{maintheoslopeone}, this question arises exactly when
\begin{enumerate}
  \item [(3)] $b=p$, $r \geq 3p-2$, $p \mid r- b$ and $\lambda =  \pm \overline{1}$,  so
$$\overline{\dfrac{a_p}{p}} =  \lambda \pm \overline{\sqrt{r/p}}.$$
\end{enumerate}
This second question
seems rather difficult to answer in general because of the large number of possible non-homothetic 
lattices involved. However, as an example, at the end of Section~\ref{peu-tres},
we show that the reduction corresponding to the standard lattice
is (after twisting) non-split and unramified (see Theorem~\ref{unramified} in Section~\ref{sectiontrivialnonsplit}).

We end this paper in Section~\ref{sectionexamples} by giving  some examples to illustrate Theorems~\ref{maintheoslopeone} and \ref{maintheopeuvstres}.
In particular, we compare our results 
with the local restrictions of some known reductions of global $p$-adic Galois representations 
attached to modular forms of level one and small weight computed in \cite{Serre73}. 
We show that  our results
match and also provide some new local information about the reductions in these special cases.

\section{Basics}\label{basics}

In this section, we recall some notation and well-known facts.
Further details can be found in \cite{Br03} and \cite{Bhattacharya-Ghate}.

\subsection{Hecke operator $T$ }
  \label{Hecke}

Let $G = \mathrm{GL}_2(\Q_p)$, $K = \mathrm{GL}_2(\Z_p)$ be the standard
maximal compact subgroup of $G$ and $Z = \Q_p^\times$ be the center of $G$.
Let $R$ be a $\Z_p$-algebra and let $V = \Sym^r R^2\otimes D^s$ be the
usual symmetric power representation of $KZ$ twisted by a power of the
determinant character $D$, modeled on homogeneous polynomials of degree
$r$ in the variables $X$, $Y$ over $R$. We will denote $\mathrm{ind}_{KZ}^{G}$ 
to mean compact induction. Thus $\mathrm{ind}_{KZ}^{G} V$ consists of functions
$f : G \rightarrow V$ such that $f(hg) = h \cdot f(g)$, for all $h \in KZ$
and $g \in G$, and $f$ is compactly supported mod $KZ$.   For $g \in G$, $v \in V$, let
$[g,v] \in \mathrm{ind}_{KZ}^{G} V$ be the function with support in
${KZ}g^{-1}$ given by 
  $$g' \mapsto
     \begin{cases}
         g'g \cdot v,  \ & \text{ if } g' \in {KZ}g^{-1} \\
         0,                  & \text{ otherwise.}
      \end{cases}$$ 
Any function in $\mathrm{ind}_{KZ}^G V$ is a finite linear combination of functions of the form $[g,v]$, for $g\in G$ and $v\in V$.  
The Hecke operator $T$ is defined by its action on these elementary functions via

\begin{equation}\label{T} T([g,v(X,Y)])=\underset{\lambda\in\F_p}\sum\left[g\left(\begin{smallmatrix} p & [\lambda]\\
                                                 0 & 1
                                                \end{smallmatrix}\right),\:v\left(X, -[\lambda]X+pY\right)\right]+\left[g\left(\begin{smallmatrix} 1 & 0\\
                                                                                                                                                   0 & p
                                                \end{smallmatrix}\right),\:v(pX,Y)\right],\end{equation}
 where $[\lambda]$ denotes the Teichm\"uller representative of $\lambda\in\F_p$.                                               

For $m = 0$, set $I_0 = \{0\}$, and 
for $m >0$, let
$I_m = \{ [\lambda_0] + [\lambda_1] p + \cdots + [\lambda_{m-1}]p^{m-1}  \> : \>  \lambda_i \in \F_p \} 
              \subset \Z_p$,
where the square brackets denote Teichm\"uller representatives. For $m \geq 1$, 
there is a truncation map
$[\quad]_{m-1}: I_{m} \rightarrow I_{m-1}$ given by taking the first $m-1$ terms in the $p$-adic expansion above;
for $m = 1$, $[\quad]_{m-1}$ is the $0$-map.
Let $\alpha =  \left( \begin{smallmatrix} 1 & 0 \\ 0  & p \end{smallmatrix} \right)$. 
For $m \geq 0$ and $\lambda \in I_m$, let
\begin{eqnarray*}
  g^0_{m, \lambda} =  \left( \begin{smallmatrix} p^m & \lambda \\ 0 & 1 \end{smallmatrix} \right) & \quad \text{and} \quad 
  g^1_{m, \lambda} = \left( \begin{smallmatrix} 1 & 0  \\ p \lambda  & p^{m+1} \end{smallmatrix} \right),
\end{eqnarray*}
noting that $g^0_{0,0}=\mathrm{Id}$ is the identity matrix and $g_{0,0}^1=\alpha$ in $G$. 
Recall the decomposition
\begin{eqnarray*}
    G & = & \coprod_{\substack{m\geq 0,\,\lambda \in I_m,\\ i\in\{0,1\}}} {KZ} (g^i_{m, \lambda})^{-1}.
\end{eqnarray*}
Thus a general element in $\mathrm{ind}_{KZ}^G V$ is a finite sum of functions of the form $[g,v]$, 
with $g=g_{m,\lambda}^0$ or $\,g_{m,\lambda}^1$, for some $\lambda\in I_m$ and $v\in V$.
For a $\Z_p$-algebra $R$, let $v = \sum_{i=0}^r c_i X^{r-i} Y^i \in V = \mathrm{Sym}^r R^2\otimes D^s$.
Expanding the formula \eqref{T} for the Hecke operator $T$ one may write
$T  = T^+ + T^-$, with 
\begin{eqnarray*}
  T^+([g^0_{n,\mu},v]) & = & \sum_{\lambda \in I_1} \left[ g^0_{n+1, \mu +p^n\lambda},    
         \sum_{j=0}^r \left( p^j \sum_{i=j}^r c_i \binom{i}{j}(-\lambda)^{i-j} \right) X^{r-j} Y^j \right], \\
  T^-([g^0_{n,\mu},v]) & = & \left[ g^0_{n-1, [\mu]_{n-1}},    
         \sum_{j=0}^r \left( \sum_{i=j}^r p^{r-i} c_i {i \choose j} 
         \left( \frac{\mu - [\mu]_{n-1}}{p^{n-1}} \right)^{i-j} \right) X^{r-j} Y^j \right] \quad (n > 0), \\
  T^-([g^0_{n,\mu},v]) & = &  [ \alpha,  \sum_{j=0}^r  p^{r-j}  c_j  X^{r-j} Y^j ] \quad  (n=0). 
\end{eqnarray*}
These explicit formulas for $T^+$ and $T^-$ will be used to compute $(T-a_p)f$, for $f \in \mathrm{ind}_{KZ}^G \Sym^r\bar\Q_p^2$.

\subsection{The mod $p$ Local Langlands Correspondence}
\label{subsectionLLC}

For $0 \leq r \leq p-1$, $\lambda \in \bar{\F}_p$ and $\eta : \Q_p^\times
\rightarrow \bar\F_p^\times$ a smooth character, let
\begin{eqnarray*}
  \pi(r, \lambda, \eta) & := & \frac{\mathrm{ind}_{KZ}^{G} \:\Sym^r\bar\F_p^2}{T-\lambda} \otimes (\eta\circ \mathrm{det})
\end{eqnarray*}
be the smooth admissible representation of ${G}$, 
known to be irreducible unless $(r,\lambda)=(0,\pm 1)$ or $(p-1,\pm 1)$, by the classification of irreducible  representations  of $G$ in characteristic $p$ in \cite{BL94, BL95, Breuil03a}.
With this notation, Breuil's semisimple mod $p$ Local Langlands Correspondence 
\cite[Def. 1.1]{Br03} 
is given by:
\begin{itemize} 
  \item  $\lambda = 0$: \quad
             $\mathrm{ind}(\omega_2^{r+1}) \otimes \eta \:\:\overset{LL}\longmapsto\:\: \pi(r,0,\eta)$,
  \item $\lambda \neq 0$: \quad
             $\left( \mu_\lambda\cdot \omega^{r+1}  \oplus \mu_{\lambda^{-1}} \right) \otimes \eta 
                   \:\:\overset{LL}\longmapsto\:\:  \pi(r, \lambda, \eta)^{ss} \oplus  \pi([p-3-r], \lambda^{-1}, \eta \omega^{r+1})^{ss}$, 
\end{itemize}
where $\{0,1, \ldots, p-2 \} \ni [p-3-r] \equiv p-3-r \mod (p-1)$.

Consider the locally algebraic representation of $G$ given by 
\begin{eqnarray*} 
\Pi_{k, a_p} = \frac{ \mathrm{ind}_{{KZ}}^{G}
\Sym^r \bar\Q_p^2 }{T-a_p}, 
\end{eqnarray*} 
where $r=k-2 \geq 0$ and $T$ is the
Hecke operator. Consider the standard lattice in $\Pi_{k,a_p}$ given by
\begin{equation} 
\label{definetheta}  
\Theta=\Theta_{k, a_p} :=
\mathrm{image} \left( \mathrm{ind}_{KZ}^{G} \Sym^r \bar\Z_p^2 \rightarrow
\Pi_{k, a_p} \right) \iso \frac{ \mathrm{ind}_{KZ}^{G} \Sym^r \bar\Z_p^2
}{(T-a_p)(\mathrm{ind}_{KZ}^{G} \Sym^r \bar\Q_p^2) \cap
\mathrm{ind}_{KZ}^{G} \Sym^r \bar\Z_p^2 }.  
\end{equation} 
It is known
that the semisimplification of the reduction of this lattice satisfies
$\bar\Theta_{k,a_p}^\mathrm{ss} \iso LL(\bar{V}_{k,a_p}^{ss})$, where $LL$ is
the (semisimple) mod $p$ Local Langlands Correspondence above \cite{B10}.
Since the map ${LL}$ is clearly injective, it is enough to know ${LL}(\bar V_{k,a_p}^{ss})$ to determine 
$\bar{V}_{k,a_p}^{ss}$.

\subsection{Useful lemmas}

 We recall some combinatorial results from \cite{Bhattacharya-Ghate}. Lemma \ref{comb6} is not stated there 
  but we skip the proof here, since it is similar.
  
\begin{lemma}\label{comb1}
     For $r\equiv a\mod (p-1)$, with $1\leq a\leq p-1$, we have $$S_r:=\sum_{\substack{0\,<j\,<\,r,\\ j\,\equiv\, a\mod(p-1)}}\binom{r}{j}\equiv 0\mod p.$$
     Moreover, we have $\frac{1}{p}S_r \equiv \frac{a-r}{a}\mod p$,  for $p>2$.
   \end{lemma}

\begin{lemma}
\label{comb2}
 Let $2p\leq r\equiv a\mod(p-1)$, with $2\leq a\leq p-1$. Then one can choose integers $\alpha_j\in\mathbb{Z}$,
 for all $j$, with $0<j<r$ and $j\equiv a\mod (p-1)$, such that 
 the following properties hold:
  \begin{enumerate}
    \item For all $j$ as above, $\binom{r}{j}\equiv \alpha_j \mod p$,
    \item $\underset{j\geq n}\sum\binom{j}{n} \alpha_j\equiv 0\mod p^{3-n}$, for $n=0$, $1$,
    \item  $\underset{j\geq 2}\sum\binom{j}{2}\,\alpha_j \equiv\begin{cases}
                                                                 0\mod p, &\text{if } 3\leq a\leq p-1\\
                                                                 \binom{r}{2}\mod p, &\text{if } a=2.
                                                             \end{cases}$
  \end{enumerate}
\end{lemma}

\begin{lemma}\label{comb3}
 We have:
  \begin{enumerate}
      \item[(i)] If $r\equiv b\mod (p-1)$, with $2\leq b\leq p$, then 
               $$T_r:=\sum_{\substack{0<\,j\,<\,r-1,\\ j\equiv\, b-1\mod (p-1)}} \binom{r}{j}\equiv b-r\mod p.$$
      \item[(ii)] 
                 If $r\equiv b\mod p(p-1)$, with $3\leq b\leq p$, then one can choose  integers $\beta_j\,$,  for all 
                  $j\equiv b-1\mod (p-1)$,  with  $b-1\leq j< r-1$, satisfying the following properties:
                  \begin{enumerate}
                     \item[(1)] $\beta_j\equiv\binom{r}{j} \mod p$, for all $j$ as above,
                     \item[(2)] $\underset{j\geq n}\sum \,\binom{j}{n}\beta_j\equiv 0\mod p^{3-n}$, for $n=0$, $1$, $2$.
                   \end{enumerate}
   \end{enumerate}
\end{lemma}

\begin{lemma}
 \label{comb6}
 Let $p\geq 3$, $2p\leq r\equiv 1\mod (p-1)$ and let $p$ divide $r$.  Then one can choose integers $\alpha_j\in\mathbb{Z}$, for all $j$ 
 with $1<j<r$ and $j\equiv 1\mod (p-1)$, such that 
 the following properties hold:
  \begin{enumerate}
   \item For all $j$ as above, $\binom{r}{j}\equiv \alpha_j \mod p$,
   \item $\underset{j\geq n}\sum \binom{j}{n}\alpha_j\equiv 0\mod p^{3-n}$, for $n=0,1,2$.
  \end{enumerate}
\end{lemma}


\subsection{Further useful results}
We now state some identities involving sums of products of 
binomial coefficients, which will be useful in treating the delicate case $r \equiv 2 \mod (p-1)$.

\begin{lemma}
\label{congruences1}
Let $p \geq 3$. If $s \equiv 1 \mod (p-1)$ and  $t = v(s-1)$,
then $p^i \binom{s}{i} \equiv 0 \mod {p^{t+2}}$, for all $i \geq 2$.
\end{lemma}

\begin{lemma}
\label{congruences2}
Let $p > 3$. If $r \equiv 2 \mod (p-1)$ and $t = v(r-2)$, then
$p^i\binom{r}{i} \equiv 0 \mod {p^{t+3}}$, for all $i \geq 3$.
\end{lemma}

Lemma~\ref{congruences1} is \cite[ Lemma 2.1]{BG13}, and the proof
of Lemma~\ref{congruences2} is very similar, so is skipped here.


\begin{lemma}
\label{ppower}
Let $p > 2$. If $P(X) = 1+pX \in \Z_p[X]$, then $P(X)^{p^t} \equiv 1 + p^{t+1}X \mod
{p^{t+2}}$, for all $t \geq 0$.
\end{lemma}

\begin{proof}
We induct on $t$. The case $t = 0$ is clear. Write
$P(X)^{p^t} = 1+p^{t+1}X + p^{t+2}Q(X)$, for $Q(X) \in \Z_p[X]$.
Then 
$P(X)^{p^{t+1}} = (1+p^{t+1}X + p^{t+2}Q(X))^p
= \sum_{i \geq 0}\binom{p}{i}p^{i(t+1)}(X+pQ(X))^i$.
We have $v_p(\binom{p}{i}p^{i(t+1)}) \geq t+3$, for all $i \geq 2$. Indeed,
for $i<p$, we have $v_p(\binom{p}{i}p^{i(t+1)}) = 1+i(t+1) = ti+i+1$, and
for $i=p$, we have $v_p(\binom{p}{i}p^{i(t+1)}) = p(t+1) \geq t+3$, as $p\neq 2$.
So 
$P(X)^{p^{t+1}} \equiv 1 + p^{t+2}X \mod {p^{t+3}}$.
\end{proof}

The first part of the following proposition generalizes Lemma~\ref{comb1} when $a = 2$ and 
$p > 3$.

\begin{prop}
\label{congrbinom2}
Let $p > 3$ and write $r = 2+n(p-1)p^t$, with $t \geq 0$ and $n > 0$. 
Then, the following identities hold:
\begin{enumerate}
\item
\begin{eqnarray*}
\sum_{{0<j<r \atop j \equiv 2 \mod (p-1)}}\binom{r}{j} 
& \equiv  & \dfrac{p(2-r)}{2} \mod {p^{t+2}}. 
\end{eqnarray*}
\item
\begin{eqnarray*}
\sum_{{0<j<r \atop j \equiv 2 \mod (p-1)}}j\binom{r}{j} 
& \equiv & 
\dfrac{pr(2-r)}{1-p} \mod {p^{t+2}}.
\end{eqnarray*}
\item
\begin{eqnarray*}
\sum_{{0<j<r \atop j \equiv 2 \mod (p-1)}}\binom{j}{2}\binom{r}{j} 
& \equiv & \dfrac{\binom{r}{2}}{1-p}
\mod {p^{t+1}}.
\end{eqnarray*}
\item For all $i \geq 3$,
\begin{eqnarray*}
\sum_{{0<j<r \atop j \equiv 2 \mod (p-1)}} \binom{j}{i}\binom{r}{j} 
\equiv 0 \mod {p^{t+3-i}}.
\end{eqnarray*}
\end{enumerate}
\end{prop}

\begin{proof}
Let $0 \leq i < r$. Let $S_{i,r} = 
\sum_{{0<j<r \atop j \equiv 2 \mod (p-1)}}\binom{j}{i}\binom{r}{j}$.
Let $\Sigma_{i,r} = (p-1)
\sum_{{j \geq 0 \atop j \equiv 2 \mod (p-1)}}\binom{j}{i}\binom{r}{j}$,
so $\Sigma_{i,r} = (p-1)(S_{i,r} + \binom{r}{i})$.
Let $f_r(x) = (1+x)^r = \sum_{\ell\geq 0}\binom{r}{\ell}x^\ell$, 
which we consider as a function $\Z_p \to \Z_p$. 
Let $g_{i,r}(x) = \frac{x^{i-2}}{i!}f^{(i)}_r(x)$,
so that
$g_{i,r}(x) = \binom{r}{i}x^{i-2}(1+x)^{r-i}
= \sum_{\ell \geq 0}\binom{r}{\ell}\binom{\ell}{i}x^{\ell-2}$.

Let $\mu' = \mu_{p-1}\setminus \{-1\}$ be the $(p-1)$-st roots of unity sans $-1$.
Then $\Sigma_{i,r} = \sum_{\xi\in \mu_{p-1}}g_{i,r}(\xi)= \sum_{\xi\in
\mu'}g_{i,r}(\xi)$, as $r > i$.
Let $U_{i,r} = \sum_{\xi\in\mu'}\xi^{i-2}(1+\xi)^{r-i}$. Then 
$\Sigma_{i,r} = \binom{r}{i}U_{i,r}$, so that finally
$(p-1)(S_{i,r}+\binom{r}{i}) = \binom{r}{i}U_{i,r}$.

When $i \geq 3$, observe that $U_{i,r} \in \Z_p$, so 
part (4) of the proposition follows from Lemma~\ref{congruences2}.

When $i=2$, we have  $U_{2,r} = \sum_{\xi\in\mu'}(1+\xi)^{n(p-1)p^t}$. 
If $\xi\in \mu'$, then $1+\xi \in
\Z_p^\times$, so $(1+\xi)^{p-1} = 1+pz_{\xi}$, for some $z_{\xi}\in \Z_p$. 
By Lemma~\ref{ppower}, $(1+pz_{\xi})^{p^t} \equiv 1 \mod {p^{t+1}}$, so
$(1+\xi)^{r-2} \equiv 1 \mod {p^{t+1}}$.
So $U_{2,r} \equiv p-2 \mod {p^{t+1}}$ and $S_{2,r} \equiv \binom{r}{2}/(1-p) \mod
{p^{t+1}}$, proving part (3).

When $i = 1$, 
write again $(1+\xi)^{p-1} = 1+pz_\xi$. Then $(1+\xi)^{(p-1)p^t} \equiv  
1+p^{t+1}z_\xi \mod {p^{t+2}}$, by Lemma \ref{ppower}.
So we get 
$$
U_{1,r} \equiv  \sum_{\xi\in\mu'} \xi^{-1}(1+\xi) + np^{t+1}
\sum_{\xi\in\mu'}\xi^{-1}(1+\xi)z_{\xi} \mod {p^{t+2}}.
$$
By computing $U_{1,2}$ and $U_{1,p+1}$ we get
that $\sum_{\xi\in \mu'}\xi^{-1}(1+\xi) = p-1$ and 
$\sum_{\xi\in\mu'}\xi^{-1}(1+\xi)z_{\xi} \equiv -1 \mod p$, which gives
$U_{1,r} \equiv p-1-np^{t+1} \mod {p^{t+2}}$, so
$S_{1,r} \equiv nrp^{t+1} \mod {p^{t+2}}$, proving part (2).

Let $S_r =  S_{0,r}$ 
and $\Sigma_r = \Sigma_{0,r}$,
so that $\Sigma_r = (p-1)(S_r + 1)$.
Note $\Sigma_r = \sum_{\xi\in \mu_{p-1}}f_r(\xi)= \sum_{\xi\in \mu'}f_r(\xi)$, as $r > 0$.
We have
$(1+\xi)^{2+n(p-1)p^t} \equiv (1+\xi)^2(1+np^{t+1}z_{\xi}) \mod {p^{t+2}}$, so that
$$
\Sigma_r \equiv \sum_{\xi\in\mu'}(1+\xi)^2 + 
np^{t+1}\sum_{\xi\in\mu'}(1+\xi)^2z_{\xi} \mod {p^{t+2}}.
$$
Let $s = \sum_{\xi\in\mu'}(1+\xi)^2$ and let
$s'= \sum_{\xi\in\mu'}(1+\xi)^2z_{\xi}$. 
Then $s = \Sigma_2 = p-1$. On the other hand, 
$s + ps' \equiv \Sigma_{p+1} \mod {p^2}$. We also have that
$\Sigma_{p+1} \equiv (p-1)(1+\binom{p+1}{p-1}) \equiv (p-1)(1+p/2) \mod {p^2}$.
So $s' \equiv -1/2 \mod p$.
Putting everything together, we get 
$\Sigma_r \equiv (p-1)-p^{t+1}n/2 \mod {p^{t+2}}$
and so
$S_r \equiv p^{t+1}n/2 \mod {p^{t+2}}$, proving part (1) as well.
\end{proof}

We now state two more propositions of a similar nature.

\begin{prop}
\label{congrbinom1}
Let $p > 3$ and write $s = 1+n(p-1)p^t$, with $t \geq 0$ and $n > 0$. Then, we have:
\begin{enumerate}
\item
\begin{eqnarray*}
\sum_{j \equiv 1 \mod (p-1)} \binom{s}{j} 
& \equiv & 1+np^{t+1} \mod {p^{t+2}}.
\end{eqnarray*}
\item
\begin{eqnarray*}
\sum_{j \equiv 1 \mod (p-1)} j\binom{s}{j} 
& \equiv & \dfrac{s(p-2)}{p-1} - snp^{t+1} \mod {p^{t+2}}.
\end{eqnarray*}
\item For all $i \geq 2$,
\begin{eqnarray*}
\sum_{j \equiv 1 \mod (p-1)} \binom{j}{i}\binom{s}{j} 
& \equiv & 0 \mod {p^{t+2-i}}.
\end{eqnarray*}
\end{enumerate}
\end{prop}

\begin{prop}
\label{congrbinom12}
Let $p > 3$ and write $r = 2+n(p-1)p^t$, with $t \geq 0$ and  $n > 0$. 
Then, we have:
\begin{enumerate}
\item
\begin{eqnarray*}
\sum_{{1<j \leq r-1 \atop j \equiv 1 \mod (p-1)}}\binom{r}{j} 
& \equiv & 2-r + 2p^{t+1}n \mod {p^{t+2}}.
\end{eqnarray*}
\item
\begin{eqnarray*}
\sum_{1<j \leq r-1 \atop j \equiv 1 \mod (p-1)} j\binom{r}{j} 
& \equiv  & nrp^{t+1} \mod {p^{t+2}}.
\end{eqnarray*}
\item
\begin{eqnarray*}
\sum_{1<j  \leq r-1 \atop j \equiv 1 \mod (p-1)}\binom{j}{2}\binom{r}{j} 
& \equiv  & \dfrac{\binom{r}{2}}{p-1} \mod {p^{t+1}}.
\end{eqnarray*}
\item For all $i \geq 3$,
\begin{eqnarray*}
\sum_{1<j \leq r-1 \atop  j \equiv 1 \mod (p-1)}\binom{j}{i}\binom{r}{j} 
& \equiv &  0 \mod {p^{t+3-i}}.
\end{eqnarray*}
\end{enumerate}
\end{prop}

\noindent We do not give the proofs of Proposition \ref{congrbinom1} and
\ref{congrbinom12} as they are very similar to the proof of Proposition
\ref{congrbinom2}. The proof of Proposition~\ref{congrbinom1} makes use of
Lemma~\ref{congruences1} instead of Lemma \ref{congruences2}.


\section{A quotient of $V_r = \Sym^r\bar{\mathbb{F}}_p^2$}
\label{sectionP}

\subsection{Definition of $P$}
Let $V_r$ denote the $(r+1)$-dimensional $\bar{\mathbb{F}}_p$-vector space of 
homogeneous polynomials in two variables $X$ and $Y$ of degree $r$ over
$\bar{\mathbb{F}}_p$. 
The group $\Gamma=\mathrm{GL}_2(\mathbb{F}_p)$
acts on $V_r$ by the formula 
$\left(\begin{smallmatrix} a & b\\
c & d
\end{smallmatrix}\right)\cdot F(X,Y)=F(aX+cY, bX+dY)$.

Let $X_r \subset V_r$ denote the
$\bar{\mathbb{F}}_p[\Gamma]$-span of the monomial $X^r$.
Let $\theta(X,Y)=X^pY-XY^p$.
The action of $\Gamma$ on $\theta$ is via the determinant.
We define two important submodules of $V_r$ as follows.
$$
V_r^*:=\{F\in V_r:\theta \mid F\}\cong \begin{cases}
                                          0, &\text{if } r<p+1\\
                                          V_{r-p-1}\otimes D, &\text{if } r\geq p+1, 
                                         \end{cases}
$$ 
$$
V_r^{**}:=\{F \in V_r:\theta^2 \mid F\}\cong\begin{cases} 0, &\text{if
} r<2p+2\\ V_{r-2p-2}\otimes D^2, &\text{if } r\geq 2p+2.  \end{cases}
$$
We set $X_r^*:=X_r\cap V_r^*$ and
$X_r^{**}=X_r\cap V_r^{**}$.

The mod $p$ reduction $\overline{\Theta}_{k,a_p}$ of the lattice
$\Theta_{k,a_p}$ is a quotient of $\mathrm{ind}_{KZ}^G V_r$, for $r=k-2$. By 
\cite[Rem. 4.4]{BG09} we know that if $v(a_p)=1$ and $r \geq p$, then  the map $\mathrm{ind}_{KZ}^G V_r \twoheadrightarrow\overline{\Theta}_{k,a_p}$
factors through $\mathrm{ind}_{KZ}^G P$, where $$P:=\frac{V_r} {X_r + V_r^{**}}.$$
We now study the $\Gamma$-module structure of the quotient $P$ of $V_r$.
In particular, we will show that $P$ has $3$ (and occasionally $2$) Jordan-H\"older (JH) factors.

\subsection{A filtration on $P$}
\label{filtrationP}
In order to make use of results in Glover \cite{G78}, let us abuse notation a bit and  model $V_r$ on  the space of
homogeneous polynomials  in two variables $X$ and $Y$ of degree $r$ with coefficients in $\F_p$.
Once we  establish
the structure of  $P$ over $\F_p$, it will immediately imply the corresponding result over $\bar\F_p$, by extension 
of scalars.

We work with the representatives $1 \leq a \leq p-1$ of the congruence classes of $r \mod (p-1)$.
Note that $a = 1$ corresponds to $b = p$ from the Introduction, whereas the other values of $a$ and $b$ coincide.

\begin{prop}\label{es1}
 Let $p\geq 3$, and $r\equiv a\mod (p-1)$, with $1\leq a\leq
p-1$. 
  \begin{enumerate}
       \item[(i)] For $r\geq p$, the $\Gamma$-module structure of $V_r/V_r^*$ is given by
               \begin{eqnarray*}
                 0\rightarrow V_{a}\rightarrow\frac{V_r}{V_r^*}\rightarrow V_{p-a-1}\otimes D^{a}\rightarrow 0,
                \end{eqnarray*} 
             and the  sequence splits  if and only if $a=p-1$.
       \item[(ii)] For $r\geq 2p+1$, the $\Gamma$-module structure of $V_r^*/V_r^{**}$ is given by
              \begin{eqnarray*}
                0\rightarrow V_{p-2}\otimes D\rightarrow\frac{V_r^*}{V_r^{**}}\rightarrow V_{1}\rightarrow 0, & & \text{ if } a=1,\\
                0\rightarrow V_{p-1}\otimes D\rightarrow\frac{V_r^*}{V_r^{**}}\rightarrow V_{0}\otimes D\rightarrow 0, & &\text{ if } a=2,\\
                0\rightarrow V_{a-2}\otimes D\rightarrow\frac{V_r^*}{V_r^{**}}\rightarrow V_{p-a+1}\otimes D^{a-1}\rightarrow 0, & & \text{ if } 3\leq a\leq p-1,
               \end{eqnarray*} 
              and the sequences above split  if and only if $a=2$.
   \end{enumerate} 
\end{prop}

\begin{proof}
  See \cite[Prop. 2.1, Prop. 2.2]{Bhattacharya-Ghate}.
\end{proof}

\begin{lemma}
\label{Xr*/Xr**}
  Let $p\geq 3$, $r\geq 2p$ and $r\equiv a\mod (p-1)$, with $1\leq a\leq p-1$. Then 
  $$X_r^*/X_r^{**}=\begin{cases}
                    V_{p-2}\otimes D, &\text{if } a=1 \text { and } p\nmid r,\\
                    0, &\text{otherwise}. 
                   \end{cases}$$
\end{lemma}

\begin{proof}
This follows from \cite[Lem. 3.1, Lem. 4.7]{Bhattacharya-Ghate}.
\end{proof}

There is a filtration on $V_r / V_r^{**}$ given by
$$
0 \subset 
\dfrac{ \langle\theta X^{r-p-1}\rangle + V_r^{**}}{  V_r^{**}} \subset
\dfrac{V_r^*} { V_r^{**}} \subset 
\dfrac{X_r + V_r^{*} }{ V_r^{**}} \subset
\dfrac{V_r}{  V_r^{**}}\, ,
$$
where the graded parts of this filtration are non-zero irreducible $\Gamma$-modules.
Looking at the image of this filtration inside $P$, we get a
filtration:
$$ 
0 \subset 
W_ 0 := \dfrac{\langle\theta X^{r-p-1}\rangle + X_r + V_r^{**}}{ V_r^{**} + X_r} \subset
W_1 := \dfrac{V_r^* + X_r}{ V_r^{**} + X_r} \subset
W_ 2 := \dfrac{V_r}{ V_r^{**} + X_r}=P.
$$
Each of the graded pieces of this filtration is either irreducible or zero.
We set $J_i := W_i / W_{i-1}$, for $i = 0$, $1$, $2$, with $W_{-1} = 0$.
Note that $W_1\cong \dfrac{V_r^*/V_r^{**}}{X_r^*/X_r^{**}}$ is never zero
by Proposition \ref{es1} (ii) and  Lemma \ref{Xr*/Xr**}. In fact, $W_1$
has two JH factors $J_0$ and $J_1$ unless  $a=1$ and $p\nmid r$, in which
case $W_0 = J_0 = 0$ and $W_1 = J_1$ is irreducible. Also, $J_2\cong
\dfrac{V_r}{V_r^*+X_r}\cong \dfrac{V_r/V_r^*}{X_r/X_r^*}$ is a proper
quotient of $V_r/V_r^*$.  Using Proposition \ref{es1}, and
\cite[(4.5)]{G78}, we obtain:

\begin{prop} 
\label{P}
 Let $p\geq 3$, $r\geq 2p$ and $ r\equiv a\mod (p-1)$, with $1\leq a\leq p-1$. 
 Then the  structure of $P$ is given by the following short exact sequences of $\,\Gamma$-modules:
 \begin{enumerate}
   \item[(i)] If $a=1$ and $p\nmid r$, then
   $$ 0\rightarrow V_1\rightarrow P\rightarrow V_{p-2}\otimes D\rightarrow 0,$$
	that is, $J_0 = 0$, $J_1 = V_1$ and $J_2 = V_{p-2}\otimes D$.
   \item[(ii)] If $a=1$ and $p\mid r$, or if $\,2\leq a\leq p-1$, then
   $$ 0\rightarrow W_1\cong V_r^*/V_r^{**}\rightarrow P\rightarrow V_{p-a-1}\otimes D^a\rightarrow 0,$$
	that is,
	\begin{enumerate}
	\item
	$J_0 = V_{p-2}\otimes D$, $J_1 = V_1$ and $J_2 = V_{p-2} \otimes D$
         if $a = 1$ and $p \> | \>r$, 
	\item
	%
	$J_0 = V_{p-1} \otimes D$, $J_1 = V_0\otimes D$ and $J_2 = V_{p-3} \otimes D^2$ if $a = 2$, $r>2p$ 
    ($J_1 = 0$ if $r=2p$),
	\item	
	$J_0 = V_{a-2}\otimes D$, $J_1 = V_{p-a+1}\otimes D^{a-1}$ and 
        $J_2 = V_{p-1-a}\otimes D^a$ if $3 \leq a\leq p-1.$
	\end{enumerate}
 \end{enumerate}
\end{prop}

The next lemma  will be used many times below and describes some explicit properties of the maps $W_i\twoheadrightarrow J_i$, 
for $i=0$, $1$, $2$. 

\begin{lemma}
 \label{generator}
Let $p\geq 3$, $r\geq 2p$, $ r\equiv a\mod (p-1)$, with $1\leq a\leq p-1$. 

If $3 \leq a\leq p-1$, then:
 \begin{enumerate}  

   \item[(i)] 
The image of 
$\theta X^{r-p-1}$ in $W_0$ maps to
$X^{a-2} \in J_0$.  

   \item[(ii)] 
The image of 
$\theta X^{r-p-a+1}Y^{a-2}$ in $W_1$ maps to
$X^{p-a+1} \in J_1$. 

   \item[(iii)] 
    The image of 
$X^{r-i}Y^{i}$ in $W_2$ maps to 
$0 \in J_2$, 
for $0\leq i\leq a-1$, whereas the image of 
$X^{r-a}Y^{a}$ in $W_2$ maps to 
$X^{p-a-1}$ in $J_2$.
  \end{enumerate}

If $a = 2$, then:
 \begin{enumerate}  

   \item[(i)] 
The image of 
$\theta X^{r-p-1}$ in $W_0$ maps to
$X^{p-1} \in J_0$. 

   \item[(ii)] 
For $r > 2p$,
the image of 
$\theta X^{r-2p}Y^{p-1}$ in $W_1$ maps to
$1 \in J_1$. 

   \item[(iii)] The image of 
$X^{r-i}Y^{i}$ in $ W_2$ maps to 
$0 \in J_2$,  
for $ i=0, 1$, whereas
the image of 
$X^{r-2}Y^{2}$ in $ W_2$ maps to 
$X^{p-3} \in J_2$.
\end{enumerate}

If $a = 1$, then:
 \begin{enumerate}  

   \item[(i)] 
If $p\mid r$,
the image of 
$\theta X^{r-p-1}$ in $W_0$ maps to
$X^{p-2} \in J_0$. 

   \item[(ii)] 
The image of 
$\theta X^{r-2p+1}Y^{p-2}$ in $W_1$ maps to
$X \in J_1$. 

   \item[(iii)] 
The image of 
$X^{r-1}Y$ in $W_2$ maps to
$X^{p-2} \in J_2$. 

  \end{enumerate}
\end{lemma}

\begin{proof}
 The proof is elementary and consists of explicit calculation with the maps given in \cite[(4.2)]{G78} 
and \cite[Lem. 5.3]{Br03}.
\end{proof}

Let $U_i$ denote the image of $\mathrm{ind}_{KZ}^G W_i$ under the map $ \mathrm{ind}_{KZ}^G P\twoheadrightarrow\bar\Theta = \bar\Theta_{k,a_p}$, and let 
$F_i:=U_i/U_{i-1}$, for $i=0$, $1$, $2$, with  $U_{-1}:=0$. Then we have the following commutative diagrams of $G$-modules:

\begin{eqnarray}
\label{gendiagram1}
\begin{gathered}
    \xymatrix{
    0 \ar[r] & \mathrm{ind}_{KZ}^G W_1\ar@{->>}[d]\ar[r] & \mathrm{ind}_{KZ}^GP \ar@{->>}[d] \ar[r] & \mathrm{ind}_{KZ}^G J_2 \ar@{->>}[d]\ar[r] & 0 \\
    0 \ar[r] & U_{1} \ar[r]  & \bar{\Theta} \>(=U_2) \ar[r] & F_2 \ar[r] & 0 ,}
\end{gathered}
\end{eqnarray}

\begin{eqnarray}\label{gendiagram2}
\begin{gathered}
    \xymatrix{
    0 \ar[r] & \mathrm{ind}_{KZ}^G J_0\ar@{->>}[d]\ar[r] & \mathrm{ind}_{KZ}^G W_1 \ar@{->>}[d] \ar[r] & \mathrm{ind}_{KZ}^G J_1 \ar@{->>}[d]\ar[r] & 0 \\
    0 \ar[r] & F_0 \> (=U_0)\ar[r]  & U_1 \ar[r] & F_1 \ar[r] & 0. }
\end{gathered}
\end{eqnarray}

\noindent The semisimplification of $\bar\Theta$ is completely determined by the $G$-modules $F_0$, $F_1$ and $F_2$.

\section{The case $3\leq a\leq p-1$}
\label{case3ap-1}

Let $r> 2p$ and $r\equiv a\mod (p-1)$, with $3\leq a \leq p-1$.
If $p\nmid r-a$ and $v(a_p) = 1$, then  $u:=\frac{a}{a-r}\cdot\frac{a_p}{p}$ is a $p$-adic unit. 
In this section we will give a proof of the following result.

\begin{theorem}
\label{3ap-1}

 Let $2p<k-2=r\equiv a\mod (p-1)$, with $3\leq a\leq p-1$, and assume that $v(a_p)=1$.  
  \begin{enumerate}

    \item[(i)] If $p\nmid r-a$, then $\bar{V}_{k,a_p}^{ss}\cong \mu_{\lambda}\cdot\omega^a\oplus \mu_{\lambda^{-1}}\cdot \omega$ is reducible, with 
               $\lambda=\bar{u}\in\F_p^\times$.

     \item[(ii)] If $p\mid r-a$, then $\bar{V}_{k,a_p}\cong \mathrm{ind}(\omega_2^{a+1})$ is irreducible.
  \end{enumerate}
\end{theorem}

Let us begin with the following elementary lemma.
\begin{lemma}\label{a mod p}
 Let $r\equiv a\mod (p-1)$, with $3\leq a\leq p-1$. Then we have the congruence 
 $$X^{r-1}Y\equiv \frac{a-r}{a}\cdot\theta X^{r-p-1}\mod\, X_r+V_r^{**}.$$ 
 Hence, if $p\mid r-a$, then $X^{r-1}Y\in X_r+V_r^{**}$.
\end{lemma}

\begin{proof} One checks that the following congruence holds:
  $$aX^{r-1}Y+\underset{k\,\in\,\mathbb{F}_p}\sum k^{p-a}(kX+Y)^r\equiv (a-r)\cdot\theta X^{r-p-1}+F(X,Y)\mod p,$$
  with $F(X,Y) :=(a-r)\cdot X^{r-p}Y^p-\underset{\substack{1<j\,\leq r\\j\equiv 1\mod (p-1)}}\sum\binom{r}{j}\cdot X^{r-j}Y^j \in V_r^{**}$, by \cite[Lem 2.3]{Bhattacharya-Ghate}.
\end{proof}

We use the notation from Section \ref{filtrationP} and the Diagrams \eqref{gendiagram1} and \eqref{gendiagram2}.
We have $J_0=V_{a-2}\otimes D$, $J_1=V_{p-a+1}\otimes D^{a-1}$, $J_2=V_{p-a-1}\otimes D^a$,
by Proposition~\ref{P}. Recall that $F_i$ denotes the factor of $\bar{\Theta}$ which is a quotient of $\mathrm{ind}_{KZ}^G J_i$, for $i=0$, $1$, $2$.

\begin{prop}
\label{Theta}
 Let $2p<r\equiv a\mod (p-1)$, with $3\leq a\leq p-1$, and suppose that $v(a_p)=1$.
 \begin{enumerate}
  \item[(i)] If $p\nmid r-a$, then $F_1=0$ and $\bar\Theta$ fits in the short exact sequence $$0\rightarrow F_0=U_1\rightarrow\bar{\Theta}\rightarrow F_2\rightarrow 0.$$
  \item[(ii)] If $p\mid r-a$, then $U_1=0$, and $\bar{\Theta}\cong F_2$ is a quotient of $\mathrm{ind}_{KZ}^G J_2$.
 \end{enumerate}
\end{prop}

\begin{proof}
 Let us consider $$f_0:=\left[\mathrm{Id},\dfrac{X^{r-a-p+2}Y^{a+p-2}-X^{r-a+1}Y^{a-1}}{a_p}\right] 
 \in\mathrm{ind}_{KZ}^G\mathrm{Sym}^r\bar{\mathbb{Q}}_p^2.$$
 As $r> 2p$,  $T^-f_0$ is integral and dies mod $p$. Also 
$T^+f_0\in \mathrm{ind}_{KZ}^G\langle X^{r-1}Y\rangle+ p\cdot\mathrm{ind}_{KZ}^G\mathrm{Sym}^r\bar{\mathbb{Z}}_p^2$ is integral.

  (i) If $p\nmid r-a$, then $T^+f_0\equiv -\frac{p}{a_p}\cdot
  \underset{\lambda\in\F_p}\sum\left[g_{1,[\lambda]}^0,\,(-[\lambda])^{a-2}X^{r-1}Y\right]\mod
p$. By Lemma \ref{a mod p}, this is the same as 
$-\frac{p}{a_p}\cdot\underset{\lambda\in\F_p}\sum\left[g_{1,[\lambda]}^0,\,(-[\lambda])^{a-2}\dfrac{(a-r)}{a}\cdot\theta
X^{r-p-1}\right]$ in $\mathrm{ind}_{KZ}^G P$. 
Hence $(T-a_p)f_0$ maps to $$ T^+f_0-a_pf_0\equiv
-\frac{p}{a_p}\cdot\underset{\lambda\in\F_p}\sum\left[g_{1,[\lambda]}^0,\,(-[\lambda])^{a-2}\dfrac{(a-r)}{a}\cdot\theta
X^{r-p-1}\right] +\left[\mathrm{Id},\theta
X^{r-p-a+1}Y^{a-2}\right]\in\mathrm{ind}_{KZ}^G P.$$ 
Note that $\theta X^{r-p-1}$ maps to $0\in J_1=W_1/W_0$. By Lemma
\ref{generator} (ii),  the polynomial $\theta X^{r-p-a+1}Y^{a-2}$  maps to $X^{p-a+1}\neq 0$ in $J_1$. We conclude that $(T-a_p)f_0$ maps to
$\left[\mathrm{Id},X^{p-a+1}\right]\in\mathrm{ind}_{KZ}^G J_1$, which
generates it as a $G$-module. Hence  $F_1=0$ and the result follows.  

(ii) If $p\mid r-a$, then by Lemma \ref{a mod p}, we know that $X^{r-1}Y$
maps to $0\in P$. 
Thus $T^+f_0$ 
dies in $\mathrm{ind}_{KZ}^GP$. Hence $(T-a_p)f_0$ is integral and maps to the image of 
 $-a_pf_0=[\mathrm{Id},\theta X^{r-p-a+1}Y^{a-2}]$ in $\mathrm{ind}_{KZ}^G P$. 
 By Proposition \ref{es1} (ii), $W_1\cong V_r^*/V_r^{**}$ is a non-split extension of the weight $J_1$ by $J_0$. 
 By Lemma \ref{generator} (ii), the polynomial $\theta
X^{r-p-a+1}Y^{a-2}$  maps to $X^{p-a+1}\neq 0$ in $J_1$ and hence
generates  $W_1 $ as a $\Gamma$-module, so $[\mathrm{Id},\theta
X^{r-p-a+1}Y^{a-2}]$ generates $\mathrm{ind}_{KZ}^G W_1$ as a
$G$-module. This proves that the surjection
$\mathrm{ind}_{KZ}^G W_1\twoheadrightarrow U_1$ is 
the zero map.
\end{proof}

\begin{prop}
\label{red}
 Let $2p<r\equiv a\mod (p-1)$, with $3\leq a\leq p-1$. Assume that $p\nmid r-a$ and $v(a_p)=1$, so
 that $u:=\frac{a}{a-r}\cdot\frac{a_p}{p}$ is a $p$-adic unit. Then 
   \begin{enumerate}
      \item[(i)] $F_0$ is a quotient of $\pi(a-2,\bar{u},\omega)$, 
      \item[(ii)] $F_2$ is a quotient of $\pi(p-a-1,\bar{u}^{-1},\omega^a)$. 
    \end{enumerate}
\end{prop}

\begin{proof}
  (i) Consider $f_0\in\mathrm{ind}_{KZ}^G\mathrm{Sym}^r\mathbb{\bar{Q}}_p^2$ given by
              $$f_0=\left[\mathrm{Id},\,\frac{X^{r-1}Y-X^{r-p}Y^p}{p}\right]=\left[\mathrm{Id},\, \frac{\theta X^{r-p-1}}{p}\right].$$
              Since $r\geq p+2$, 
              $T^-f_0\equiv 0\mod p$. By the formula for the Hecke operator,
              $$(T-a_p)f_0\equiv T^+ f_0 - a_p f_0\equiv\underset{\lambda\in\mathbb{F}_p}\sum 
              \left[g_{1,[\lambda]}^0,X^{r-1}Y\right]-\overline{(a_p/p)}
              \left[\mathrm{Id},\,\theta X^{r-p-1}\right]\mod p.$$ By Lemma 
              \ref{a mod p}, $(T-a_p)f_0$ maps to 
              $\overline{\frac{(a-r)}{a}}\cdot\underset{\lambda\in\mathbb{F}_p}\sum \left[g_{1,[\lambda]}^0,\,\theta X^{r-p-1}\right]-\overline{(a_p/p)}
              \left[\mathrm{Id},\,\theta X^{r-p-1}\right]\in\mathrm{ind}_{KZ}^G P$, 
              which in fact lies in the submodule $\mathrm{ind}_{KZ}^G W_0$.
               This maps to 
              $\overline{\frac{(a-r)}{a}}\cdot(T-\bar{u})\left[\mathrm{Id},\,X^{a-2}\right]\in \mathrm{ind}_{KZ}^G J_0$, by Lemma \ref{generator} (i).
               The element 
              $\left[\mathrm{Id},X^{a-2}\right]$ generates 
              $\mathrm{ind}_{KZ}^G J_0$ as a $G$-module, so the map $\mathrm{ind}_{KZ}^G J_0\twoheadrightarrow F_0$ must factor through $\pi(a-2,\bar{u},\omega)$.
 
 (ii) Let us consider $f=f_0+f_1+f_2\in\mathrm{ind}_{KZ}^G\mathrm{Sym}^r\mathbb{\bar{Q}}_p^2$, where 
              \begin{eqnarray*}
               f_2 &=&\underset{\lambda\in\mathbb{F}_p}\sum \left[g_{2,p[\lambda]}^0,\frac{1}{p}\cdot(Y^r-X^{p-1}Y^{r-p+1})\right],\\
               f_1 &=&\left[g_{1,0}^0,\frac{(p-1)}{pa_p}\underset{\substack{0<j<r\\j\equiv a\mod (p-1)}}\sum\alpha_j\cdot X^{r-j}Y^j\right],\\
               f_0 &=& \begin{cases}
                    0, & \text{ if } 3\leq a< p-1,\\
                    \left[\mathrm{Id}, \,\dfrac{(1-p)}{p}\cdot(X^r-X^{r-p+1}Y^{p-1})\right], & \text{ if } a=p-1,
                   \end{cases}
                   \end{eqnarray*}
 where the $\alpha_j$ are integers from Lemma \ref{comb2}.  We compute that 
 $T^+f_2\in\mathrm{ind}_{KZ}^G\langle X^{r-1}Y\rangle + p\cdot\mathrm{ind}_{KZ}^G\mathrm{Sym}^r\bar\Z_p^2$, which maps to 
 $0\in\mathrm{ind}_{KZ}^G J_2$, by Lemma \ref{generator} (iii).  By Lemma \ref{comb2} and as $p\geq 5$, both $T^+f_1$ and $T^-f_1$ die mod $p$. We have
 $$-a_pf_2\equiv \underset{\lambda\in\mathbb{F}_p}\sum \left[g_{2,p[\lambda]}^0,\frac{a_p}{p}\cdot X^{p-1}Y^{r-p+1} \right] \mod X_r.$$
 Using that  $X^{p-1}Y^{r-p+1}\equiv X^{r-a}Y^a\mod V_r^*$, and Lemma \ref{generator} (iii), we get that $-a_pf_2$
  maps to $\frac{a_p}{p}\cdot\underset{\lambda\in\mathbb{F}_p}\sum \left[g_{2,p[\lambda]}^0, X^{p-a-1}\right] \in \mathrm{ind}_{KZ}^G J_2$.
 Moreover, $T^+f_0-a_pf_1+T^-f_2$ is integral and congruent to 
 $$ \left[g_{1,0}^0,\,\underset{\substack{0<j<r\\j\equiv a\mod (p-1)}}\sum\dfrac{p-1}{p}\left(\binom{r}{j}-\alpha_j\right)\cdot X^{r-j}Y^j  + Y^r \right]\mod p,$$ 
 unless $a=p-1$, in which case we also have the terms $\underset{\lambda\in\F_p}\sum \left[g_{1,[\lambda]}^0,\,(-[\lambda])^{p-2}\cdot X^{r-1}Y\right]$ from $T^+f_0$, in addition to the above.
 In any case, $T^+f_0-a_pf_1+T^-f_2$ maps to $\left[g_{1,0}^0,\,\dfrac{r-a}{a}\cdot X^{p-a-1}\right]\in\mathrm{ind}_{KZ}^G J_2$, by Lemma \ref{comb2}, Lemma \ref{comb1} and
 Lemma \ref{generator} (iii).
 If $a=p-1$, then  $T^-f_0$ dies mod $p$ and $-a_pf_0\equiv \left[\mathrm{Id}, \,\frac{a_p}{p}\cdot X^{r-p+1}Y^{p-1} \right] \mod X_r$ maps to 
 $\frac{a_p}{p}\cdot\left[\mathrm{Id}, \,1\right]$ in $\mathrm{ind}_{KZ}^G J_2=\mathrm{ind}_{KZ}^G V_0$. 
 
 So $(T-a_p)f$ is always integral and maps to 
 $\frac{a_p}{p}\left(T-\frac{p}{a_p}\cdot\frac{(a-r)}{a}\right)[g_{1,0}^0,\, X^{p-a-1}]$ in $\mathrm{ind}_{KZ}^G J_2$. Therefore $F_2$ is a 
 quotient of $\pi(p-a-1, \bar{u}^{-1},\omega^a)$, as desired.
 \end{proof}

\begin{proof}[Proof of Theorem \ref{3ap-1}]
 We have
 \begin{enumerate}
   \item[(i)] If $p\nmid r-a$, Proposition \ref{Theta} (i) tells us that  $\bar{\Theta}^{ss}\cong (F_0\oplus F_2)^{ss}$. By Proposition \ref{red}, we have surjections 
   $ \pi(a-2,\bar{u},\omega)\twoheadrightarrow F_0$, and
    $\pi(p-a-1,\bar{u}^{-1},\omega^a)\twoheadrightarrow F_2$. Using the fact that $\bar{\Theta}$ lies in the image of the mod $p$ Local Langlands 
    Correspondence, we deduce that 
    $\bar{\Theta}^{ss}\cong \pi(a-2,\bar{u},\omega)^{ss}\oplus\pi(p-a-1,\bar{u}^{-1},\omega^a)^{ss}$. 

   \item[(ii)] If $p\mid r-a$, then Proposition \ref{Theta} (ii) and \cite[Prop. 3.3]{BG09} together imply that $\bar{V}_{k,a_p}\cong \mathrm{ind}(\omega_2^{a+1})$ is irreducible.
 \end{enumerate}
\end{proof}


\section{The case $a=1$}
\label{casea=1}

We now treat the case $r\equiv 1\mod (p-1)$ separately since the computations are more complicated.

\begin{theorem}
\label{a=1}
 Let $p\geq 5$, $r\geq 2p$ and $r\equiv 1\mod (p-1)$. Assume that $v(a_p) = 1$. 
 \begin{enumerate}
  \item[(i)] If $p\nmid r$, then $\bar V_{k,a_p}\cong \mathrm{ind}(\omega_2^2)$.
  \item[(ii)] If $p\mid r$, then $\bar V_{k,a_p}^{ss}\cong \mu_\lambda \cdot \omega\oplus\mu_{\lambda^{-1}} \cdot \omega$, where $\lambda^2-c\lambda+1=0$ with 
  $c=\overline{\dfrac{a_p}{p}-\dfrac{r-p}{a_p}}\in\bar \F_p$.
 \end{enumerate}
\end{theorem}

Recall that by Proposition~\ref{P}, we have $J_2=V_{p-2}\otimes D$, $J_1=V_1$, and $J_0=\begin{cases}
                                                                                                 V_{p-2}\otimes D, & \text{if } p\mid r,\\
                                                                                                 0, & \text{if } p\nmid r.
                                                                                               \end{cases}$

\begin{prop}\label{F_2=0}
 If $p\geq 3$, $ r > 2p$ and $r\equiv 1\mod (p-1)$, then $F_2=0$. As a consequence, 
 \begin{enumerate}
 \item[(i)] If $p\nmid r$, then $\bar{\Theta}\cong F_1$ is a quotient of $\mathrm{ind}_{KZ}^G J_1$. 
 \item[(ii)] If $p\mid r$, then $\bar{\Theta}\cong U_1$, and the bottom row of Diagram \eqref{gendiagram2} reduces to 
 $$0\rightarrow F_0 \rightarrow\bar{\Theta}\rightarrow F_1\rightarrow 0.$$
 \end{enumerate}
\end{prop}

\begin{proof}
 Consider the function $f=f_0\in\mathrm{ind}_{KZ}^G\mathrm{Sym}^r\bar{\Q}_p^2$ given by 
 $$f_0=\left[\mathrm{Id},\dfrac{XY^{r-1}-2X^pY^{r-p}+X^{2p-1}Y^{r-2p+1}}{p}\right].$$
 One checks that $T^+f_0,\,T^-f_0,\,a_pf_0$ are all integral, and that $T^+f_0\equiv 0\mod p$, $a_pf_0\in\mathrm{ind}_{KZ}^G V_r^*$ and
 $T^-f_0\equiv [\alpha, XY^{r-1}]\mod p$. By Lemma \ref{generator} (iii) and by $\Gamma$-linearity,
                                                           $XY^{r-1}=\left(\begin{smallmatrix}
                                                                     0 & 1\\
                                                                     1 & 0
                                                                    \end{smallmatrix}\right)\cdot X^{r-1}Y$
         maps to $-Y^{p-2}$, under the map $W_2\twoheadrightarrow J_2$.                                                           
So $(T-a_p)f$ maps to $[\alpha,-Y^{p-2}]\in\mathrm{ind}_{KZ}^G J_2$. 
 As $[\alpha,-Y^{p-2}]$ generates $\mathrm{ind}_{KZ}^G J_2$ as a $G$-module,
 we have $F_2=0$.
\end{proof}

\begin{prop}
 \label{elimination and}
 Let $p\geq 3$, $r> 2p$, $r\equiv 1\mod (p-1)$ and $p\mid r$. Then we have 
 \begin{enumerate}
  \item[(i)] $F_1=0$, and 
  \item[(ii)] $F_0$ is a quotient of $\dfrac{\mathrm{ind}_{KZ}^G (V_{p-2}\otimes D)}{T^2-cT+1}$, where
         $c=\overline{\dfrac{a_p}{p}-\dfrac{r-p}{a_p}}\in\bar\F_p$.
 \end{enumerate}
\end{prop}

\begin{proof}
 \begin{enumerate}
  \item[(i)] Consider the function $f=f_2+f_1+f_0\in\mathrm{ind}_{KZ}^G\mathrm{Sym}^r\bar{\Q}_p^2$ given by  
  \begin{eqnarray*}
   f_2 &=&\underset{\lambda\in\F_p^*}\sum\left[g_{2,p[\lambda]}^0,\, \dfrac{[\lambda]^{p-2}}{a_p}\cdot(Y^r-X^{r-p}Y^p)\right]
   +\left[g_{2,0}^0,\,\dfrac{1-p}{a_p}\cdot(XY^{r-1}-X^{r-p+1}Y^{p-1})\right],\\
   f_1 &=&\left[g_{1,0}^0,\:\dfrac{X^pY^{r-p}-XY^{r-1}}{p}+ 
   \dfrac{p-1}{a_p^2}\cdot\underset{\substack{0<j<r-1\\j\equiv 0\mod (p-1)}}\sum \beta_j X^{r-j}Y^j\right],\\
   f_0 &=& \left[\mathrm{Id},\dfrac{1-p}{a_p}\cdot(X^r-X^pY^{r-p})\right],
  \end{eqnarray*}
  where $\beta_j$ are the integers from Lemma \ref{comb3} (ii). Using the facts that $p\mid r-p$ and $p\geq 3$, 
  we see that $T^+f_2$ and $T^-f_0$ die 
  mod $p$. We compute that 
  \begin{equation}
   \label{eqn1} 
     T^-f_1-a_pf_0\equiv [\mathrm{Id}, -XY^{r-1}-(X^r-X^pY^{r-p})]\mod p
     \equiv[\mathrm{Id}, \theta Y^{r-p-1}]\mod X_r.
   \end{equation}
   
  Note that $T^+f_0+T^-f_2-a_pf_1$ is congruent mod $p$ to
  $$\left[g_{1,0}^0,\,\underset{\substack{0<j<r-1\\j\equiv 0\mod (p-1)}}\sum
  \dfrac{p-1}{a_p}\left(\binom{r}{j}-\beta_j\right)\cdot X^{r-j}Y^j +\dfrac{p-1}{a_p}(r-p)\cdot XY^{r-1}-\dfrac{a_p}{p}\cdot\theta Y^{r-p-1}\right],$$
  which is integral because $v(a_p)=1$, $p\mid r-p$ and each $\beta_j\equiv \binom{r}{j}\mod p$. Rearranging the terms, we can write 
  \begin{equation*}
    T^+f_0+T^-f_2-a_pf_1\equiv \left[g_{1,0}^0,\,(p-1)\left(F(X,Y)
    +\dfrac{p-r}{a_p}\cdot\theta Y^{r-p-1}\right)-\dfrac{a_p}{p}\cdot\theta Y^{r-p-1}\right],
  \end{equation*}
where $F(X,Y) := \underset{\substack{0<j<r-1\\j\equiv 0\mod (p-1)}}\sum\dfrac{1}{a_p}\cdot\left(\binom{r}{j}-\beta_j\right)\cdot X^{r-j}Y^j -\dfrac{p-r}{a_p} \cdot X^pY^{r-p}\in V_r^{**}$, as it satisfies
  the criteria given in \cite[Lem. 2.3]{Bhattacharya-Ghate}. Thus inside $\mathrm{ind}_{KZ}^G P$, we have
  \begin{equation}
  \label{eqn2}
        T^+f_0+T^-f_2-a_pf_1\equiv \left[g_{1,0}^0,\,\left(-\dfrac{p-r}{a_p}-\dfrac{a_p}{p}\right)\cdot\theta Y^{r-p-1}\right].
  \end{equation}

  Then we compute that $T^+f_1-a_pf_2$ is congruent 
  to
\begin{eqnarray*}
  \underset{\lambda\in\F_p}\sum\left[g_{2,p[\lambda]}^0,\,(-[\lambda])^{r-p-1}(-p+1)\cdot X^{r-1}Y\right]  & - & 
\underset{\lambda\in\F_p^*}\sum\left[g_{2,p[\lambda]}^0,\,(-[\lambda])^{p-2}(-p+1)\cdot (Y^r-X^{r-p}Y^p)\right] \\
                   & - & [g_{2,0}^0, XY^{r-1}-X^{r-p+1}Y^{p-1}]\mod p.
\end{eqnarray*}
Going modulo $X_r$ and $V_r^{**}$, we get
\begin{equation}
\label{eqn3}
           T^+f_1-a_pf_2 \equiv \underset{\lambda\in\F_p^*}\sum\left[g_{2,p[\lambda]}^0,\,(-[\lambda])^{p-2}\theta X^{r-p-1}\right]
           +\left[g_{2,0}^0,\,\theta Y^{r-p-1}+\frac{(r-2p+1)}{p-1}\cdot \theta X^{r-2p+1}Y^{p-2}\right].   
\end{equation} 
  
We know $\theta X^{r-p-1}, \theta Y^{r-p-1}$ map to $0 \in J_1=W_1/W_0$
and $\theta X^{r-2p+1}Y^{p-2}$ maps to $X\in J_1=V_1$, by Lemma
\ref{generator} (ii). By the equations \eqref{eqn1}, \eqref{eqn2} and
\eqref{eqn3} above, $(T-a_p)f$ is integral and its image in
$\mathrm{ind}_{KZ}^G P$ lies in $\mathrm{ind}_{KZ}^G W_1$ and maps to
$[g_{2,0}^0,\,-X]\in\mathrm{ind}_{KZ}^G J_1$, which generates
$\mathrm{ind}_{KZ}^G J_1$ over $G$. Hence $F_1=0$. 

\item[(ii)]
 Consider the function $f=f_2+f_1+f_0\in\mathrm{ind}_{KZ}^G\mathrm{Sym}^r\bar{\Q}_p^2$ given by 
\begin{eqnarray*}
   f_2 &=&\underset{\lambda,\mu\in\F_p}\sum\left[g_{2,p[\mu]+[\lambda]}^0,\, \dfrac{Y^r-X^{r-p}Y^p}{a_p}\right],  \\
   f_1 &=&\underset{\lambda\in\F_p}\sum\left[g_{1,[\lambda]}^0,\,-\dfrac{1}{p}\cdot(X^{r-1}Y-X^{r-p}Y^p)+\dfrac{p-1}{a_p^2}\cdot
   \underset{\substack{1<j<r\\j\equiv 1\mod (p-1)}}\sum \alpha_j X^{r-j}Y^j\right],\\
   f_0 &=& \left[\mathrm{Id},\dfrac{1-p}{a_p}\cdot(X^{r-1}Y-X^{r-p}Y^p)\right],
\end{eqnarray*}
 where the $\alpha_j$ are the integers from Lemma \ref{comb6}. We check
that $T^+f_2$ and $T^-f_0$ die mod $p$, since $p\mid r-p$ and $r\geq 2p$.
Next we compute that modulo $p$ and $X_r$, $T^+f_0-a_pf_1+T^-f_2$ is
congruent to $$
\underset{\lambda\in\F_p}\sum\left[g_{1,[\lambda]}^0,\,\dfrac{(p-1)(r-p)}{a_p}X^{r-1}Y
+\underset{\substack{1<j<r\\ j\equiv 1\mod
(p-1)}}\sum\dfrac{p-1}{a_p}\left(\binom{r}{j}-\alpha_j\right)\cdot
X^{r-j}Y^j +\dfrac{a_p}{p}\cdot\theta X^{r-p-1}\right].$$ Rearranging the
terms, we get $$T^+f_0-a_pf_1+T^-f_2\equiv
\underset{\lambda\in\F_p}\sum\left[g_{1,[\lambda]}^0,\,\left(\dfrac{a_p}{p}-\dfrac{r-p}{a_p}\right)\theta
X^{r-p-1}-F(X,Y)\right],$$ where $$F(X,Y)=\dfrac{r-p}{a_p}\cdot
X^{r-p}Y^p+\underset{\substack{1<j<r\\ j\equiv 1\mod
(p-1)}}\sum\dfrac{1}{a_p}\left(\binom{r}{j}-\alpha_j\right)\cdot
X^{r-j}Y^j$$ can be checked to be in $V_r^{**}$, using \cite[Lem.
2.3]{Bhattacharya-Ghate}. Thus in $\mathrm{ind}_{KZ}^G P$, we have
 
\begin{equation} 
\label{eqn1a} 
T^+f_0-a_pf_1+T^-f_2\equiv
 \underset{\lambda\in\F_p}\sum\left[g_{1,[\lambda]}^0,\,\left(\dfrac{a_p}{p}-\dfrac{r-p}{a_p}\right)\theta
X^{r-p-1}\right].  
\end{equation}
 
We check that
$T^+f_1-a_pf_2\equiv\underset{\mu,\lambda\in\F_p}\sum\left[g_{2,p[\mu]+[\lambda]}^0,\,-X^{r-1}Y-(Y^r-Y^pX^{r-p})\right]\mod
p$, so \begin{eqnarray} \label{eqn2a} T^+f_1-a_pf_2\equiv
\underset{\mu,\lambda\in\F_p}\sum\left[g_{2,p[\mu]+[\lambda]}^0,\,-\theta
X^{r-p-1}\right]\mod X_r.  \end{eqnarray}
 
Finally, we compute \begin{equation} \label{eqn3a} T^-f_1-a_pf_0\equiv
      -a_pf_0\equiv\left[\mathrm{Id}, -\theta X^{r-p-1}\right]\mod p.
\end{equation} By Lemma \ref{generator},  the image of $\theta X^{r-p-1}$
(in $ P$) lands in $W_0$ and maps to $X^{p-2}\in J_0$.
  Using the formula for the $T$ operator, and the equations
\eqref{eqn1a}, \eqref{eqn2a}, \eqref{eqn3a}, we obtain that $(T-a_p)f$ is
integral, and in fact it has the same image in $\mathrm{ind}_{KZ}^G P$ as that of $-(T^2-cT+1)[\mathrm{Id},
X^{p-2}]\in \mathrm{ind}_{KZ}^G J_0$, where
$c=\overline{\dfrac{a_p}{p}-\dfrac{r-p}{a_p}}\in\bar\F_p$.
\end{enumerate} 
\end{proof}

\begin{proof}[Proof of Theorem \ref{a=1}]
 Part (i) follows from Proposition \ref{F_2=0} (i) and \cite[Prop. 3.3]{BG09}, at least if $p > 3$.
 If $p\mid r$, then 
  $\bar\Theta$
 is a quotient of $\dfrac{\mathrm{ind}_{KZ}^G V_{p-2}\otimes D}{T^2-cT+1}$, by Propositions \ref{F_2=0} (ii) and  \ref{elimination and}. 
 Now  part (ii) follows by the mod $p$ semisimple Local Langlands Correspondence.
\end{proof}


\section{The case $a=2$}
\label{casea2}
This section deals with the remaining case $r\equiv 2\mod (p-1)$, for $r\geq 2p$ and $p>3$. 
This case is special in the sense that there are further complications with the computation. 
We shall show that the reduction $\bar{V}^{ss}_{k,a_p}$ depends on the size of the 
valuation of the quantity $$a_p^2-\binom{r}{2}p^2,$$ which is at least $2$, since the slope $v(a_p)=1$.
In fact, we show the reduction is determined according to the following trichotomy: whether
the valuation of this quantity is equal to, greater than or less than the integer
$2 + v(r-2)$.

\begin{thm}
\label{thma2}

Let $p> 3$, $r\geq 2p$, with $r\equiv 2\mod (p-1)$. Assume that $v(a_p)=1$.

\begin{enumerate}[(i)]

\item
If $v(a_p^2-\binom{r}{2}p^2) = 2+v(r-2)$, then
$$
\bar{V}_{k,a_p}^{ss}\cong\mu_{\lambda} \cdot \omega^2\oplus \mu_{\lambda^{-1}} \cdot \omega
$$
where
$\lambda=\overline{\dfrac{2\left(a_p^2-\binom{r}{2}p^2\right)}{(2-r)pa_p}}\in\bar\F_p^\times$.

\item
If $v(a_p^2-\binom{r}{2}p^2) > 2+v(r-2)$, then
$\bar{V}_{k,a_p}\cong\mathrm{ind}(\omega_2^{p+2})$.

\item
If $v(a_p^2-\binom{r}{2}p^2) < 2+v(r-2)$, then
$\bar{V}_{k,a_p}\cong\mathrm{ind}(\omega_2^{3})$.

\end{enumerate}
\end{thm}

\begin{remark}
When $a = 2$,
there are {\it a priori} three possibilities for $\bar{V}_{k,a_p}^{ss}$, 
if one counts the reducible cases as one possibility.
Remarkably, the theorem shows that all three possibilities do indeed occur.
The first case in the theorem is  the generic case: for most $r$, in fact 
for $r \not\equiv 2 \mod p$, we have $v(r-2) = 0$;
furthermore the condition $v(a_p^2-\binom{r}{2}p^2) = 2$ holds
whenever the unit
$(\frac{a_p}{p})^2$ and the binomial coefficient ${r \choose 2}$ have distinct 
reductions in $\bar\F_p$.
\end{remark}

\subsection{Preliminaries}
\label{finfinity}

Let us introduce some convenient notation. For $x \in \Q$, $x$ non-negative, write 
$O(p^x)$ for an integral function multiplied by a constant with valuation  greater 
than or equal to $x$. 
Throughout this section let us denote  
$$c := \frac{a_p^2-\binom{r}{2}p^2}{pa_p},$$ 
and  set $\tau = v(c)$, $t = v(2-r)$
and $t_0 = \min(\tau,t)$.

\begin{lemma}
\label{defauxfunc}
Let $p> 3$, and let $r\geq 2p$, with $r\equiv 2\mod (p-1)$. 
If
$$
\chi = \sum_{\ell = 0}^{t} 
                    \> a_p^{\ell} \> [g^0_{\ell,0}, Y^r-X^{r-2}Y^2],
$$
then 
\begin{eqnarray*}
  (T-a_p)\chi = [\alpha, 
                                                Y^r] + a_p[1,X^{r-2}Y^2] + p^{t+1} h_\chi +
O(p^{t_0+2}),
\end{eqnarray*}
where $h_\chi$ is an integral linear combination of terms of the form
$[g, X^r]$ and $[g,X^{r-1}Y]$, for $g \in G$.
\end{lemma}

\begin{proof} 
Write $\chi=\sum_{l=0}^t\chi_l$, where $\chi_l:=a_p^l\,[g_{l,0}^0, Y^r-X^{r-2}Y^2]$, for $l \geq 0$.
By the formula for the Hecke operator,
$$T^-\chi_0=[\alpha, Y^r]+O(p^{t_0+2}) \quad\text{and} \quad -a_p\chi_0+T^-\chi_1=a_p[1, X^{r-2}Y^2]+O(p^{t_0+2}),$$
as $r-2>t+2\geq t_0+2$. Similarly, for $1\leq l\leq t-1$, one computes 
$$T^-\chi_{l+1}-a_p\chi_l+T^+\chi_{l-1}= p^{t+1}h_{l}+O(p^{t_0+2}),$$
where each $h_{l}$ is an integral linear combination of terms of the form 
$[g, X^{r-1}Y]$. 
Here we use the facts that for $j\geq 3$, $v(p^j{r\choose j})\geq t+3$, by Lemma~\ref{congruences2}, and that $v({r\choose 2}-1)\geq t$, therefore  $v(p^2({r\choose 2}-1))\geq t+2$ and
$v(a_p^2-p^2)=v(pa_pc+p^2({r\choose 2}-1))\geq \min(\tau+2, t+2)=t_0+2$.
Similarly for $l=t$, we compute
$$-a_p\chi_t+T^+\chi_{t-1}=p^{t+1}h_{t}+O(p^{t_0+2}),$$
where $h_{t}$ is an integral linear combination of terms of the form 
$[g, X^{r-1}Y]$ and $[g, Y^r]$. 
Finally 
$$T^+\chi_t=p^{t+1}h_{t+1}+O(p^{t_0+2}),$$
where $h_{t+1}$ is an integral linear combination of terms of the form $[g, X^{r-1}Y]$.
Combining all the equations above, we obtain the lemma, with 
$h_\chi=\sum_{l=1}^{t+1} h_l$.
\end{proof}

For $g \in G$, set $$\chi_g := g\chi = \sum_{\ell = 0}^{t} 
                          \> a_p^{\ell} \> [gg^0_{\ell,0}, Y^r-X^{r-2}Y^2]. $$ 

By Proposition \ref{P} (ii) (b), the JH factors of $P$ are given by
$J_0=V_{p-1}\otimes D$ and $J_2=V_{p-3}\otimes D^2$, and in addition,
$J_1=V_0\otimes D$ when $r \gneq 2p$ (we set $J_1 = 0$ when $r = 2p$).
Recall that the factors $F_i$ of $\bar\Theta$ are quotients of
$\mathrm{ind}_{KZ}^G J_i$, for $i=0$, $1$, $2$, as explained in Diagrams
\eqref{gendiagram1} and \eqref{gendiagram2}. We now study each of the
factors $F_i$ in turn, starting with $F_2$.

\subsection{Study of $F_2$}

The following proposition tells us about  $F_2$, when $t \leq \tau$.

\begin{prop} 
\label{F2a2}
Let $p > 3$, $r \geq 2p$, with $r \equiv 2 \mod (p-1)$ and assume $v(a_p) = 1$.
\begin{enumerate} 
\item [(i)] If $\tau > t$, 
                then $F_2 = 0$.
\item [(ii)] If $\tau = t$, 
                then $F_2$ is a quotient
of 
 $\pi(p-3,\lambda^{-1},\omega^2)$, where $\lambda = 
                     \overline{\dfrac{2(a_p^2-\binom{r}{2}p^2)}{pa_p(2-r)}} \in \bar\F_p^\times$.
\end{enumerate}
\end{prop}

\begin{proof}
Assume $t \leq \tau$, so $t_0 = t$ in Lemma~\ref{defauxfunc}.
Let $f = f_0 + f_\infty  \in \mathrm{ind}_{KZ}^G \> \mathrm{Sym}^r \bar\Q_p^2$, with
\begin{eqnarray*} 
  f_0 &  = & 
        \frac{p-1}{pa_p(2-r)} \left[1, \sum_{0 < j < r \atop {j \equiv 2 \mod (p-1)}}\binom{r}{j}X^{r-j}Y^j + \frac{p(r-2)}{2}X^{r-2}Y^2\right], \\
  f_{\infty} &  = & \frac{1}{p(2-r)} \left(\sum_{\lambda\in\F_p^\times}\chi_{g^0_{1,[\lambda]}}
+ (1-p)\chi_{g^0_{1,0}}\right).
\end{eqnarray*}

We have  $T^-f_0 = O(p)$. Indeed
$$T^-f_0 = \frac{p-1}{pa_p(2-r)} \left[\alpha, 
\sum_{0 < j < r \atop {j \equiv 2 \mod (p-1)}} \binom{r}{j}p^{r-j}X^{r-j}Y^j\right]  +
O(p),$$ and for all $j$ in the sum, we have $r-j \geq p-1 \geq 3$,
so this is a consequence of Lemma~\ref{congruences2}.
Also,
\begin{eqnarray*}
  T^+f_0  & = & \frac{p-1}{pa_p(2-r)} \Bigg ( \sum_{\lambda\in\F_p^\times} \left[ g_{1,[\lambda]}^0,\sum_{i = 0}^rp^i(-[\lambda])^{2-i} \left(
\sum_{0 < j < r \atop {j \equiv 2 \mod (p-1)}}  \binom{r}{j}\binom{j}{i} + \frac{p(r-2)}{2} \binom{2}{i} \right) X^{r-i}Y^i \right] \\
   &  & \qquad \qquad + \quad \left[g_{1,0}^0,\sum_{0 < j < r \atop {j \equiv 2 \mod (p-1)}}\binom{r}{j}p^jX^{r-j}Y^j +
\frac{p^3(r-2)}{2} X^{r-2}Y^2\right] \Bigg).
\end{eqnarray*}
By parts (1) and (2) of Proposition~\ref{congrbinom2}, we may 
write $\sum_j\binom{r}{j}+p(r-2)/2 = p^{t+2}a$, and $\sum_j j\binom{r}{j} + p(r-2) = p^{t+1}b$, for
some $a,b\in \Z_p$. Moreover, by part (3) and (4) of the same proposition, we have
$p^2 \sum_j \binom{j}{2} \binom{r}{j} + p^3(r-2)/2 = p^2 \binom{r}{2}/(1-p) \mod p^{t+3}$ and
$p^i\sum_j\binom{j}{i}\binom{r}{j} = 0 \mod {p^{t+3}}$, for $i \geq 3$. Here all four sums are over
$j$ with $0 < j < r$ and $j \equiv 2 \mod (p-1)$. Hence, 
we may write
\begin{eqnarray*}
T^+f_0 & = & \frac{p-1}{pa_p(2-r)} \Bigg(
\sum_{\lambda\in\F_p^\times}
\left[g_{1,[\lambda]}^0, [\lambda]^2 a p^{t+2} X^r - b [\lambda] p^{t+2} X^{r-1}Y + 
                   p^2\dfrac{\binom{r}{2}}{1-p} X^{r-2}Y^2\right] \\
  &   & \qquad \qquad + \quad \left[g_{1,0}^0,\binom{r}{2}p^2X^{r-2}Y^2 \right] \Bigg) + O(p),
\end{eqnarray*}
by Lemma~\ref{congruences2}.

On the other hand, a simple computation shows that
\begin{eqnarray*}
\frac{1}{p(2-r)} \left( \sum_{\lambda \in \F_p^\times} [g^0_{1,[\lambda]} \alpha, Y^r]  +  (1-p) [g^0_{1,0} \alpha, Y^r] \right) = \frac{p-1}{p(2-r)} \left[ 1, \sum_{0 < j < r \atop {j \equiv 2 \mod (p-1)}} \binom{r}{j} X^{r-j} Y^j \right],
\end{eqnarray*}
so, by Lemma~\ref{defauxfunc}, we have
\begin{eqnarray*}
(T-a_p)f_ {\infty} & = & 
\frac{p-1}{p(2-r)} \left[1, \sum_{0 < j < r \atop {j \equiv 2 \mod (p-1)}} \binom{r}{j} X^{r-j} Y^j\right]  \\
 & + & 
\frac{a_p}{p(2-r)}\sum_{\lambda\in\F_p^\times}[g^0_{1,[\lambda]},X^{r-2}Y^2]
+
\frac{a_p(1-p)}{p(2-r)}[g^0_{1,0},X^{r-2}Y^2] + h
+ O(p),
\end{eqnarray*}
where $h$ is an integral linear combination of terms of the form $[g, X^r]$ and
$[g,X^{r-1}Y]$, for $g \in G$.

Putting everything together we obtain:
\begin{eqnarray*}
(T-a_p)f = (T-a_p)(f_0+f_\infty) &  = &  
\frac{a_p^2-\binom{r}{2}p^2}{pa_p(2-r)}\sum_{\lambda\in
\F_p}[g_{1,[\lambda]}^0,X^{r-2}Y^2]
-\frac{1}{2}[1,X^{r-2}Y^2]
+ h' + O(p),
\end{eqnarray*}
where $h'$ is an integral linear combination of elements of the form
$[g,X^r]$ and $[g,X^{r-1}Y]$, for $g \in G$.

By  Lemma~\ref{generator} (iii), for $a=2$, we see $X^{r-1}Y$ and $X^r$ die in $J_2$, and $X^{r-2}Y^2$ maps to $X^{p-3} \in J_2$. Thus the image of $(T-a_p)f$ in $\mathrm{ind}_{KZ}^G J_2$ 
is
$$
\overline{\frac{a_p^2-\binom{r}{2}p^2}{pa_p(2-r)}} \sum_{\lambda\in
\F_p}[g_{1,[\lambda]}^0,X^{p-3}]
-\frac{1}{2}[1,X^{p-3}],
$$
which immediately proves both parts of the proposition.
\end{proof}

\subsection{Study of $F_1$}

Before studying the structure of $F_1$, let us first prove the following useful lemma.

\begin{lemma}
\label{goodcasenew}
 Let $p\geq 3$, $r\geq 2p$ and  $r\equiv 2\mod (p-1)$. 
 \begin{enumerate}
  \item[(i)] The image of $X^{r-1}Y$  in $P$ lands in $W_1$ 
   and maps to $-\dfrac{r}{2}$, under the map $W_1\twoheadrightarrow J_1$. 
  \item[(ii)] If $p\mid r$, then $X^{r-1}Y\equiv\theta X^{r-p-1}\mod X_r+V_r^{**}$.
  \end{enumerate}
\end{lemma}

\begin{proof}
   By Lemma \ref{generator} (iii), $X^{r-1}Y$ maps to $0 \in J_2$, so it lands in $W_1$. The polynomial 
    $X^{r-1}Y-XY^{r-1}\equiv \theta\cdot(X^{r-p-1}+\frac{r-2p}{p-1}\cdot X^{r-2p}Y^{p-1} 
  +Y^{r-p-1})\mod V_r^{**}$ and so it maps to $-r\in J_1$, by 
    Lemma \ref{generator} (ii).
    As $\Gamma$ acts by determinant on $J_1$, and as $XY^{r-1}=\left(\begin{smallmatrix}
                                                                                             0 & 1\\
                                                                                             1 & 0
                                                                                            \end{smallmatrix}\right)\cdot X^{r-1}Y$, we obtain that $X^{r-1}Y$ and  $Y^{r-1}X$ map to $-r/2$ 
                                                                                            and $r/2$ respectively  in $J_1$, 
proving (i).

             If $p\mid r$, then $X^{r-1}Y$ maps to $0\in J_1$, so in fact it lands in $W_0 \subset P$.
               Now consider the polynomial $F(X,Y)=\underset{k\in\F_p}\sum k^{p-2}(kX+Y)^r\in X_r$. Using the fact 
               $p\mid r$, we obtain $$F(X,Y)\equiv-\underset{\substack{0\leq j\leq r\\j\equiv 1\mod (p-1)}}\sum\binom{r}{j}X^{r-j}Y^j\equiv -2X^{r-p}Y^p+H(X,Y)\mod p,$$
               where $H(X,Y)=2X^{r-p}Y^p-\underset{\substack{1< j<r-1\\j\equiv 1\mod (p-1)}}\sum\binom{r}{j}X^{r-j}Y^j$. One checks that $H(X,Y)\in V_r^{**}$, using the 
               criteria given in \cite[Lem. 2.3]{Bhattacharya-Ghate}. Thus we have  $X^{r-p}Y^p\in X_r+V_r^{**}$, proving (ii).
\end{proof}

\begin{prop}
\label{F1a2}
   Let $p\geq 3$, $r \geq 2p$, with $r\equiv 2\mod (p-1)$ and assume $v(a_p)=1$.
   \begin{enumerate}
    \item[(i)] If $p\mid r$, then $F_1=0$.
    \item[(ii)] If $p\nmid r$, then $F_1$ is a proper quotient of
$\pi(0,\nu,\omega)$, with $\nu=\overline{\dfrac{rp}{2a_p}}
\in\bar\F_p^\times$. 
   \end{enumerate}
\end{prop}

\begin{proof} Recall that the JH factor $J_1=V_0\otimes D$ of $P$ occurs only when $r>2p$. 
            Therefore $F_1=0$ for $r=2p$, and we assume $r>2p$ for the rest of the proof. 

If $p\mid r$, we consider the function
         \begin{equation}\label{function2}
         f_0=\left[\mathrm{Id},\,\dfrac{X^{r-p}Y^p-X^{r-2p+1}Y^{2p-1}}{a_p}\right]\in\mathrm{ind}_{KZ}^G\mathrm{Sym}^r  
         \bar{\Q}_p^2,
         \end{equation}
          and check that $T^-f_0$ dies mod $p$, and that $T^+f_0$ 
          maps to $\underset{\lambda\in\F_p^*}\sum\left[g_{1,[\lambda]}^0,
          \dfrac{p}{a_p}\cdot\theta X^{r-p-1}\right]$ in $\mathrm{ind}_{KZ}^G P$,
          by Lemma \ref{goodcasenew} (ii).
          Clearly $T^+f_0$
          maps to zero in $\mathrm{ind}_{KZ}^G J_1$. 
          Thus $(T-a_p)f_0$ maps to the
          image of $-a_pf_0\equiv [\mathrm{Id},\,-\theta X^{r-2p}Y^{p-1}]$ in
          $\mathrm{ind}_{KZ}^G J_1$, which is $[\mathrm{Id},-1]$ by Lemma
          \ref{generator} (ii).
          As $[\mathrm{Id},-1]$ generates $\mathrm{ind}_{KZ}^G
          J_1$ over $G$, we have  $F_1=0$.

If $p\nmid r$, we consider the function $f=f_0+f_1\in\mathrm{ind}_{KZ}^G\mathrm{Sym}^r\bar\Q_p^2$, where  
               \begin{eqnarray}\label{function3}
                f_0 &=& \left[\mathrm{Id},\,-\frac{r}{2a_p}\cdot(X^{r-1}Y -2X^{r-p}Y^p + X^{r-2p+1}Y^{2p-1})\right],\\
              \label{function4}  f_1 &=& \left[g_{1,0}^0,\,\frac{1}{p}\cdot(X^{r-1}Y - X^{r-p}Y^p)\right].
               \end{eqnarray}
               As $r>2p$, both $T^-f_0$ and $T^-f_1$ die mod $p$. Then we compute
               \begin{eqnarray*}
                -a_pf_0 &\equiv& \left[\mathrm{Id},\,\frac{r}{2}\cdot(\theta X^{r-p-1}-\theta X^{r-2p}Y^{p-1})\right]\mod p,\\
                T^+f_0-a_pf_1 &\equiv& \left[g_{1,0}^0,\,-\dfrac{rp}{2a_p}\cdot X^{r-1}Y-\frac{a_p}{p}\cdot \theta X^{r-p-1}\right]\mod p,\\
                T^+f_1 &\equiv& \underset{\lambda\in\F_p}\sum\left[g_{2,p[\lambda]}^0,\, X^{r-1}Y\right]\mod p.
               \end{eqnarray*}
               Clearly $(T-a_p)f$ is integral and its image in $\mathrm{ind}_{KZ}^G J_2$ is $0$ by Lemma \ref{generator} (iii). Therefore $(T-a_p)f$ lands in 
               $\mathrm{ind}_{KZ}^G W_1\subset \mathrm{ind}_{KZ}^G P$.
               Applying Lemma \ref{generator} (ii) and Lemma \ref{goodcasenew} (i), we obtain that it maps to 
               $$-\frac{r}{2}\cdot\left(\left[\mathrm{Id},\,1\right]-\left[g_{1,0}^0,\, \frac{rp}{2a_p}\right]+\underset{\lambda\in \F_p}\sum\left[g_{2,p[\lambda]}^0,\,1\right]\right)
               =-\frac{r}{2}\cdot\left(T-\frac{rp}{2a_p}\right)\left[g_{1,0}^0,\,1\right]\in\mathrm{ind}_{KZ}^G J_1.$$
               Since $r\not\equiv 0\mod p$, and as $\left[g_{1,0}^0,\,1\right]$ generates $\mathrm{ind}_{KZ}^G J_1$ over $G$, the map $\mathrm{ind}_{KZ}^G J_1\twoheadrightarrow F_1$ must factor through
               $\pi(0,\overline{\frac{rp}{2a_p}},\omega)$.

Since $\mathrm{ind}_{KZ}^G J_1 \cong (\mathrm{ind}_{KZ}^G \bar\F_p)
\otimes (\omega \circ \det)$, the elements of $\mathrm{ind}_{KZ}^G J_1$
can be thought of as compactly supported $\bar\F_p$-valued functions on
the Bruhat-Tits tree of $\mathrm{SL}_2(\Q_p)$, with the usual right
translation action of $G$ twisted by $\omega \circ \det$.  Now
$(T-a_p)f_1$ is integral and maps to
$-\frac{r}{2}\cdot\underset{\lambda\in\F_p} \sum[
g_{2,p[\lambda]}^0,\,1]$ in $\mathrm{ind}_{KZ}^G J_1$.  Using the
definition of the usual sum-of-neighbours action of the Hecke operator on
functions on the tree, or the formula for $T$ given in (\ref{T}), we can
see that this function cannot lie in the image of $T-\nu$, for any
$\nu\in\bar\F_p^\times$. This shows that $F_1$ must be a proper quotient
of $\pi(0,\overline{\frac{rp}{2a_p}},\omega)$. 
\end{proof}

\subsection{Study of $F_0$}

Recall $c := (a_p^2-\binom{r}{2}p^2)/pa_p$, $\tau = v(c)$, $t = v(2-r)$. We study $F_0$, when $\tau \leq t$.

\begin{prop}
\label{F0a2}
Let $p\geq 3$, $r \geq 2p$, with $r\equiv 2\mod (p-1)$ and assume $v(a_p)=1$.
\begin{enumerate}
\item [(i)] If $\tau < t$, then $F_0 = 0$.
\item [(ii)] If $\tau = t$, then $F_0$ is a quotient
of 
$\pi(p-1,\lambda,\omega)$, where $\lambda = \overline{\dfrac{2(a_p^2-\binom{r}{2}p^2)}{pa_p(2-r)}} \in \bar\F_p^\times$.
\end{enumerate}
\end{prop}

\begin{proof}
Let $r \geq 2p$. Write $r = 2+n(p-1)p^t$, with $t \geq 0$ and $n \geq 1$, $n$ co-prime to $p$.
We assume that $\tau \leq t$.

Since the functions we will define below are slightly more complicated than others
that have occurred so far, we first define some basic `building block' functions. Let
$$
A = [1,X^{r-1}Y], \quad B = \left[1,\sum_{1<j<r-1 \atop j \equiv 1 \mod (p-1)}\binom{r-1}{j}X^{r-j}Y^j\right] \quad \text{and} \quad C = [1,X^{r-p}Y^p].
$$
We set $A' = A-C$ and  $B' = B+(r-1)C$. Also, for $g \in G$, let
$$
\Phi_g = [g,\sum_{1<j \leq r-1 \atop  j \equiv 1 \mod (p-1)}\binom{r}{j}X^{r-j}Y^j + (r-2)X^{r-p}Y^p] \quad \text{and} \quad
\Psi_g = \sum_{\mu \in \F_p^\times}[\mu]^{-1}\chi_{gg^0_{1,[\mu]}},
$$
where $\chi_g$ is defined in Section~\ref{finfinity}.

Now consider the function $f = f_0 + f_1 + f_\infty \in \mathrm{ind}_{KZ}^G \> \mathrm{Sym}^r \bar\Q_p^2$, with
\begin{eqnarray*}
  f_0 & = & \frac{1-p}{pc}\left(A' + \frac{rp^2}{2a_p^2}B'\right) =
                  \frac{1-p}{pc}\left(A + \frac{rp^2}{2a_p^2}B\right) + \frac{p-1}{a_p}C, \\
  f_1 & = & \frac{-1}{2a_pc}\left(\sum_{\lambda \in \F_p^\times}\Phi_{g^0_{1,[\lambda]}} 
                  + (1-p)\Phi_{g^0_{1,0}}\right), \\ 
  f_\infty & = & \frac{1}{2c}\left(\frac{1}{1-p}\sum_{\lambda \in \F_p^\times} \Psi_{g^0_{1,[\lambda]}} 
                  + \Psi_{g^0_{1,0}}\right).
\end{eqnarray*}
The equivalent expressions for $f_0$ above come from the identity 
$-1+\frac{rp^2}{2a_p^2}(r-1) = -\frac{cp}{a_p}$.

We now compute $(T-a_p)f$ in several simpler steps.

In radius `$-1$', we have $T^-A = O(p^{r-1})$ and $T^-B = O(p^{t+3})$, by 
Lemma~\ref{congruences1}, since $r-j \geq p \geq 5$ and $T^-C =
O(p^{r-p})$. Since $r-1 \geq r-p \geq t+3$, we see  $T^-f_0 = O(p^2)$, 
and so dies mod $p$.
We now compute the radius 0 term $-a_pf_0 + T^-f_1$.
We have $T^-\Phi_g = [g \alpha, 
                       rpXY^{r-1}] +O({p^{t+3}})$, by
Lemma~\ref{congruences2}, 
so, for $\lambda \in \F_p$,
$$
T^-\Phi_{g^0_{1,[\lambda]}} = rp \left[1,\sum_{i=0}^{r-1}\binom{r-1}{i}[\lambda]^{r-1-i}X^{r-i}Y^i \right]
+O({p^{t+3}}),
$$ 
so 
$$
T^-f_1 =
\frac{rp(1-p)}{2a_pc} \left[1,\sum_{1 \leq j<r-1 \atop j \equiv 1 \mod (p-1)}\binom{r-1}{j}X^{r-j}Y^j \right] + O(p^2).
$$
We obtain that $-a_pf_0 + T^-f_1 = -[1,X^{r-1}Y-X^{r-p}Y^p] + O(p)$.

Denote by $h_{\infty,1}$ the part of $(T-a_p)f_\infty$ that lives in
radius $1$, and by $h_{\infty,2}$ the part of $(T-a_p)f_\infty$ that lives in
radius $2$.
Let us now compute the radius 1 term $T^+f_0 - a_pf_1 + h_{\infty,1}$.

We have
\begin{eqnarray*}
T^+A = 
\sum_{\lambda \in \F_p}[g^0_{1,[\lambda]},
-[\lambda]X^r+pX^{r-1}Y], \qquad T^+C=
\sum_{\lambda \in \F_p}[g^0_{1,[\lambda]},-[\lambda]X^r] + O(p^2),
\end{eqnarray*}
and
$$
T^+B = 
\sum_{\lambda \in \F_p^\times} \left[g^0_{1,[\lambda]},
-[\lambda]\sum_{1<j<r-1 \atop j \equiv 1 \mod (p-1)}\binom{r-1}{j}X^r
+p\sum_{1<j<r-1 \atop j \equiv 1 \mod (p-1)}j\binom{r-1}{j}X^{r-1}Y \right]
+O(p^{t+2}),
$$
using the fact that, for $i \geq 2$,
$p^i\sum_{1<j<r-1}\binom{r-1}{j}\binom{j}{i} = O(p^{t+2})$, as follows from
part (3) of Proposition~\ref{congrbinom1} (and Lemma~\ref{congruences1} which
shows that the $j = r-1$ term is of no consequence). 
Using parts (1) and (2) of the same proposition we see that
$\sum_{1<j<r-1}\binom{r-1}{j} = 1-r + np^{t+1} + O(p^{t+2})$
and 
$p\sum_{1<j<r-1} j\binom{r-1}{j} = p^2(r-1)/(1-p) + O(p^{t+2})$. All three sums here are as usual
over $j \equiv 1 \mod (p-1)$.
Substituting, we get
$$
T^+B = 
\sum_{\lambda \in \F_p^\times}[g^0_{1,[\lambda]},
-[\lambda]\left(1-r+np^{t+1}\right)X^r
+\frac{p^2(r-1)}{1-p}X^{r-1}Y] +O(p^{t+2}).
$$
This gives
$$
T^+f_0 =
\frac{1-p}{c}[g^0_{1,0},X^{r-1}Y]
+ \frac{1}{c}\sum_{\lambda \in \F_p^\times}[g^0_{1,[\lambda]},X^{r-1}Y] 
+ h
+ O (p),
$$
where $h$ is an integral linear combination of terms of the form $[g, X^r]$, with $g \in G$.
Rewrite this as
$$
T^+f_0 =
\frac{1-p}{2c} \left[g^0_{1,0},\binom{r}{1}X^{r-1}Y\right]
+ \frac{1}{2c}\sum_{\lambda \in \F_p^\times} \left[g^0_{1,[\lambda]},\binom{r}{1}X^{r-1}Y\right]
+ \frac{2-r}{2c}\sum_{\lambda \in \F_p}[g^0_{1,[\lambda]},X^{r-1}Y]
+ h + O(p).
$$
We also have 
\begin{eqnarray*}
-a_pf_1  & = & 
\frac{1-p}{2c} \left[g^0_{1,0},\sum_{1<j \leq r-1 \atop j \equiv 1 \mod (p-1)}\binom{r}{j}X^{r-j}Y^j \right]
+ \frac{1}{2c}\sum_{\lambda \in \F_p^\times} \left[g^0_{1,[\lambda]},\sum_{1<j \leq r-1 \atop j \equiv 1 \mod (p-1)}\binom{r}{j}X^{r-j}Y^j \right] \\
&   & {\hskip 5cm}  - \frac{2-r}{2c}\sum_{\lambda \in \F_p}[g^0_{1,[\lambda]},X^{r-p}Y^p] + O(p).
\end{eqnarray*}
So 
\begin{multline*}
T^+f_0-a_pf_1 =
\frac{1-p}{2c} \left[g^0_{1,0},\sum_{1 \leq j \leq r-1 \atop j \equiv 1 \mod (p-1)}\binom{r}{j}X^{r-j}Y^j \right]
+ \frac{1}{2c}\sum_{\lambda \in \F_p^\times}
\left[g^0_{1,[\lambda]},\sum_{1 \leq j \leq r-1 \atop j \equiv 1 \mod (p-1)}\binom{r}{j}X^{r-j}Y^j \right] \\
+ \frac{2-r}{2c}\sum_{\lambda \in \F_p}[g^0_{1,[\lambda]},X^{r-1}Y-X^{r-p}Y^p]
+ h + O(p).
\end{multline*}
We now observe, by Lemma~\ref{defauxfunc}, that
$$
(T-a_p)\Psi_g =
(p-1) \left[g,\sum_{1 \leq j \leq r-1 \atop j \equiv 1 \mod (p-1)} \binom{r}{j}X^{r-j}Y^j \right] 
+ a_p\sum_{\mu \in \F_p^\times}[\mu]^{-1}[gg^0_{1,[\mu]},X^{r-2}Y^2]
+ O(p^{\tau+1}),
$$
since $t_0 = \tau$. We deduce that 
\begin{multline*}
(T-a_p)f_\infty = 
\frac{-1}{2c}\left((1-p) \left[g^0_{1,0},\sum_{1 \leq j \leq r-1 \atop j \equiv 1 \mod (p-1)} \binom{r}{j}X^{r-j}Y^j\right]
+ \sum_{\lambda \in \F_p^\times} \left[g^0_{1,[\lambda]},\sum_{1 \leq j \leq r-1 \atop j \equiv 1 \mod (p-1)} \binom{r}{j}X^{r-j}Y^j \right]\right) \\
+
\frac{a_p}{2c}\left(\sum_{\mu \in \F_p^\times}[\mu]^{-1}[g^0_{1,0}g^0_{1,[\mu]},X^{r-2}Y^2]
+
\frac{1}{1-p}\sum_{\lambda\in\F_p^\times}\sum_{\mu \in \F_p^\times}[\mu]^{-1}[g^0_{1,[\lambda]}g^0_{1,[\mu]},X^{r-2}Y^2]\right)
+ O(p).
\end{multline*}
Putting everything together, in radius 1 we get that
$$
T^+f_0 -a_pf_1 + h_{\infty,1} = 
\frac{2-r}{2c}\sum_{\lambda \in \F_p}[g^0_{1,[\lambda]},X^{r-1}Y-X^{r-p}Y^p]
+ h 
+ O(p).
$$

Finally, let us compute the radius 2 term $T^+f_1 + h_{\infty,2}$. 
We have
$$
T^+\Phi_g
= \sum_{\mu \in \F_p} \left[gg^0_{1,[\mu]},\sum_{i\geq
0}p^i\sum_{1<j \leq r-1 \atop j \equiv 1 \mod (p-1)}(-[\mu])^{j-i}\binom{j}{i}\binom{r}{j}X^{r-i}Y^i +
(r-2)\sum_{i=0}^p(-[\mu])^{p-i}p^i\binom{p}{i}X^{r-i}Y^i \right].
$$
Using all parts of Proposition \ref{congrbinom12}, and Lemma~\ref{congruences2} for the $\mu = 0$ terms, we get 
$$
T^+\Phi_g
= \sum_{\mu\in \F_p^\times}\frac{[\mu]^{-1}}{1-p} \left[gg^0_{1,[\mu]},p^2\binom{r}{2}X^{r-2}Y^2 \right]
+ p^{t+1} h' + O(p^{t+2}),
$$
where
$h'$ is an integral linear combination of terms of the form
$[g,X^r]$, for $g\in G$.
So
$$
\frac{-1}{2a_pc}T^+\Phi_g =
\frac{-p^2\binom{r}{2}}{2a_pc(1-p)}
\sum_{\mu \in \F_p^\times}[\mu]^{-1}[gg^0_{1,[\mu]},X^{r-2}Y^2] + h''
+O(p),
$$
where $h''$ is an integral linear combination of terms of the form $[g, X^r]$, for $g \in G$.
Notice that $\frac{-p^2\binom{r}{2}}{2a_pc(1-p)} = -\frac{a_p}{2c(1-p)} + O(p)$, so finally we get
$$
T^+f_1 =
\frac{-a_p}{2c}
\left(\frac{1}{1-p}\sum_{\lambda \in\F_p^\times}
\sum_{\mu \in \F_p^\times}[\mu]^{-1}[g^0_{1,[\lambda]}g^0_{1,[\mu]},X^{r-2}Y^2]
+ 
\sum_{\mu\in \F_p^\times}[\mu]^{-1}[g^0_{1,0}g^0_{1,[\mu]},X^{r-2}Y^2]\right) + h'''
+O(p),
$$
where $h'''$ is an integral linear combination of terms of the form $[g, X^r]$, for $g \in G$.
Thus in radius 2, we get that $T^+f_1 + h_{\infty,2} = h''' + O(p)$. 

So, putting everything together, we see that
\begin{equation}
\label{october}
  (T-a_p)f=\frac{2-r}{2c}\sum_{\lambda \in \F_p}[g^0_{1,[\lambda]},\theta X^{r-p-1}] - [1,\theta X^{r-p-1}]+h''''+O(p),
\end{equation}
where $h''''$ is an integral linear combination of terms of the form $[g,X^r]$, 
for $g\in G$. 

Now the term $h''''$ above dies in $\mathrm{ind}_{KZ}^G P$.
Moreover, by Lemma~\ref{generator} (i),   $\theta X^{r-p-1}$ 
 can be identified with $X^{p-1}  \in J_0$. So, 
 we see that the image of $(T-a_p)f$ in $\mathrm{ind}_{KZ}^G P$ lies in  $\mathrm{ind}_{KZ}^G J_0$ and is 
$$
\overline{\frac{2-r}{2c}} \sum_{\lambda \in \F_p}[g^0_{1,[\lambda]},X^{p-1}] - [1,X^{p-1}].
$$
This immediately implies both parts of the proposition.
\end{proof}

\subsection{Conclusion}
\label{enda2}

\begin{proof}[Proof of Theorem \ref{thma2}] We apply
Proposition~\ref{F2a2}, Proposition~\ref{F1a2} and Proposition \ref{F0a2}
to find the possible structures of $F_2$, $F_1$ and $F_0$ respectively.
In each case we obtain a unique possibility for the semisimplification of
$\bar\Theta$ and we conclude using the (semisimple) mod $p$  Local
Langlands Correspondence (LLC). 

\begin{enumerate}[(i)]

\item 
If $v(a_p^2-\binom{r}{2}p^2) = v(r-2) + 2$, then
let $\lambda$ be the reduction mod $p$ of
$\frac{2(a_p^2-\binom{r}{2}p^2)}{pa_p(2-r)}$, which is in $\bar\F_p^\times$.
Then $F_2$ is a quotient of the principal series representation
$\pi(p-3,\lambda^{-1},\omega^2)$, $F_1=0$ if $p\mid r$ and is otherwise a
quotient of $\pi(0, \nu,\omega)$, with $\nu = \overline{\frac{rp}{2a_p}} \in \bar\F_p^\times$, and $F_0$ is a quotient of
$\pi(p-1,\lambda,\omega)$.  Thus none of $F_0$, $F_1$ or $F_2$ can have a
supersingular JH factor, and $\bar\Theta$ is forced to be reducible with
semisimplification
$\pi(p-1,\lambda,\omega)^{ss}\oplus\pi(p-3,\lambda^{-1},\omega^2)$.  

\item 
If $v(a_p^2-\binom{r}{2}p^2) > v(r-2) + 2$, 
then  $F_2=0$, $F_1=0$ if $p\mid r$ and is otherwise a
quotient of $\pi(0, \nu,\omega)$. 
Again the mod $p$ LLC forces $F_1=0$, and hence
$\bar\Theta\cong F_0\cong\pi(p-1,0,\omega)$. 

\item
If $v(a_p^2-\binom{r}{2}p^2) < v(r-2) + 2$, then necessarily $r \equiv 2 \mod p$, so 
$F_1$ is a proper quotient of $\pi(0, \overline{\frac{p}{a_p}},\omega)$, and $F_0=0$. 
Since $\bar\Theta$ lies in the image of the mod $p$ LLC, 
we must have  $F_1=0$, and hence $\bar\Theta\cong F_2\cong\pi(p-3,0,\omega^2)$.
\end{enumerate}
\vspace{-0.6cm}
\end{proof}

\begin{remark}\label{remark a=2}
In part (i), clearly $F_2 \neq 0$. 
Also, if $\lambda \neq \overline{\frac{rp}{2a_p}}$, then $F_0 = \pi(p-1, \lambda, \omega)$ and  $F_1 = 0$, as we will see in Proposition \ref{F1a2bis}.  
If $\lambda =\overline{ \frac{rp}{2a_p}}$, then
    $\lambda = \pm 1$. 
In this case, an additional possibility  in part (i) is that 
   $F_0 = \mathrm{St} \otimes (\mu_\lambda \cdot \omega \circ \det)$ and 
        $F_1 = \mu_\lambda \cdot \omega \circ \det \neq 0$, by Proposition~\ref{F1a2} (ii). 
This might explain
       why it does not seem possible to show that $F_1 = 0$ uniformly in all three parts.
\end{remark}

\section{Reduction without semisimplification} \label{peu-tres}

In this section we investigate more subtle properties of the reduction $\bar{V}_{k,a_p}$. 
In particular, if the reduction is reducible and non-split, we wish to describe what type
of extension one obtains. 
We shall assume throughout that $p \geq 5$, except for the last subsection where we
allow $p = 3$.

\subsection{Results}

In earlier sections we computed the semisimplification
$\bar{V}_{k,a_p}^{ss}$ of $\bar V_{k,a_p}$, using a particular lattice.
In \cite{Col}, 
Colmez defined functors (see Section \ref{functor} for more details)
that give a correspondence between certain $G$-stable lattices in
$\Pi_{k,a_p}$ or its $p$-adic completion $\hat\Pi_{k,a_p}$ with respect
to the lattice $\Theta_{k,a_p}$, and $G_{\Q_p}$-stable lattices
in $V_{k,a_p}$, in a way that is compatible with reduction modulo $p$. 
So by choosing a suitable lattice $\Theta'$ inside $\hat\Pi_{k,a_p}$ we get a
lattice inside $V_{k,a_p}$, which we denote by $V(\Theta')$, and the
reduction of $V(\Theta')$ can be computed from $\bar\Theta'$, 
it is simply $V(\bar\Theta')$.
%

In this section we apply this in order to primarily investigate the following question.  
Assume that $\overline{V}_{k,a_p}^{ss}$ is isomorphic to
$\omega \oplus 1$ (up to a twist).
It is well known, by work of Ribet, that there is a
lattice inside $V_{k,a_p}$  that reduces to a non-split extension of $1$ by
$\omega$ (up to a twist). One can  ask whether this extension is ``peu ramifi\'ee" or
``tr\`es ramifi\'ee" in the sense of \cite{Serre87}. 
The answer does not depend on the choice of the lattice, which is unique up to homothety. In fact, 
the extension is given by a (non-zero) class in 
$\mathrm{H}^1(G_{\Q_p}, \bar\F_p(\omega)) = \Q_p^\times / (\Q_p^\times)^p \otimes \bar\F_p$,
which has dimension 2 over $\bar\F_p$, and changing the basis used to
define this cohomology class only changes the 
class by a (non-zero) constant. The extensions corresponding to the (non-zero) elements in the line 
$\Z_p^\times/(\Z_p^\times)^p \otimes \bar\F_p$ are called  ``peu ramifi\'ee'', whereas the extensions
corresponding to the remaining (non-zero) classes are (by definition) called ``tr\`es ramifi\'ee''.  
In the context of the results proved so far in this paper (see Sections \ref{case3ap-1} through \ref{casea2} and Theorem~\ref{maintheoslopeone}), this 
question arises in exactly 
two cases:
\begin{enumerate}
\item $a=p-1$, $p$ does not divide $r+1$ and $\lambda = \overline{\frac{a_p}{p(r+1)}} = \pm 1$, so $\overline{\frac{a_p}{p}} = \pm \overline{(r+1)}$.
\item $a=2$, $v(a_p^2-\binom{r}{2}p^2) = 2 + v(r-2)$ and 
$\overline{u(a_p)} = \bar\eps$, with $\eps = \pm1$, so
$\overline{\frac{a_p}{p}} = \eps \overline{\frac{r}{2}}$ or $\eps\overline{(1-r)}$,
where 
$$u(a_p):={\frac{2(a_p^2-{r\choose2}p^2)}{(2-r)pa_p}}$$
is a $p$-adic unit.
\end{enumerate}
In these cases, the standard lattice $\Theta = \Theta_{k,a_p}$  studied in 
Section~\ref{case3ap-1} (when $a=p-1$) and Section~\ref{casea2} (when $a = 2$) 
does not always allow us to
answer the question. For instance, in the first case 
$V(\bar\Theta)$ is a twist of an 
extension of $\omega$ by $1$,
and in the second case 
$V(\bar\Theta)$ is a split extension when $F_1 =0$ (see Corollary \ref{computedsplit}).

We can however, study the reduction of certain non-standard lattices, and prove the following result.
To simplify the statement, we introduce the following
terminology. Let $\eps \in \{\pm 1\}$. Assume we are in case (2) above, so that $r\equiv 2\mod (p-1)$.  
The question of  ``peu'' or ``tr\`es ramifi\'ee'' arises when 
$$\dfrac{a_p}{p} \in \left\{\alpha\in\bar\Q_p \> \Big| \> v(\alpha) =
0, \>  v\left(\alpha-\binom{r}{2}\alpha^{-1}\right) = v(r-2), \> 
\overline{\frac{2(\alpha-\binom{r}{2}\alpha^{-1})}{2-r}} =  \bar\eps\right\}.$$
Assume further that $r \not\equiv 2/3 \mod p$. It's easy to see then that this (big) set is a disjoint union 
of two sets 
$$\left\{\alpha\in\bar\Q_p: v(\alpha) = 0, v(\alpha-\eps r/2) > v(r-2)\right\} \,
\sqcup\, \left\{\alpha\in\bar\Q_p: v(\alpha) = 0, v(\alpha-\eps(1-r)) > v(r-2)\right\}.$$ 
The first set is empty if $r \equiv 0 \mod p$ and the second is
empty if $r \equiv 1 \mod p$. 
\begin{definition}
We say that `$a_p$ is close to $\eps pr/2$'
if $a_p/p$ is in the first set, and that `$a_p$ is close to $\eps
p(1-r)$' if $a_p/p$ is in the second set. 
\end{definition}
\noindent For $a_p/p$ in the big set above, 
 $a_p$ is close to $\eps pr/2$ if and only if  $\overline{a_p/p} = \overline{\eps r/2}$, and similarly,  $a_p$ is close to
$\eps p(1-r)$ if and only if $\overline{a_p/p} = \overline{\eps (1-r)}$.

\begin{thm}[Theorems \ref{ap-1ramifiethm}, \ref{a2peuramifiethm} and \ref{a2ramifiethm}]
\label{peuram}
Let $p \geq 5$, let $r \geq p+1$ and assume that $v(a_p)=1$.

\begin{enumerate}

\item
Suppose $r \equiv p-1 \mod (p-1)$, $p$ does not divide $r+1$ and
$\overline{{\frac{a_p}{p(r+1)}}} = \bar\eps$, with $\eps \in\{ \pm 1\}$.
Then there exists a lattice $L$ 
in $V_{k,a_p}$ such that
$\bar{L}$ is a non-split, ``peu ramifi\'ee" extension of
$\mu_{\eps} $ by $\mu_{\eps} \cdot \omega$.

\item 
Suppose that $r\equiv 2\mod (p-1)$  and $r \not\equiv 2/3 \mod p$. 
\begin{enumerate}
\item 
Suppose $a_p$ is close to $\eps pr/2$, for some $\eps\in\{\pm 1\}$.
Then there exists a lattice $L$ 
in $V_{k,a_p}$ such that
$\bar{L}$ is a non-split, ``peu ramifi\'ee" extension of
$\mu_{\eps} \cdot \omega$ by $\mu_{\eps} \cdot \omega^2$.

\item 
Suppose that $a_p$ is close to $\eps p(1-r)$, for some $\eps\in\{\pm1\}$.
If $r\equiv 2 \mod p$, further assume that $u(a_p)-\eps$ is a
uniformizer of $E=\Q_p(a_p)$ or that $E$ is unramified over $\Q_p$.
\begin{enumerate}
\item If $v(u(a_p)-\eps) < 1$, then 
there exists a lattice $L$ 
in $V_{k,a_p}$ such that
$\bar{L}$ is a non-split, ``peu ramifi\'ee" extension of
$\mu_{\eps} \cdot \omega$ by $\mu_{\eps} \cdot \omega^2$.

\item If $v(u(a_p)-\eps) \geq 1$, then
there exists a lattice $L$ 
in $V_{k,a_p}$ such that
$\bar{L}$ is a non-split, ``tr\`es ramifi\'ee" extension of
$\mu_{\eps} \cdot \omega$ by $\mu_{\eps} \cdot \omega^2$.

\item Moreover, 
depending on the choice of
$a_p$ satisfying $\overline{a_p/p} = \overline{\eps(1-r)}$, all isomorphism
classes of ``tr\`es ramifi\'ee" extensions appear.
\end{enumerate}
\end{enumerate}
\end{enumerate}
\end{thm}
\noindent The condition  $r\not\equiv2/3\mod p$ in part (2) of
the theorem ensures that 
once we fix $\eps$,
cases (a) and (b) above are mutually exclusive.
We do not know 
what happens in case $r \equiv 2/3 \mod p$, although
this is also a case where the distinction between ``peu" and ``tr\`es
ramifi\'ee" could be made.
The technical conditions imposed  
 when $r\equiv 2 \mod p$ in part (2) (b) are used
in the proof of Proposition \ref{S1S2} (see also Remark
\ref{F1zerorembis}) and could possibly be removed if one finds a
 more direct proof of Proposition \ref{F1a2bis}.

\subsection{Colmez's Montr\'eal functor}
\label{functor}

Let $E$ be a finite extension of $\Q_p$, with ring of integers
$\mathcal{O}_E$ and residue field $k_E$.

We introduce the categories on which Colmez's functors are defined (see
\cite[Intro., 4.] {Col}).
We denote by $\Reptors(G)$ the category of $\O_E[G]$-modules that are
smooth, admissible, of finite length and have a central character. We
denote by $\RepOE(G)$ the category of $\O_E[G]$-modules $\Theta$ that are complete
and separated for the $p$-adic topology, without $p$-torsion and such
that for all $n$, $\Theta/p^n\Theta$ is in $\Reptors(G)$. Finally we
denote by $\RepE(G)$ the category of $E[G]$-modules $\Pi$ with a lattice
which is in $\RepOE(G)$. Hence objects of $\RepE(G)$ are $p$-adic Banach
spaces, and the lattice as in the definition defines the topology of the
Banach space.

Let $\Pi$ be a $p$-adic Banach space. Two $\O_E$-lattices $\Theta$
and $\Theta'$ in $\Pi$ are said to be commensurable if there exists $n$,
$m$ such that $p^n\Theta \subset \Theta' \subset p^m\Theta$. In this case
$\Theta$ defines the topology of $\Pi$ if and only if $\Theta'$ does.
Let $\Pi$ be in $\RepE(G)$ and $\Theta\subset\Pi$ be a lattice which is
in $\RepOE(G)$, then any lattice $\Theta'$ that is commensurable with
$\Theta$ and stable by $G$ is also in $\RepOE(G)$.

In \cite{Col}, Colmez has defined exact functors, which we all denote by
$V$, that go from these categories to categories of continuous representations of
$G_{\Q_p}$ with coefficients in $E$ or $\O_E$. More precisely $V$ goes
from $\Reptors(G)$ to torsion $\O_E$-representations of finite length,
from $\RepOE(G)$ to representations on free $\O_E$-modules of finite
rank, and from $\RepE(G)$ to finite-dimensional $E$-representations.

These various functors are compatible: if $\Pi\in \RepE(G)$ and $\Theta$
is a lattice in $\Pi$ as in the definition, then $V(\Theta)$ is a lattice
in $V(\Pi)$, and $V(\bar\Theta) = \overline{V(\Theta)}$
(see \cite[Th\'eoreme IV.2.14]{Col} for the properties of these functors).
Moreover $V(\hat\Pi_{k,a_p}) = V_{k,a_p}$, and $V(\bar\Theta_{k,a_p})^{ss} =
\bar{V}_{k,a_p}^{ss}$. We will always consider lattices in
$\hat\Pi_{k,a_p}$ that are commensurable with $\hat\Theta_{k,a_p}$, and
all the lattices $L \subset V_{k,a_p}$ as in Theorem \ref{peuram}
are of the form $V(\Theta')$ for
such a lattice $\Theta' \subset \hat\Pi_{k,a_p}$.

\subsection{Preliminaries in characteristic $p$}

We fix a finite field $k_E$ of characteristic $p$.
Let $\St$ be the Steinberg representation with coefficients in $k_E$. 
The definition of $\St$ is recalled in \S \ref{defmu} below.

\subsubsection{Extensions of $\pi(p-3,1,\omega)$ by $\St$}

\begin{prop}[Proposition VII.4.22 of \cite{Col}]
\label{split}
Any extension of $\pi(p-3,1,\omega)$ by $\St$ is split.
\end{prop}

\begin{cor}
\label{computedsplit}
If $a = 2$, $\overline{u(a_p)} = \pm 1$ 
and $F_1 = 0$, then $V(\bar{\Theta})$ is a
split extension of $\mu_{\overline{u(a_p)}} \cdot \omega$ by
$\mu_{\overline{u(a_p)}} \cdot \omega^2$.
\end{cor}

\begin{proof} Let us assume for simplicity that $\overline{u(a_p)} =  1$.
It follows from the computations of Section 6 that in this case
$\bar{\Theta}$ is an extension of $\pi(p-3,1,\omega^2)$ by
$\pi(p-1,1,\omega)$.
Hence $\bar{\Theta}$ contains a line $L$ on
which $G$ acts by $\omega$, and the quotient is an extension of
$\pi(p-3,1,\omega^2)$ by $\St\otimes\,\omega$, hence is split. 
The functor $V$ takes the line $L$  to zero (\cite[Th\'eor\`eme 0.10]{Col}), 
and so $V(\bar{\Theta}) = V(\bar{\Theta}/L)$ is split.
\end{proof}

\begin{remark}
The conditions of Corollary \ref{computedsplit} are fulfilled for example 
when $\overline{a_p/p} = \pm (1-r)$ and $r$ is not $1$ or $2$ or $2/3$
modulo $p$.
\end{remark}

\begin{remark}
Note that the proof of Proposition \ref{split} uses as an
intermediate step the fact that if $\Pi$ is an extension of
$\pi(p-3,1,\omega)$ by $\St$, then $V(\Pi)$ is a split extension of $1$
by $\omega$.
\end{remark}

\subsubsection{Linear forms on $\St$}
\label{defmu}

Let $\mathcal{C}^0(\P^1(\Q_p),k_E)$ be the smooth $G$-representation of
locally constant functions on $\P^1(\Q_p)$ with values in $k_E$. Then 
the Steinberg representation $\St$ 
is the quotient $\mathcal{C}^0(\P^1(\Q_p),k_E)/k_E$ 
of this space by the space of constant functions.

Let $A$ be the subgroup of $G$ that is isomorphic to $\Q_p^\times$, given by
the set of
$\matr a001 \in G$, $a\in \Q_p^\times$, and let $i : \Q_p^\times \to A$, $a \mapsto 
\matr a001$. Let $w \in G$ be the element $\matr 0110$. 

Let $I(1) \subset G$ be the subgroup of $K$ of matrices that are
upper triangular and unipotent modulo $p$.
It is well known that $\St^{I(1)}$ is a line. 
If we view $\St$ as $\mathcal{C}^0(\P^1(\Q_p),k_E)/k_E$, then 
$\St^{I(1)}$ is generated by the image of the
characteristic function of $\Z_p$.

\begin{lemma}
\label{linearSt}
There exists a non-zero linear form $\mu : \St \to k_E$ which is invariant under
the action of $A$. The form $\mu$ is unique up to multiplication by a
scalar. Moreover, $\mu \circ w = -\mu$, and $\mu$ is non-zero on
$\St^{I(1)}$.
\end{lemma}

\begin{proof}
The existence and uniqueness of a non-zero $\mu$ which 
is invariant under the action of $A$
come from \cite[Lemme VII.4.16]{Col}.
Colmez also gives a description of such a linear form:
it is a multiple of the map
$f \mapsto f(0) -f(\infty)$, for $f \in \mathcal{C}^0(\P^1(\Q_p),k_E)/k_E$,
so the other two properties are clear.
\end{proof}

In fact we will use the following variant:
\begin{lemma}
\label{linearStomega}
Let $m\in \Z$.
There exists a non-zero linear form $\mu : \St\otimes\,\omega^m \to \omega^m$ which is
$A$-equivariant. It is unique up to multiplication by a scalar. It
satisfies $\mu \circ w = (-1)^{m+1}\mu$, and is non-zero on
$(\St\otimes\,\omega^m)^{I(1)}$.
\end{lemma}

\subsubsection{Extensions of $1$ by $\St$}
\label{tau}

Let $H = \Hom_{cont}(\Q_p^\times,k_E)$. 
In \cite[Para. VII.4.4]{Col}, Colmez attaches to any $\tau$ in $H$ a representation
$E_{\tau}$ of $\GL_2(\Q_p)$ with coefficients in $k_E$ 
which is an extension of
$1$ by the Steinberg representation $\St$ 
and which is non-split if and only if 
$\tau$ is non-zero.
(Note that in what follows we take $E_{\tau}$ to be the representation
that Colmez calls $E_{-\tau}$.)
Any object of $\Reptors(G)$ which is a 
non-split extension of $1$ by $\St$ with coefficients in $k_E$
is isomorphic to some $E_{\tau}$
(Th\'eor\`eme VII.4.18 in \cite{Col}).

Note that $\Q_p^\times = \mu_{p-1} \times (1+p\Z_p) \times p^{\Z}$.
As $k_E$ is killed by $p$, we see that for all $\tau\in H$,
$\tau(\mu_{p-1}) = 0$. Hence $\tau$ is entirely determined by
$\tau(1+p)$ and $\tau(p)$. Let $[\tau] \in \P^1(k_E)$ be the class of
$(\tau(1+p),\tau(p))$, when the extension is not split. 
Then the isomorphism class of $E_{\tau}$ depends
only on $[\tau]$. 

Let $\Pi$ be an extension of $1$ by $\St$. Here is a method for computing
$[\tau]$, following \cite[VII.4.4]{Col}:

\begin{lemma}
\label{computetau}
Let $\Pi$ be an extension of $1$ by $\St$. 
Let $\mu$ be as in Lemma \ref{linearSt} and $e$ an element of $\Pi$ that
is not contained in the subrepresentation $\St$.
Then either $\mu(i(1+p)e-e) = \mu(i(p)e-e) = 0$ and  the extension is
split, or $\Pi$ is
isomorphic to $E_{\tau}$, for $\tau$ such that
$[\tau] = (\mu(i(1+p)e-e):\mu(i(p)e-e))$.
\end{lemma} 

Note that for all $x\in \Pi$ and $g\in G$, $gx-x$ is in the
subrepresentation of $\Pi$ isomorphic to the Steinberg representation, 
so we can indeed apply $\mu$ to the given elements.

We use also the following variant:
\begin{lemma}
\label{computetauomega}
Let $\Pi$ be an extension of $\omega^m$ by $\St\otimes\,\omega^m$, for some
$m\in\Z$. 
Let $\mu$ be as in Lemma \ref{linearStomega} and $e$ an element of $\Pi$ that
is not contained in the subrepresentation $\St \otimes \omega^m$.
Then either $\mu(i(1+p)e-e) = \mu(i(p)e-e) = 0$ and  the extension is
split, or $\Pi$ is
isomorphic to $E_{\tau}\otimes\omega^m$, for $\tau$ such that
$[\tau] = (\mu(i(1+p)e-e):\mu(i(p)e-e))$.
\end{lemma} 

\subsubsection{Criterion for the extension to be ``peu ramifi\'ee"}

Colmez shows that if $\tau \neq 0$, then
there exists a unique non-split extension class of 
$\pi(p-3,1,\omega)$ by $E_{\tau}$ which we denote by $\Pi_{\tau}$ 
(\cite[Proposition VII.4.25]{Col}). 

Then Colmez's functor attaches to $\Pi_{\tau}$ a Galois representation
$V(\Pi_{\tau})$ which is a non-split extension of 
$1$ by $\omega$ (\cite[Proposition VII.4.24]{Col}).
The description of $V(\Pi_{\tau})$  in \cite[Para. VII.3]{Col} 
gives:

\begin{prop}
$V(\Pi_{\tau})$ is a 
non-split extension of $1$ by $\omega$ and it is a
``peu ramifi\'ee" extension if and only if
$\tau$ is zero on $\Z_p^\times$, that is, $[\tau] = (0:1)$.
\end{prop}

Hence:

\begin{prop}
\label{criterion}
Suppose that $\Pi$ is a non-split extension of 
$\pi(p-3,1,\omega^{m+1})$ by $E_{\tau}\otimes\omega^m$ for some $m\in\Z$ and
some non-zero $\tau$. Then $V(\Pi)$ is a non-split extension of $\omega^m$ by
$\omega^{m+1}$, which is ``peu ramifi\'ee" if $[\tau] = (0:1)$ and ``tr\`es
ramifi\'ee" otherwise.
\end{prop}

\subsubsection{A special extension}

The representation $\pi(0,1,1)$ is a non-split extension of $1$ by
$\St$. We want to know to which $\tau \in H$ it corresponds.

\begin{prop}
\label{pi011}
The representation $\pi(0,1,1)$ is isomorphic to the representation
$E_{\tau}$,  for $\tau\in H$ defined by $\tau(\Z_p^\times) = 0$ and
$\tau(p)=1$, that is, $[\tau] = (0:1)$.
\end{prop}

\begin{proof}
This is the first part of the statement of \cite[Lemma 6.18]{Pas}.
\end{proof}

\subsubsection{A preliminary lemma}

\begin{lemma}
\label{quotientlemma}
Let $F$ be a quotient of $\mathrm{ind}_{KZ}^G V_0$, with $F^{ss} \cong \St\oplus 1$. 
Then $F \cong \pi(0,1,1)$.
\end{lemma}
    
\begin{proof}
The Steinberg representation $\St$ is known not to be a
quotient of $\mathrm{ind}_{KZ}^G V_0$. So $F$ is a non-split extension of the trivial
representation by the Steinberg representation, hence is isomorphic to
$E_\tau$ for some $\tau: \Q_p^\times \to k$.

As before, let $I(1) \subset G$ be the subgroup of $K$ of matrices that
are upper triangular and unipotent modulo $p$. If 
$\mathrm{ind}_{KZ}^G V_0 \to 1$ is a non-zero $G$-equivariant map, then it induces a surjection
$(\mathrm{ind}_{KZ}^G V_0)^{I(1)} \to 1$. Indeed, $(\mathrm{ind}_{KZ}^G
V_0)^{I(1)}$ contains the line generated by $[1,1]$, which generates
$\mathrm{ind}_{KZ}^G V_0$ over $G$, hence has non-zero image by any non-zero
$G$-equivariant map.

Consider now the map $\mathrm{ind}_{KZ}^G V_0 \to 1$ given by the
composition of the maps $\mathrm{ind}_{KZ}^G V_0 \to F$ and $F \to 1$.
Taking $I(1)$-invariants, we get a map $(\mathrm{ind}_{KZ}^G V_0)^{I(1)} \to 1$ that factors
through $F^{I(1)}$, so in particular the map $F^{I(1)} \to 1$ is surjective. 
It follows
from the proof of \cite[Lemma 6.18]{Pas} that the natural map
$(E_\tau)^{I(1)} \to 1$ is surjective if and only if $[\tau] = (0:1)$. So
$F$ is isomorphic to 
$\pi(0,1,1)$ by Proposition
\ref{pi011}.  
\end{proof}

\subsection{A lemma on changing lattices}

We denote by $R = \O_E$ the ring of integers of the finite extension $E$
of $\Q_p$.
The following lemma is very
similar to Proposition VII 4.5 of \cite{Col}, but it is important for us
that $M$ is allowed to be reducible. 

\begin{lemma}
\label{changelattice}
Let $\Theta \in \RepOE(G)$ and $\Pi = \Theta \otimes_{\O_E} E \in
\RepE(G)$, with $\Pi$ irreducible.
Suppose that $\bar\Theta$ has a subrepresentation $M$ that is
indecomposable. Then there exists an $\O_E[G]$-lattice $\Theta'$ in $\Pi$, commensurable to
$\Theta$, such that $\bar\Theta'$ is indecomposable and contains $M$ as a
subrepresentation.
\end{lemma}

\begin{proof}
The second paragraph of the proof of 
Proposition VII 4.5 of \cite{Col} applies without change, with $M$
playing the role of $W_1$: note that at this point of the proof the fact
that $W_1$ is irreducible does not play a role anymore, only the fact
that it is indecomposable.
\end{proof}

\begin{cor}
\label{peuortres}
Let $\Theta \subset \hat{\Pi}_{k,a_p}$ be an $R[G]$-lattice commensurable
to $\hat{\Theta}_{k,a_p}$. Suppose that $\bar\Theta$ is an extension of 
$\pi(p-3,1,\omega^{m+1})$ by $M$, where $M$ is an extension of
$\omega^m$ by $\St\otimes\,\omega^m$, for some $m\in \Z$. 

If $M$ is non-split and 
isomorphic to $E_{\tau}\otimes \omega^m$ for $[\tau] = (0:1)$,
then there exists a lattice $\Theta'$ commensurable to $\Theta$ such that
$V(\bar\Theta')$ is a non-split, ``peu ramifi\'ee" extension of $\omega^m$
by $\omega^{m+1}$.

If $M$ is non-split and 
isomorphic to $E_{\tau}\otimes \omega^m$ for $[\tau] = (1:x)$, for some
$x\in k_E$, 
then there exists a lattice $\Theta'$ commensurable to $\Theta$ such that
$V(\bar\Theta')$ is a non-split, ``tr\`es ramifi\'ee" extension of $\omega^m$
by $\omega^{m+1}$.
\end{cor}

\begin{proof}
It suffices to apply Lemma \ref{changelattice} to $\Theta$, $M$ and then
apply Proposition \ref{criterion}.
\end{proof}

\subsection{The case $a=p-1$ and $r \geq 2p-2$ and $p$ does not divide $r+1$ and
$\overline{\frac{a_p}{p(r+1)}} = \pm 1$}

Let $\eps = \overline{\frac{a_p}{p(r+1)}}$, so that $\eps = \pm 1$.

\begin{thm}
\label{ap-1ramifiethm}
Let $p \geq 5$ and $r \geq 2p-2$. 
Suppose $p-1$ divides $r$, $p$ does not divide $r+1$ and
$\overline{a_p/p(r+1)} =  \pm 1$.
Then, there exists a $\GL_2(\Q_p)$-stable lattice $\Theta'$ 
in $\Pi_{k,a_p}$ such that
$V(\bar{\Theta}')$ is a non-split, ``peu ramifi\'ee" extension of
$\mu_{\overline{a_p/p(r+1)}}$ by $\mu_{\overline{a_p/p(r+1)}} \cdot \omega$.
\end{thm}

\begin{proof}
Without loss of generality, assume that $\overline{a_p/p(r+1)} = 1$.
Let $\Theta$ be the standard lattice, so it is an $\O_E$ module for $E = \Q_p(a_p)$.  
We proved in Section
\ref{case3ap-1} (see Propositions~\ref{Theta} and \ref{red}, with $a = p-1$, $\bar u = 1$, 
and the proof of Theorem~\ref{3ap-1}) that the reduction of $\Theta$ is an extension
$$
0 \to F_0 \to \bar\Theta \to F_2 \to 0,
$$
where $F_0$ is isomorphic to $\pi(p-3,1,\omega)$
and $F_2$ is isomorphic to $\pi(0,1,1)$.
Let $\pi$ be the projection $\Theta \to \bar\Theta$ and $\varpi$ be a
uniformizer of $E$. We define a new $G$-invariant lattice by
setting $\Theta'' = \Theta + (1/\varpi)\pi^{-1}(F_0)$. Then the inclusion
$\Theta \to \Theta''$ induces a map $\bar\Theta \to \bar\Theta''$, and the
kernel of this map is $F_0$. 
So $\bar\Theta''$ has a subrepresentation isomorphic to
$F_2$, that is, $\pi(0,1,1)$, and the quotient is necessarily
isomorphic to $\pi(p-3,1,\omega)$. So we can apply Corollary
\ref{peuortres} to the lattice $\Theta''$, to deduce the existence of 
the lattice $\Theta'$ with the desired properties.
\end{proof}

\begin{remark}
Actually in this case the lattice $\Theta'$ can be made entirely explicit. 
Consider the lattice $\mathcal{V}_r = \Sym^r\Zbar_p^2 + \eta\Sym^{r-(p+1)}\Zbar_p^2$ 
in $\Sym^r\Qbar_p^2$, where $\eta = \theta/p$.
This lattice is stable under the action of $K$,
as for all $\gamma\in K$,
we have $\gamma \eta \in \Zbar_p\eta + \Sym^{p+1}\Zbar_p^2$.
Let $\Theta'$ be the image of $\mathrm{ind}_{KZ}^G \mathcal{V}_r$ inside
$\mathrm{ind}_{KZ}^G \Sym^r\Qbar_p^2/((T-a_p)\mathrm{ind}_{KZ}^G\Sym^r\Qbar_p^2)$
and let $\bar{\Theta}'$ be its reduction modulo 
$\mathfrak{m}_{\Qbar_p}$. Then easy but tedious computations show that
$V(\bar{\Theta}')$ is a non-split, ``peu ramifi\'ee" extension of
$\mu_{\eps}$ by $\mu_{\eps} \cdot \omega$.
\end{remark}

\subsection{The case $a=2$ and $a_p$ is close to $\eps pr/2$ and $r$ is not $2/3 \mod p$}
\label{a2r2}

Let $E = \Q_p(a_p)$, let $R = \O_E$ and let $\m_E$ be its maximal ideal, with
uniformizer $\varpi$.


\subsubsection{Statement}

\begin{thm}
\label{a2peuramifiethm}
Let $p \geq 5$ and $r \geq p+1$. Suppose that $r\equiv 2\mod (p-1)$, 
$r\not\equiv 2/3\mod p$, and that
$a_p$ is close to $\eps pr/2$, for $\eps\in \{\pm 1\}$.
Then there is a lattice $\Theta'$ in $\hat\Pi_{k,a_p}$ such that
$V(\bar\Theta')$ is a ``peu ramifi\'ee" extension of $\mu_{\eps} \cdot
\omega$ by $\mu_{\eps} \cdot \omega^2$.  
\end{thm}

Note that the condition that $a_p$ is close to $\eps pr/2$ implies that
$\overline{a_p/p} \neq \overline{\eps (1-r)}$, as $r \not\equiv 2/3 \mod
p$.
It also implies
that $r \not\equiv 0 \mod p$.

\subsubsection{The case $r = p+1$}

Let us begin with the case $r = p+1$. 
The semisimplification of the reduction modulo $p$ of $V_{k,a_p}$ 
is already known in this case by \cite{Br03}, \cite{B10}, and is 
$\bar{V}_{k,a_p}^{ss} \cong \mu_\lambda \cdot\omega^2 \oplus
\mu_{\lambda^{-1}} \cdot\omega$, with 
$\lambda = \overline{2 a_p / p} \in \bar\F_p^\times$. 
The distinction between ``peu ramifi\'ee" and ``tr\`es ramifi\'ee"
was also studied in an unpublished paper of Berger-Breuil \cite{BB} using $(\varphi,\Gamma)$-modules 
(see also \cite{Vie} for a more general result).
We give a proof using our methods as this is the smallest weight for which the 
distinction between ``peu ramifi\'ee" and ``tr\`es ramifi\'ee" arises when $v(a_p) = 1$. We keep
the same notation as in Section \ref{casea2}.

\begin{prop}
\label{casep+1}
Let $\lambda = \overline{2a_p/p}$. Then $F_0 = 0$,
$F_1$ is a quotient of $\pi(0,\lambda,\omega)$ 
and $F_2$ is a quotient of $\pi(p-3,\lambda^{-1},\omega^2)$. 
\end{prop}

\begin{proof}
As mentioned, this was computed by Breuil in \cite{Br03}. We sketch a proof
by the methods used in Section \ref{casea2}, but the computation is
simpler for $r = p+1$ than what was done there.  
For $F_0$, we already have $J_0 = 0$ in the filtration
of $P$, for $r = {p+1}$. For $F_1$, compute $(T-a_p)[\Id,(1/p)\theta]$ and
observe that 
$\theta \equiv 2X^pY \equiv - 2XY^p \mod X_{p+1}$. For $F_2$, 
compute $(T-a_p)(\sum_{\lambda}[g_{1,[\lambda]}^0,(1/p)(X^{p-1}Y^2-Y^{p+1}])$.
\end{proof}

\begin{cor}
$\bar{\Theta}_{p+3,a_p}$ is an extension of
$\pi(p-3,\lambda^{-1},\omega^2)$ by $\pi(0,\lambda,\omega)$ and 
$\bar{V}^{ss}_{p+3,a_p} \iso \mu_{\lambda} \cdot \omega^2 \oplus \mu_{\lambda^{-1}} \cdot \omega$,
for $\lambda=\overline{2a_p/p}$.
\end{cor}

\begin{thm}
\label{p+1thm} 
Let $p \geq 5$ and $r = p+1$.
Suppose that $\overline{a_p/p} = \pm 1/2$, so that $\lambda = \pm 1$. Then there exists a lattice
$\Theta'$ in $\hat{\Pi}_{p+3,a_p}$ such that $V(\bar{\Theta}')$ is a
``peu ramifi\'ee" extension of $\mu_{\lambda^{-1}} \cdot \omega$ by
$\mu_{\lambda} \cdot \omega^2$.
\end{thm}

\begin{proof}
By twisting, we may assume $\lambda = 1$. The theorem then 
follows from Corollary \ref{peuortres}, Proposition \ref{pi011} and
the previous corollary.
\end{proof}

\begin{remark}
Actually, as the referee pointed out to us, in Theorem \ref{p+1thm} we can simply take for
$\Theta'$ the standard lattice $\Theta$.
Indeed, we already know that $\bar\Theta$ is an extension of 
$\pi(p-3,\lambda^{-1},\omega^2)$ by $\pi(0,\lambda,\omega)$ and we need only see that the
extension is non-split. This can be deduced from the fact that 
$\bar\Theta$ is a quotient of $\ind_{KZ}^G V_{p+1}/V_2$, and $V_{p+1}/V_2$
is a non-split extension of $V_{p-3}\otimes D^2$ by $V_0 \otimes D$.
\end{remark}

\subsubsection{The case $r > p+1$}

Now we will generalize the theorem above to the cases $r>p+1$.  For
simplicity, we shall assume $\overline{a_p/p}=+ {r/2}\in k_E$ since the
proof in the case $\overline{a_p/p}=- {r/2}$ is identical.  Recall that the
conditions $v(a_p)=1$ and $\overline{a_p/p}= {r/2}$ together imply that
$r\not\equiv 0\mod p$. 

Recall that when $a = 2$, the submodule $W_1\cong
V_r^*/V_r^{**}=J_0\oplus J_1\subset  P$, where $J_0=V_{p-1}\otimes D$
and $J_1=V_0\otimes D$.  
The image of $\mathrm{ind}_{KZ}^G
W_1$ in $\bar\Theta_{k,a_p}$ was denoted by $U_1$.  


Let $F_1'$ be the image
in $U_1$ of the submodule $\mathrm{ind}_{KZ}^G J_1$ of $\ind_{KZ}^G W_1$.
Then $F_0':=U_1/F_1'$ is a quotient of $\mathrm{ind}_{KZ}^G J_0$ and we have a 
diagram similar to Diagram \eqref{gendiagram2}:

\begin{eqnarray}\label{gendiagram3}
\begin{gathered}
    \xymatrix{
    0 \ar[r] & \mathrm{ind}_{KZ}^G J_1\ar@{->>}[d]\ar[r] & \mathrm{ind}_{KZ}^G W_1 \ar@{->>}[d] \ar[r] & \mathrm{ind}_{KZ}^G J_0 \ar@{->>}[d]\ar[r] & 0 \\
    0 \ar[r] & F_1' \>\ar[r]  & U_1 \ar[r] & F_0' \ar[r] & 0. }
\end{gathered}
\end{eqnarray}



\begin{prop}
\label{F'0}
Suppose that $r \equiv 2 \mod (p-1)$ and that $v(a_p^2-\binom{r}{2}p^2)
= 2 + v(r-2)$. Let $\lambda \in \bar\F_p^\times$ be the reduction mod $p$
of $\frac{2(a_p^2-\binom{r}{2}p^2)}{pa_p(2-r)}$. Suppose that $\overline{a_p/p}$ is not equal to 
$\lambda \overline{(1-r)}$. Then $F_0' = 0$.
\end{prop}

Note that 
$F_0 \neq 0$, by Remark~\ref{remark a=2},
so  $F_0$ and $F'_0$ do not always coincide, 
although each occurs as a common factor in both
$\mathrm{ind}_{KZ}^GJ_0$ and $U_1$.

In order to prove the proposition, we need a few lemmas.

\begin{lemma}
\label{lemmanew} 
Let $p\geq 3$, $r\geq 2p$, with $ r\equiv 2\mod (p-1)$.  Then the image
of $F(X,Y)=X^{r-1}Y +\dfrac{r}{2}\cdot\theta X^{r-2p}Y^{p-1}
-\dfrac{r}{2}\cdot\theta Y^{r-p-1}$ in $P$ lands in $W_0 (\cong J_0\subset P)$ and it maps to
$\dfrac{2-r}{2}\cdot X^{p-1} \in J_0$.
\end{lemma}

\begin{proof}
Using Lemma \ref{generator} and Lemma \ref{goodcasenew} (i), we get that
$F(X,Y)$ maps to $0\in J_2$, and it maps to
$-\frac{r}{2}+\frac{r}{2}+0=0\in J_1$. Therefore its image in $P$ lands
in $W_0$, as claimed. To find its image in $J_0=V_{p-1}\otimes D$, we
rewrite $F(X,Y)$ modulo $X_r+V_r^{**}$ as follows:  
\begin{eqnarray*}
 F(X,Y) &\equiv& X^{r-1}Y +\dfrac{r}{2}\cdot\theta X^{r-2p}Y^{p-1} -\dfrac{r}{2}\cdot\theta Y^{r-p-1}+\frac{1}{2}\cdot\underset{k\in\F_p}\sum k^{p-2}(kX+Y)^r\mod X_r\\
 &\equiv& \dfrac{r}{2}\cdot\theta X^{r-2p}Y^{p-1} -\dfrac{r}{2}\cdot\theta Y^{r-p-1}+X^{r-1}Y -\frac{1}{2}\cdot\underset{\substack{0\leq j\leq r\\j\equiv 1\mod (p-1)}}\sum\binom{r}{j}X^{r-j}Y^j\mod p.
\end{eqnarray*}
The last polynomial is divisible by $\theta$, since, by Lemma
\ref{comb3} (i), the sum of its coefficients is $0$. 
Writing it as an element in $\theta V_{r-p-1}=V_r^*$, we must have $$F(X,Y)\equiv
\theta\cdot\left(\left(-\frac{r}{2}+\frac{r}{2}\right)Y^{r-p-1}+\left(1-\frac{r}{2}\right)X^{r-p-1}+
c\cdot X^{r-2p}Y^{p-1}\right)\mod X_r +V_r^{**},$$ for some $c\in\F_p$.
We already know that $F(X,Y)$ maps to $0\in J_1$, so $c=0\in\F_p$, by
Lemma \ref{generator} (ii). Now the result follows by 
Lemma~\ref{generator} (i) 
applied to the right hand side of the congruence above.  
\end{proof}

\begin{lemma}
\label{projection}
For $r>2p$, the projection $\mathrm{Pr} : W_1\cong V_r^*/V_r^{**}
\twoheadrightarrow J_0=V_{p-1}\otimes D$ takes the polynomials $\theta
X^s,\,\theta Y^s \text{ and } \,\theta X^{r-2p}Y^{p-1}$ to $X^{p-1},\,
Y^{p-1}$ and $X^{p-1}+Y^{p-1}$  respectively, where $s=r-p-1$.  Moreover,
the image of $X^{r-1}Y$ in $P$ lands in $W_1$ and it projects to
$(1-r)\cdot X^{p-1}\in J_0$.
\end{lemma}
 
\begin{proof}
The image of $\theta X^{s}$ is $X^{p-1}$, 
Lemma \ref{generator} (i). The image of $\theta Y^s$ is then obtained by
applying the matrix $w=\left(\begin{smallmatrix}0 & 1\\ 1 & 0
\end{smallmatrix}\right)$ to it and by using $\Gamma$-linearity of
$\mathrm{Pr}$.  The image  of $\theta X^{r-2p}Y^{p-1}$ under
$\mathrm{Pr}$ can be calculated  using the action of $w$, the diagonal
matrices and the unipotent matrices of $\Gamma$, on the polynomial
$\theta X^{r-2p}Y^{p-1}$, and is left as an exercise.
   
The polynomial 
$H(X,Y):=X^{r-1}Y+\dfrac{r}{2}\cdot\theta
X^{r-2p}Y^{p-1}-\dfrac{r}{2}\cdot \theta Y^s-\dfrac{2-r}{2}\cdot \theta
X^s\in V_r$ maps to $0$ in $ P$,  by Lemma \ref{lemmanew} and Lemma
\ref{generator} (i).  Therefore using the previous calculations we obtain
$$ \mathrm{Pr}(X^{r-1}Y)=
-\frac{r}{2}\cdot(X^{p-1}+Y^{p-1})+\frac{r}{2}\cdot
Y^{p-1}+\frac{2-r}{2}\cdot X^{p-1}=(1-r)\cdot X^{p-1}.$$     
\end{proof}

\begin{proof}[Proof of Proposition \ref{F'0}]
We consider the function $f=\left[\mathrm{Id},\,\frac{1}{p}\cdot \theta
X^s\right]\in\mathrm{ind}_{KZ}^G \> \Sym^r \bar\Q_p^2$ and compute that $(T-a_p)f$
is integral and reduces to $-\left[\mathrm{Id},\,\frac{a_p}{p}\cdot
\theta X^s\right]
+\underset{\lambda\in\F_p}\sum\left[g_{1,[\lambda]}^0,\,X^{r-1}Y\right]$
mod $p$, which further maps to $-\left[\mathrm{Id},\,\frac{a_p}{p}\cdot
X^{p-1}\right]
+\underset{\lambda\in\F_p}\sum\left[g_{1,[\lambda]}^0,\,(1-r)\cdot
X^{p-1}\right]$ in $\mathrm{ind}_{KZ}^G J_0$, by Lemma \ref{projection}.

This shows that $F_0'=0$, when $r\equiv 1\mod p$. Suppose now that $r
\not\equiv 1 \mod p$. Then $F_0'$ is
a quotient of $\pi(p-1,\frac{a_p}{p(1-r)},\,\omega)$.
It follows from 
the proof of Theorem~\ref{thma2} (i) that the JH
factors of $U_1$ are the same as the JH factors of
$\pi(p-1,\lambda,\omega)$. But we have assumed that $\lambda$ is not
equal to the reduction mod $p$ of $\frac{a_p}{p(1-r)}$, so
$\pi(p-1,\lambda,\omega)$ and $\pi(p-1,\overline{\frac{a_p}{p(1-r)}},\,\omega)$ have
no common JH factor. So $F'_0 = 0$.
\end{proof}

\begin{proof}[Proof of Theorem \ref{a2peuramifiethm}]
Consider the subrepresentation $U_1$ of $\bar\Theta_{k,a_p}$. 
Its JH factors are those of $\pi(p-1,1,\omega)$, that is $\omega$ and
$\St\circ\,\omega$. 
Consider its subrepresentation $F_1'$ and quotient
$F_0' = U_1/F_1'$ as above. We are in a
situation where Proposition \ref{F'0} applies, as we remarked after the
statement of Theorem \ref{a2peuramifiethm}.
So $U_1 = F'_1$ is a quotient of $\mathrm{ind}_{KZ}^G (V_0\otimes D)$.
This forces $U_1$ to be $\pi(0,1,\omega)$, by Lemma \ref{quotientlemma}.
The theorem now follows from Corollary \ref{peuortres} and Proposition
\ref{pi011}, since $\bar\Theta_{k,a_p}$ is an   extension of $\pi(p-3,
1,\omega^2)$ by $U_1$.
\end{proof}

\subsection{The case $a=2$ and $r \geq 2p$ and $a_p$ is close to $\eps
p(1-r)$ and $r$ is not $2/3$ modulo $p$}
\label{a2ramifie}
Let $E = \Q_p(a_p)$, let $R = \O_E$ and let $\m_E$ be its maximal ideal, with
uniformizer $\varpi$.

\subsubsection{Statement}

\begin{thm}
\label{a2ramifiethm}

Let $p \geq 5$ and $r \geq 2p$. Suppose that $r\equiv 2\mod(p-1)$, 
$r\not\equiv2/3$ modulo $p$,
and that
$a_p$ is close to $\eps p(1-r)$, for 
$\eps \in \{\pm 1\}$.
Let $u = u(a_p) = \frac{2(a_p^2-\binom{r}{2}p^2)}{pa_p(2-r)}$. 
If $r \equiv 2 \mod p$, assume further that either $E$ is unramified over $\Q_p$ or $u-\eps$ is a
uniformizer of $E$. 

\begin{enumerate}

\item
If $v(u-\eps) < 1$, then 
there exists a lattice $\Theta'$ in $\hat\Pi_{k,a_p}$ such that
$V(\bar\Theta')$ is a ``peu ramifi\'ee" extension of
$\mu_{\eps} \cdot \omega$ by $\mu_{\eps} \cdot \omega^2$.

\item
If $v(u-\eps) \geq 1$, then
there exists a lattice $\Theta'$ in $\hat\Pi_{k,a_p}$ such that
$V(\bar\Theta')$ is a ``tr\`es ramifi\'ee" extension of
$\mu_{\eps} \cdot \omega$ by $\mu_{\eps} \cdot \omega^2$.

\item
Moreover, for a fixed $r\not\equiv 1\mod p$, 
all isomorphism classes of ``tr\`es ramifi\'ee" extensions appear
depending on the choice of $a_p$. 
\end{enumerate}
\end{thm}

Note that the condition that $a_p$ is close to $\eps p(1-r)$ implies that
$\overline{u(a_p)} =\bar\eps \neq \overline{ rp/2a_p}$, as $r \not\equiv 2/3
\mod p$. It also implies $r \not\equiv 1 \mod p$.

Under the conditions of the theorem, we have that $F_1 = 0$,
by the following proposition. 

\begin{prop}
\label{F1a2bis}
Suppose that $v(a_p^2-\binom{r}{2}p^2) = 2 + v(r-2)$. If
$\lambda = \overline{\frac{2(a_p^2-\binom{r}{2}p^2)}{pa_p(2-r)}} \neq
\overline{rp/2a_p}$, then $F_1 = 0$.
\end{prop}

\begin{proof}
Recall that $F_1$ is a proper quotient of $\pi(0,\overline{rp/2a_p},\omega)$,
by Proposition~\ref{F1a2}, and we can assume that $p$ does not divide $r$.

If
$\overline{rp/2a_p} \neq \pm 1$, then $\pi(0,\overline{rp/2a_p},\omega)$ is
irreducible, and so $F_1 = 0$.

In general (allowing $\overline{rp/2a_p}$ to be $ \pm 1$), by the mod $p$ LLC, we know that the JH factors of $F_0$, $F_1$ and $F_2$
are contained in the JH factors of $\pi(p-1,\lambda,\omega)$ and
$\pi(p-3,\lambda^{-1},\omega^2)$. As $\lambda \neq \overline{rp/2a_p}$,
the JH factors of $\pi(0,\overline{rp/2a_p},\omega)$ are not contained in this set
and so $F_1 = 0$.
\end{proof}

\begin{remark}
\label{F1zerorem}
The final paragraph of the proof of Proposition~\ref{F1a2bis} 
does not make any distinction between whether 
$\overline{rp/2a_p} = \pm 1$ or not.
However, it
will be useful later to know that if
$\overline{rp/2a_p} \not= \pm 1$, then we can prove $F_1 = 0$ without
using the mod $p$ LLC and solely by computing explicit elements
in the kernel of the map $\ind_{KZ}^GV_r \to \bar\Theta_{k,a_p}$; see
Remark \ref{F1zerorembis}.  This is why we have separately mentioned the proof 
given in the second paragraph above.
\end{remark} 

For the rest of this section,
fix $a_p$, $r$ and $u = u(a_p)$, 
as in the hypotheses of Theorem
\ref{a2ramifiethm}. We suppose also that $\eps = 1$ 
(and so $\overline{a_p/p} = 1-r$) 
to simplify notation, since 
the proof in the case that $\eps = -1$ is identical.

We define an element $\delta\in R$ as follows:
\begin{eqnarray*}
  \delta = 
  \begin{cases}
     u-1 & \text{if $v(u-1)<1$,} \\
       p & \text{if $v(u-1) \geq 1$.}
  \end{cases}
\end{eqnarray*}
Note that $\delta$ is a uniformizer of $E$
if and only if either $u-1$ is a uniformizer of $E $ or $E$ is unramified.

\subsubsection{Comparing lattices for different values of $a_p$}
\label{compalat}

In this subsection (and only in this subsection), we make the assumption that $r \not\equiv 2 \mod
p$.

Let $\Lambda = \ind_{KZ}^G\Sym^rR^2$. Let $\Theta$ be the standard
lattice for $a_p$ and let 
$\Theta^0_R$ be the standard lattice when we take
$a_p^0 = p(1-r) \in \Z_p \subset R$ for the value of $a_p$, thought of as an $R$-lattice.
Then:

\begin{prop}
\label{compquotient}
$\Theta\otimes_R \frac{R}{(\delta)}$ and $\Theta^0_R \otimes_R \frac{R}{(\delta)}$ are
isomorphic as $R[G]$-modules.
\end{prop}

Let $M = (T-a_p)(\Lambda\otimes_R E) \cap \Lambda$ and 
$M_0 = (T-a_p^0)(\Lambda\otimes_R E) \cap \Lambda$.
We have that 
$\Theta\otimes_RR/(\delta) = \Lambda / (M + \delta \Lambda)$
and 
$\Theta^0_R \otimes_RR/(\delta) = \Lambda / (M_0 + \delta \Lambda)$
as representations of $G$, so
to prove Proposition \ref{compquotient} it is enough to prove the
following result:

\begin{prop}
\label{compsublattice}
$M + \delta \Lambda = M_0 + \delta\Lambda$.
\end{prop}

In order to prove this proposition, we introduce the submodule $M'$ of
 $M$ generated over $R[G]$ by all the  functions of the form $(T-a_p)f$
that we used to compute $\bar{V}^{ss}_{k,a_p}$ in the proofs of
Propositions~\ref{F0a2}, \ref{F1a2} 
 and \ref{F2a2}.  To be more precise,
$M'$ is generated by the integral functions  $(T-a_p)f$ introduced in the
proofs of these propositions, and also  the functions $f$ used in
\cite[Rem. 4.4]{BG09} to show that $X_r$ and $V_r^{**}$ do not contribute
to $\bar{\Theta}$.

Let $M'_0$ be the analogous $R[G]$-submodule of $M_0$ generated by the
integral functions $(T-a^0_p)f^0$, for the rational functions $f^0$ as
above defined using $a_p^0 = p(1-r)$. 
We have:

\begin{lemma}
\label{compcomputed}
$M' + \delta \Lambda = M_0' + \delta \Lambda$.
\end{lemma}

\begin{proof}
It is enough to show that  the various functions $(T-a_p)f$ and $(T-a_p^0)f^0$ used to define
$M'$ and $M'_0$ are same modulo $\delta\Lambda$. 

The functions $(T-a_p)f$ and $(T-a_p^0)f^0$ used to show that $V_r^{**}$ and $X_r$ do
not contribute to $\bar{\Theta}$ and $\bar{\Theta}^0_R$ respectively, are the same even up to
$p\Lambda$. 
Next, the identity 
$$u-1=\frac{2}{2-r}\cdot\frac{p}{a_p}\left(\frac{a_p}{p}-\frac{r}{2}\right)\left(\frac{a_p}{p}-(1-r)\right)$$ 
implies the equality of valuations $v(u-1)=v\left(\frac{a_p}{p}-(1-r)\right)$, 
as the other factors are units by the hypotheses $v(a_p)=1$, $r\not\equiv 2\mod p$ and $\overline{a_p/p}\neq \pm{r/2}\in\bar\F_p$. 
Using the definition of $\delta$, it follows that
$$v\left(\frac{a_p}{p}-\frac{a_p^0}{p}\right)=v\left(\frac{p}{a_p}-\frac{p}{a_p^0}\right)=v(u-1)\geq v(\delta).$$
Using this, one checks case by case that for each of the functions $f$ and $f^0$ 
used to define 
$M'$ and $M_0'$ respectively, the integral functions
$(T-a_p)f$ and $(T-a_p^0)f^0$ are the same modulo $\delta\Lambda$.
%
%
\end{proof}

\begin{remark}
  The lemma also holds for $r \equiv 2 \mod p$, since  $a_p/p$ is close to 
  $1-r$ implies that we may replace the inequality above
  with $v(\frac{a_p}{p} - \frac{a_p^0}{p}) = v(u-1) + v(2-r) \geq v(\delta)$. 
\end{remark}

\begin{lemma}
\label{equalmodvarpi}
$M' + \varpi \Lambda = M + \varpi \Lambda$ and
$M'_0 + \varpi \Lambda = M_0 + \varpi \Lambda$.
\end{lemma}

\begin{proof}
Since $M' \subset M$, there is a $G$-equivariant surjection $\Lambda / (M' + \varpi \Lambda) \to 
\Lambda / (M + \varpi\Lambda)$. This map is in fact an isomorphism.
Indeed, looking back at Section \ref{casea2}, we see that we effectively
computed a list of possible JH factors  in $\Lambda / (M' + \varpi \Lambda)$, that is,
we computed the set $S=\JH(\Lambda / (M' + \varpi \Lambda))$.
More precisely, Propositions~\ref{F2a2}, \ref{F0a2} and (part of) 
\ref{F1a2bis} show that $S = \{\omega, \St\otimes\, \omega, \pi(p-3,1,\omega^2)\}$.
The surjection above implies that
$\JH(\Lambda / (M + \varpi \Lambda)) \subset S$. 
However, considerations about the mod $p$ LLC
show that $\JH(\Lambda / (M + \varpi \Lambda)) = S$, and so we
have $\JH(\Lambda / (M' + \varpi \Lambda)) = S = \JH(\Lambda / (M +
\varpi \Lambda))$.
Hence the kernel of the map above is zero, so $M' + \varpi \Lambda = M + \varpi \Lambda$.

The same reasoning applies to $M_0$ and $M'_0$.
\end{proof}

\begin{remark}
\label{F1zerorembis}
It is in the proof of this lemma that we use the additional condition  
$r \not\equiv 2 \mod p$. Indeed, as $r \not\equiv 2/3 \mod p$, we have 
$\overline{rp/2a_p} \neq \pm 1$ if and only if $r \not\equiv 2 \mod p$,
and we can use the second paragraph in the proof of Proposition~\ref{F1a2bis}.
When $r \equiv 2 \mod p$, 
$\JH(\Lambda / (M + \varpi \Lambda)) \subset S$, but we do not have
equality.  Indeed, in this case we have $\overline{\frac{rp}{2a_p}}=\overline{\frac{1}{1-r}}=-1$ 
and Proposition~\ref{F1a2bis} is proved using the mod $p$ LLC and so does not give information about $M'$. 
Propositions~\ref{F2a2}, \ref{F0a2} and \ref{F1a2}
do show that the set
$S = \JH(\Lambda / (M' + \varpi \Lambda))=\{\omega,
\St\otimes\,\omega,\omega\otimes\mu_{-1},\pi(p-3,1,\omega^2)\}$, 
as $\lambda=\bar\eps=1$.
However, considerations about the image of the mod $p$ LLC show that necessarily 
$\JH(\Lambda / (M + \varpi \Lambda)) =
\{\St\otimes\,\omega, \omega,\pi(p-3,1,\omega^2)\} \neq S$.
\end{remark}

We now remark that $M$ and $M_0$ are saturated, that is, if $z\in
\Lambda$ is such that $\varpi z\in M$, then $z \in M$.

We then deduce from Lemma \ref{complattice} below,
applied to $(\Lambda,M,M')$ and $(\Lambda,M_0,M'_0)$, that
$M' + \delta \Lambda = M + \delta \Lambda$
and
$M'_0 + \delta \Lambda = M_0 + \delta \Lambda$,
hence
$M + \delta \Lambda = M_0 + \delta \Lambda$, which completes 
the proof 
of Proposition \ref{compsublattice}.

\begin{lemma}
\label{complattice}
Let $\Lambda$ be a free $R$-module, $M \subset \Lambda$ a saturated
submodule, and $M' \subset M$ a submodule such that 
$M + \varpi\Lambda = M' + \varpi \Lambda$. 
Then $M + \varpi^n\Lambda = M' + \varpi^n \Lambda$, for all $n \geq 1$. 
\end{lemma}

\begin{proof}
We reason by induction on $n$. Suppose that
$M + \varpi^n\Lambda = M' + \varpi^n \Lambda$. Let $x \in M$.
There exist $y \in M'$ and $z \in \Lambda$ with $x = y + \varpi^n z$.
Then $\varpi^n z \in M$, hence $z \in M$, as $M$ is saturated. So there
exist $y'\in M'$ and $z'\in \Lambda$, with $z = y' + \varpi z'$.
Then $x = (y+\varpi^n y') + \varpi^{n+1}z' \in M'+\varpi^{n+1}\Lambda$.
\end{proof}

Note that in the course of proving Proposition~\ref{compsublattice} we
have also proved:
\begin{cor}
\label{compresrepr}
$M + \varpi \Lambda = M_0 + \varpi \Lambda$, and hence
$\bar{\Theta}$ is isomorphic to $\bar{\Theta}^0_R$.
\end{cor}

\subsubsection{Study of the standard lattice}
\label{standard}

We allow again the case $r \equiv 2 \mod p$.

Let $\Theta = \Theta_{k,a_p} \subset \Pi_{k,a_p}$ be the standard lattice, and let $\hat{\Theta}
\subset \hat{\Pi}_{k,a_p}$ be its completion.
Let $\Lambda = \ind_{KZ}^G\Sym^rR^2$.

As $F_1=0$, by Proposition~\ref{F1a2bis},  we know that
$\bar{\Theta}$ is an extension of 
$\pi(p-3,1,\omega^2)$ by
$\pi(p-1,1,\omega)$.
Hence we have a filtration of $\bar{\Theta}$ with successive quotients
$\omega$, $\St\otimes\,\omega$ and $\pi(p-3,1,\omega^2)$. We want to lift 
this filtration to $\hat{\Theta}$.

Let  $\Lambda^0 = \text{ind}_{KZ}^G\Sym^r\Z_p^2$ and let ${\Theta}^0$ be the standard lattice defined over $\Z_p$, 
for $a_p^0 = p(1-r) \in \Z_p$. Thus $\Theta^0_R = \Theta^0 \otimes_{\Z_p} R$. 
Let $\mathcal{B}_1$ be the set with one element which is the image of 
$[\Id,\theta X^{r-p-1}] + [\alpha,\theta Y^{r-p-1}]$ in $\bar\Theta^0$, so 
that the line $L$ 
in $\bar{\Theta}^0$
generated by $\mathcal{B}_1$ is the line on which $G$ acts by $\omega$.
Let $\mathcal{B}_2$ be a free family of
elements of $\bar{\Theta}^0$ such that $\mathcal{B}_1 \cup \mathcal{B}_2$
generate the subrepresentation $\pi(p-1,1,\omega)$. 
Let $\mathcal{B}_3$ be a free family of
elements of $\bar{\Theta}^0$ such that 
$\mathcal{B}_1 \cup \mathcal{B}_2 \cup \mathcal{B}_3$ is a basis of
$\bar{\Theta}^0$,
and moreover the image of $\mathcal{B}_3$ in $\bar{\Theta}^0/L$ 
generates the subrepresentation of $\bar{\Theta}^0/L$ 
isomorphic to $\pi(p-3,1,\omega^2)$, 
which is 
possible as $\bar{\Theta}^0/L$ 
is a split extension of $\pi(p-3,1,\omega^2)$ 
by $\St\otimes\,\omega$, by
Proposition~\ref{split}.
We lift each $\mathcal{B}_i$ to a free family
$\mathcal{B}'_i$ in $\Lambda^0$, 
taking $\mathcal{B}'_1$ to be the singleton set with the 
element $[\Id,\theta X^{r-p-1}]+[\alpha,\theta Y^{r-p-1}]$. 

For $\hat{\Theta}$ coming from any $a_p$, we now take $S_i$ to be 
the closure of the free
$R$-submodule generated by the image of $\mathcal{B}'_i$ inside
$\hat{\Theta}$, under the composition of maps $\Lambda^0 \to \Lambda
\to \hat{\Theta}$. By construction, we have a decomposition of
$k_E$-vector spaces
$\bar{\Theta} \cong (\bar{\Theta}^0\otimes_{\F_p}k_E) = 
\bar{S}_1 \oplus \bar{S}_2 \oplus \bar{S}_3$ 
(the first equality is Corollary \ref{compresrepr}),
with
$\bar{S}_1 = L \otimes_{\F_p} k_E$ the line on which $G$ acts by $\omega$ and $\bar{S}_1 \oplus
\bar{S}_2$ is the subrepresentation isomorphic to  
$\pi(p-1,1,\omega)$. 
In particular, $\hat{\Theta} = S_1 \oplus S_2 \oplus S_3$.
We set $M_1 = S_1$ and $M_2 = S_1 \oplus S_2$. Let 
$\bar{M}_i$ be their reductions modulo $\m_E$, for $i = 1$, $2$. The sequence of 
$R$-modules $0 \subset M_1 \subset M_2 \subset \hat\Theta$ lifts the filtration of $\bar\Theta$. 

For $1 \leq i,j \leq 3$, we define functions $a_{i,j}: G \to \Hom(S_i,S_j)$ so that, for all
$g\in G$ and $x\in S_i$, we have $gx = a_{i,1}(g)(x)  + a_{i,2}(g)(x) +
a_{i,3}(g)(x)$.
If $j > i$, then we have $a_{i,j}(g)\in \m_E\Hom(S_i,S_j)$, for all $g \in G$. 
Moreover, by the choice of $\mathcal{B}_3$, we have
$a_{3,2}(g) \in \m_E\Hom(S_3,S_2)$, for all $g \in G$.

\begin{prop}
\label{S1S2}
We have:
\begin{enumerate}
\item
$a_{1,2}(g) \in \delta\Hom(S_1,S_2)$, for all $g \in G$.

\item
$a_{3,2}(g) \in \delta\Hom(S_3,S_2)$, for all $g \in G$.

\end{enumerate}
\end{prop}

\begin{proof}
If $\delta$ is a uniformizer of $E$, both results follow from the remarks just 
before the statement of the proposition.

Now 
suppose that $\delta$ is not a uniformizer of $E$. 
In particular, by the  hypotheses of Theorem~\ref{a2ramifiethm}, we have 
 $r \not\equiv 2 \mod p$ and we can
 use the results from Section~\ref{compalat}.

Let us denote by a superscript $0$ all the analogous constructions with
$a^0_p = p(1-r)$. So $E^0 = \Q_p$, $\delta^0 = p$ is a uniformizer of
$E^0$, $S_i^0$ is the subspace of $\hat{\Theta}^0$ defined as before. 
Hence we get that $a_{1,2}^0(g) \in p\Hom(S_1^0,S_2^0)$, for all
$g \in G$, and $a_{3,2}(g) \in p\Hom(S_3^0,S_2^0)$, for all $g \in G$.

We consider now $\hat\Theta^0_R = \hat\Theta^0\otimes_{\Z_p}R$
and let $S^0_{i,R}$ be the closure of the
image of $S^0_i$ in $\hat\Theta^0_R$. Then $S^0_{i,R}$ is the closure of
the $R$-submodule generated by
the image of the elements of $\mathcal{B}'_i$ inside $\hat\Theta^0_R$.

Then we can also define elements $a_{i,j,R}^0(g) \in
\Hom(S^0_{i,R},S^0_{j,R})$ which come from the elements 
$a_{i,j}^0(g) \in \Hom(S^0_i,S^0_j)$ by $R$-linearity.
In particular, $a_{i,j,R}^0(g) \in p\Hom(S^0_{i,R},S^0_{j,R})$ 
if $j = 2$ and $i=1$ or $i=3$.


By Proposition \ref{compquotient}, there is an  
isomorphism of $R[G]$-modules $\psi: \hat\Theta^0_R/(\delta) 
\to \hat\Theta /(\delta)$.
By the construction of the submodules $S_i$ of $\hat\Theta$, we have that 
$\psi$ induces an isomorphism between $S^0_{i,R}/(\delta)$ and $S_i/(\delta)$. 
Indeed, if $M$ and $M_0$ are as defined just after the statement of
Proposition~\ref{compquotient}, then both spaces arise from the $R$-module 
generated by the image of the elements in $\mathcal{B}'_i$ in 
$\Lambda/(M_0 + \delta \Lambda) = \Lambda / (M + \delta \Lambda)$.

In particular, if $b'$ is an element of $\mathcal{B}'_i$, $b$ is its image in $\hat\Theta$
and $b^0$ is its image in $\hat\Theta^0_R$, then
$a_{i,j}(g)(b)$ modulo $\delta$ is the image by $\psi$ of
$a_{i,j,R}^0(g)(b^0)$ modulo $\delta$.
When $j = 2$ and $i = 1$ or $3$, 
we know that $a_{i,j,R}^0(g)(b^0)$ is in $pS^0_{j,R}$, so it is also in
$\delta S^0_{j,R}$, so 
$a_{i,j,R}^0(g)(b^0) = 0$ modulo $\delta$, so $a_{i,j}(g)(b) = 0$ modulo $\delta$.
Thus $a_{1,2}(g)$ and $a_{3,2}(g)$ are the zero functions modulo $\delta$.
\end{proof}

\subsubsection{Changing the lattice}

Let $S'_1 = \delta S_2$, $S'_2 = S_1$, $S'_3 = S_3$. 
Let $\Theta' = S'_1 \oplus S'_2 \oplus S'_3$, so $\Theta'$ is complete, and 
let $M'_1 = S'_1$ and $M'_2 = S'_1 \oplus S'_2$, so that $0 \subset M_1' \subset M_2' \subset \Theta'$ is
a filtration of $R$-modules. Note that $M'_2$ does not depend on the choice of $S_2$, since 
$M'_2 = \delta M_2 + S_1$.

\begin{prop}
The lattice $\Theta'$ is stable under the action of $G$. Moreover, 
$\bar{\Theta}'$ has a filtration 
$0 \subset \bar M_1' \subset \bar M_2' \subset \bar\Theta'$ 
with successive JH factors isomorphic to
$\St\otimes\,\omega$,
$\omega$ and $\pi(p-3, 1, \omega^2)$. 
\end{prop}

\begin{proof}
The first statement follows from Proposition~\ref{S1S2}. 
Note that $M_1'$ is not $G$-stable, but that $\bar{M}_1'$ is, since
$\bar{M}'_1 = \delta S_2 / \varpi \delta S_2  \cong \bar{S_2} = 
\overline{M_2/M_1} \cong \bar{M}_2/\bar{M}_1 \cong \St \otimes \,\omega$.  
Now $M'_2 / M'_1 \cong S_1 \cong M_1$, so its reduction is 
$\bar S_1=\omega$. Finally, $\Theta' / M'_2 \cong S_3 \cong \Theta/M_2$ 
and its reduction must be the remaining JH factor $\pi(p-3, 1, \omega^2)$. 
\end{proof}

Using Corollary \ref{peuortres}, we now have to study 
$\bar{M}'_2\subset \bar\Theta'$
which is an extension of $\omega$ by $\St\otimes\,\omega$: we must see if
it is non-split and compute a $\tau$ such that it is isomorphic to
$E_{\tau}\otimes\omega$ (note that the fact that it is non-split will
follow from the computation of $\tau$).

\subsubsection{Computations in the new lattice}

Let $\Lambda = \ind_{KZ}^G\Sym^rR^2$ as before
and let $\pi : \Lambda \to \hat{\Pi}_{k,a_p}$, so that
$\pi(\Lambda) = \Theta$. 
 We also denote by $\bar\pi$ the usual
map $\Lambda \to \bar\Theta$.
If $x\in \Lambda$ satisfies $\pi(x)\in \Theta'$, denote
by $\psi(x)$ its image in $\bar{\Theta}'$.
Let  $N \cong \St \otimes \,\omega$ be 
the image of $M'_1 = S'_1$, or equivalently of $\delta M_2$, in $\bar{\Theta}'$,
and let $M$ be the image of $M'_2$ in  $\bar{\Theta}'$.

\begin{lemma}
\label{kerpsi}
We have $p^2\Lambda$ and $p\ker\bar\pi$ are contained in $\ker \psi$.  If $v(\delta)<1$, then so is $p\Lambda$.
Moreover,  $\delta S_1$ and $\delta S_3$ also die in $\bar\Theta'$. 
\end{lemma}

\begin{proof}
By the construction of $\Theta'$, we have $\pi(\delta\Lambda)= \delta\Theta \subset \Theta'$.
So if $t\in R$, with $v(t) > v(\delta)$, then $t\Lambda \subset \ker\psi$.
The statement
about $\delta S_1$ and $\delta S_3$ follows from the construction of
$\Theta'$.
\end{proof}

In particular:
\begin{lemma}
\label{sameimage}
Let $F$, $F'$ be elements of $\Sym^rR^2$ such that the image of $F-F'$ in
$V_r$ is in $X_r + V_r^{**}$. Then $\psi([g,pF]) = \psi([g,pF'])$, for all
$g\in G$.
\end{lemma} 

\begin{proof}
The hypothesis implies that $[g,F-F'] \in \ker\bar\pi$. Now apply the previous lemma. 
\end{proof}

\begin{lemma}
\label{imageXr}
Let $F \in \Sym^r R^2$ be such that its image in $V_r$ lies in $X_r +
V_r^{**}$. Then, for all $g\in G$, $[g,F] \in \ker \pi + p\Lambda$. 
In particular, $\pi([g,F]) \in
\Theta'$ and 
if $v(\delta) < 1$, then $\psi([g,F]) = 0$.
\end{lemma}

\begin{proof}
This follows from the computations described in \cite[Rem. 4.4]{BG09}, showing that the map 
$\ind_{KZ}^G V_r \rightarrow \bar\Theta$ factors through $\ind_{KZ}^G P$. These are recalled in
the proof of the next lemma, so we omit the proof. 
%
\end{proof}

\begin{lemma}
\label{imageVr**}
Let $F \in \Sym^r\Z_p^2$ be such that its image in $V_r$ lies in $X_r + V_r^{**}$. 
Then, for all $g\in G$, $[g,F]$ is in $\ker \psi + p\Lambda^0$. More precisely,
there is a $z \in \Lambda^0$, independent of $a_p$, such that $[g,F] + p z \in \ker \psi$.  
\end{lemma}

\begin{proof}
Let $M$ be the $\Z_p[K]$-submodule of $\Sym^r\Z_p^2$ generated by $X^r$ and 
$\theta^2 \Sym^{r-2(p+1)} \Z_p^2$. Then the reduction of $F$ modulo $p$ 
lies in $X_r + V_r^{**}$, the image of $M$ in $V_r$. So $F$ lies in the sum $M + p \Sym^r\Z_p^2$
and $[g,F]$ lies in the sum of the $\Z_p[G]$-submodule generated by 
$[\Id,M]$ and  $p\Lambda^0$.

Hence, it is enough to show that $[\alpha,Y^r]$ and $[1,\theta^2 F'(X,Y)]$, for $F'(X,Y) \in \Sym^{r-2p-2} \Z_p^2$, 
are in $\ker\psi + p\Lambda^0$, as this last module is $G$-stable. To do this we 
recall the computations described in \cite[Rem. 4.4]{BG09} of $(T-a_p)f$ used to 
show that $\ind_{KZ}^G X_r$ and $\ind_{KZ}^G V_r^{**}$ go to 
zero under the map $\ind_{KZ}^G V_r \rightarrow \bar\Theta$. Indeed, 
computing $(T-a_p)f \in \ker \pi \subset \ker \psi$, for 
$f = [\Id, (\theta/X)Y^{r-p}]$, 
we see that $[\alpha, Y^r] \in \ker \pi + p \Lambda^0 + a_p \Lambda^0 \subset \ker \pi +  p \Lambda^0 
+ p {\mathfrak m}_E \Lambda \subset p \Lambda^0 + \ker \psi$,
since $a_p \equiv p(1-r) \mod p {\mathfrak m}_E$, and both $\ker \pi$ and $p \ker \bar \pi$ are in  
$\ker \psi$ (cf. Lemma~\ref{kerpsi}). Similarly, computing $(T-a_p)f$, for $f = [\Id, (1/a_p) \theta^2 F']$, we get
$[\Id, \theta^2 F'] \in \ker \pi + p (p/a_p) \Lambda^0 \subset \ker \pi + p \Lambda^0 + p {\mathfrak m}_E \Lambda
\subset \ker \psi + p \Lambda^0$, since $p/a_p \equiv 1/(1-r) \mod {\mathfrak m}_E$. 

The argument above shows that, up to $\ker \psi$, 
each of the functions $[\alpha, Y^r]$ and $[\Id, \theta^2 F']$ is of the form $-p z$, for 
some $z \in \Lambda^0$, and that $z$ is  independent of $a_p$ (it only depends on $a_p/p\mod {\mathfrak m}_E$, which is fixed to be $\overline{(1-r)}$ in our case). 
%
\end{proof}

We also have:
\begin{lemma}
\label{imageinSt}
$\psi(\delta \Lambda) \subset N$.
\end{lemma}

\begin{proof}
Let $z$ be an element of $\Lambda$. Then
$\pi(z)$ is an element of $\hat{\Theta}$. Write it as
$\pi(z) = s_1 + s_2 + s_3$, with $s_i \in S_i$. Then the $\delta s_i$ and 
hence $\pi(\delta z)$ are in
$\Theta'$, and $\psi(\delta z)$ is the sum of the reductions of the $\delta s_i$,
for $i = 1$, $2$, $3$. 
But $\delta S_1$ and $\delta S_3$ die in $\bar\Theta'$, by Lemma~\ref{kerpsi}.
Thus $\psi(\delta z)$ is equal to the reduction of $\delta s_2 \in
S'_1$, which by definition, lies in $N$. 
\end{proof}

\subsubsection{A linear form on $N$}

$M$ is an extension
of $\omega$ by $N = \St\otimes \,\omega$, so it is of the form $E_{\tau}\otimes
\omega$ for some $\tau$. We need to compute $\tau$ in order to see if the
extension given by $V(\bar{\Theta}')$
is ``peu ramifi\'ee" or ``tr\`es ramifi\'ee".

We fix $ \mu$ a linear form on $N$ as in Lemma \ref{linearStomega}.
Recall that the element $X^{p-1} \in J_0 = V_{p-1} \otimes D$ corresponds to the image
of $\theta X^s$ in $P$, for $s = r-p-1$, by Lemma~\ref{generator} (i).  
Thus the image of $[\Id, \theta X^s]$ generates 
$F_0 = \bar{M}_2 = \pi(p-1,1, \omega)$ inside $\bar\Theta$. By Lemma~\ref{imageinSt},
the image of $\delta \Lambda$ in $\bar\Theta'$ lies in the subrepresentation $N$ of $M$. 
In fact, $N$ is the image of $\delta M_2$ in $\bar\Theta'$, so it is generated 
as a $G$-representation by the image of $[\Id,\delta\theta X^s]$.
The image of $[\Id,\delta \theta X^s]$ actually lies in $N^{I(1)}$, as we can easily check, 
and it is non-zero since it generates $N$.
 So by the properties of
$\mu$, we know that $\mu([\Id,\delta\theta X^s]) \neq 0$. We normalize
$\mu$ by setting $\mu([\Id,\delta\theta X^s]) = 1$. As $w[\Id,\delta\theta
X^s] = -[\Id,\delta \theta Y^s]$ we see that $\mu([\Id,\delta\theta Y^s]) =
-1$. Also, since $\alpha$ lies in $A$ up to an element of $Z$, we have 
$\mu(\psi(\alpha\delta z)) = \mu(\psi(\delta z))$, for all $z\in \Lambda$, 
by the equivariance of $\mu$ under the action of $A$. These last 
three properties of $\mu$ will be enough for the computations.

For later use in Section \ref{deltap} we prove the following lemma. As in 
Section \ref{standard} and
the proof of Proposition \ref{S1S2}, we decorate by a superscript
$\vphantom{x}^0$ all objects for $a_p^0 = p(1-r)$.
\begin{lemma}
\label{muindepap}
Let $a_p$ be an element for which $\delta = p$. 
Then  $\mu(\psi(pz)) = \mu^0(\psi^0(pz))$, for all $z \in \Lambda^0$.
\end{lemma}

\begin{proof}
This is true by construction of $\mu$ and $\mu^0$, for $z = z_0 := [\Id,\theta
X^s] \in \Lambda^0$. 

Since $N \cong N^0 \otimes k_E$, there is a $G$-equivariant map from $N^0$ to $N$. 
Any such map takes $I(1)$-invariants to $I(1)$-invariants.
Let $\lambda : N^0 \to N$ be the unique $G$-equivariant map sending $\psi^0([\Id,p\theta X^s])$ to
$\psi([\Id,p\theta X^s])$. Then $\mu^0 = \mu \circ \lambda$, by the
uniqueness of the linear form $\mu$ (see Lemma \ref{linearStomega}). 
As $\psi(pz_0) = \lambda(\psi^0(pz_0))$, and $\psi$, $\psi^0$ and $\lambda$
are $G$-equivariant, we get that
$\psi(pz) = \lambda(\psi^0(pz))$, for all $z \in \Z_p[G]z_0$, hence by
applying $\mu$, 
$\mu(\psi(pz)) = \mu^0(\psi^0(pz))$, for all $z \in \Z_p[G]z_0$.

For an arbitrary $z \in \Lambda^0$,
we can write  $z = b_1 + b_2 +
b_3 + x$, for $b_i$ in the subspace of $\Lambda^0$ generated by
$\mathcal{B}'_i$ and for some $x \in \ker\bar\pi^0$.
Note that $\bar\pi^0(\Z_p[G]z_0)$ is 
the subspace $F_0$ of $\bar\Theta^0$, 
generated by $\bar\pi^0(\mathcal{B}'_1)$ and $\bar\pi^0(\mathcal{B}'_2)$. 
So $\bar\pi^0(b_1+b_2)\in\bar\pi^0(\Z_p[G]z_0)$
and we can
write $b_1 + b_2=z_1+ x'$, for some $z_1\in \Z_p[G]z_0$ and
$x'\in \ker\bar\pi^0$.
Hence  we have $z=z_1+z'$, where $z_1\in \Z_p[G]z_0$ and
$z':=b_3+x+x'$. 

Note that $\pi^0(b_3)\in S_3^0$, $\pi(b_3)\in S_3$ and $x,x'\in\ker\bar\pi^0\subset\ker\bar\pi$ (considering $\ker\bar\pi^0$ as a part of $\Lambda$ via
the inclusion $\Lambda^0 \subset \Lambda$), by Corollary~\ref{compresrepr}. Hence $pz'$ lies in both $\ker\psi$ and  $\ker\psi^0$, by Lemma
\ref{kerpsi}. So $\mu(\psi(pz)) = \mu(\psi(pz_1))$ and $\mu^0(\psi^0(pz)) = \mu^0(\psi^0(pz_1))$.
But we have  already shown  $\mu(\psi(pz_1) )=
\mu^0(\psi^0(pz_1))$ and thus we are done.
\end{proof}

Let $e' = [\Id,\theta X^s] + [\alpha,\theta Y^s] \in \Lambda$, so that
$\pi(e')\in S_1 \subset \Theta'$, and let $e = \psi(e')\in \bar{\Theta}'$. 
Then $e$ is an element of $M$ which is not in $N$.
Let $\gamma_0 = \matr {1+p}001$ and $\gamma_1 = \matr p001$. Then
according to Section~\ref{tau}, to know
$\tau$ it suffices to compute
$\mu(\gamma_0e-e)$ and $\mu(\gamma_1e-e)$.

\subsubsection{The case $v(\delta) < 1$}

\begin{prop}
  $\mu(\gamma_0e-e) = 0$.
\end{prop}

\begin{proof}
We have $\gamma_0e'-e' \in p\Lambda$, so $\psi(\gamma_0e'-e') = 0$, by Lemma~\ref{kerpsi}.
So $\gamma_0e = e$ and $\mu(\gamma_0e-e)=0$.
\end{proof}

\begin{prop}
\label{gamma1delta}
$\mu(\gamma_1e-e) = 2$.
\end{prop}

\begin{proof}
We have $\gamma_1e' = [g_{1,0}^0,\theta X^s] + [\Id,\theta Y^s]$, so
$\gamma_1e'-e' = [g_{1,0}^0,\theta X^s] + [\Id,\theta Y^s] - [\Id,\theta X^s]- [\alpha,\theta Y^s]$. 

The computations in the proof of Proposition \ref{F0a2} show that 
\begin{eqnarray}
\label{one}
\sum_{\lambda\in \F_p}[g_{1,[\lambda]}^0,\theta X^s]- u(a_p)[\Id,\theta X^s] \in \ker\pi + p\Lambda.
\end{eqnarray}
Indeed, taking $f$ as in that proof, we have $(T-a_p)f$ is integral by (\ref{october}), 
so is in $\ker \pi$.
Then the inclusion above follows, multiplying \eqref{october} by $u(a_p)=\frac{2c}{2-r}$, which is a unit in this section, and by noting that $[g, X^r]\in \ker \pi+p\Lambda$, for $g\in G$.

Applying the matrix $-w=-\matr 0{1}{1}0$ on the left hand side of \eqref{one} and going mod $p\Lambda$,  we get 
\begin{eqnarray}
  \label{two}
  \sum_{\lambda\in \F_p^*}[g_{1,[\lambda]}^0,\theta X^s] + [\alpha,\theta Y^s]- u(a_p)[\Id,\theta Y^s] \in \ker\pi+p\Lambda,
\end{eqnarray} 
using $w g_{1,[\lambda]}^0= g_{1,[\lambda]^{-1}}^0 \matr 0{-[\lambda]^{-1}} {[\lambda]} p w$, for $\lambda\in\F_p^*$, 
and  $w g_{1,0}^0= \alpha w$. 
Subtracting (\ref{two}) from (\ref{one}), and writing $u(a_p) = 1+\delta$, we have
$$
[g_{1,0}^0,\theta X^s]-[\Id,\theta X^s] + [\Id,\theta Y^s] 
- [\alpha,\theta Y^s]
+ ([\Id,\delta\theta Y^s]-[\Id,\delta\theta X^s]) \in \ker\pi+p\Lambda.
$$
Since $v(\delta) < 1$,  
we have 
$\psi(\gamma_1e'-e') = \psi([\Id,\delta \theta X^s]-[\Id,\delta \theta Y^s])$, by Lemma~\ref{kerpsi}. Hence the result.
\end{proof}

So we have proved the following proposition, which gives part (1) of Theorem
\ref{a2ramifiethm}.
\begin{prop}
\label{taudelta}
If $v(\delta) < 1$, then the extension $0 \rightarrow N \rightarrow M \rightarrow \omega \rightarrow 0$ is non-split, where 
$M\cong E_\tau\otimes \omega$, with $[\tau] = (0:1)$.
\end{prop}

\subsubsection{The case $\delta = p$}
\label{deltap}

\begin{prop}\label{gamma0e-e}
$\mu(\gamma_0e-e) = 3r-2$.
\end{prop}

\begin{proof}
In the lattice $\Lambda$ we have 
$\gamma_0e'-e' - p([\Id,s\theta X^s - X^{r-p}Y^p]-[\alpha,XY^{r-1}]) \in
p^2\Lambda$. 
So $$\mu(\gamma_0e-e) = s - \mu(\psi([\Id,pX^{r-p}Y^p])) - \mu(\psi([\alpha,pXY^{r-1}])),$$ 
by Lemma~\ref{kerpsi}. Note that $\mu(\psi([\alpha,pXY^{r-1}])) = \mu(\psi([\Id,pXY^{r-1}]))$, by the
equivariance of $\mu$ with respect to the action of $A$. 


We know from Lemma~\ref{lemmanew} that the image of 
\begin{eqnarray}
  \label{defF'}
F'(X,Y) = X^{r-1}Y + (r/2) \theta X^{r-2p}Y^{p-1} -(r/2)\theta Y^s -(1-r/2)\theta X^s 
\end{eqnarray}
in $V_r$ lies in $X_r + V_r^{**}$. So, by Lemma \ref{sameimage},
and using that $X^{r-p}Y^p = X^{r-1}Y - \theta X^s$, we have 
$\mu(\psi([\Id,pX^{r-p}Y^p])) = (-r/2)\mu(\psi([\Id,p\theta X^{r-2p}Y^{p-1}]))
-r$.
Applying $w$ to $F'(X,Y)$, we see that the image of 
$$ XY^{r-1} - (r/2)\theta X^{r-2p}Y^{p-1} +(r/2)\theta X^s
+(1-r/2)\theta Y^s$$ in $V_r$ is also in $X_r + V_r^{**}$ (note that
$(r/2) \theta X^{r-2p}Y^{p-1}$ and $(r/2) \theta X^{p-1}Y^{r-2p}$ have the same
image in $V_r/V_r^{**}$, for $r \geq 2p$). So by Lemma \ref{sameimage}, we have
$\mu(\psi([\Id,pXY^{r-1}])) = (r/2)\mu(\psi([\Id,p\theta X^{r-2p}Y^{p-1}]))
+1-r$.
So finally $\mu(\gamma_0e-e) = 3r-2$.
\end{proof}

\begin{prop}
\label{existsc} 
$\mu(\gamma_1e-e) = 2\bar{t}+c$,  where $\bar{t}$ is the image in $k_E$ of $t \in R$ with $u(a_p) = 1+tp$, 
and $c \in \F_p$ is independent of $a_p$.

\end{prop}

\begin{lemma}
\label{descrelement}
There are elements $z$
and $z'$ in $\Lambda^0$, independent of $a_p$, such that
$$
\sum_{\lambda\in \F_p}[g_{1,[\lambda]}^0,\theta X^s]- u(a_p)[\Id,\theta X^s]
+ pz
\in \ker\psi
$$
and
$$
\sum_{\lambda\in \F_p^*}[g_{1,[\lambda]}^0,\theta X^s] + [\alpha,\theta Y^s]-
u(a_p)[\Id,\theta Y^s]
+pz'
\in \ker\psi.
$$
\end{lemma}

\begin{proof}
The existence of $z$ can be proven by 
revisiting carefully the computations leading to the proof of Proposition
\ref{F0a2}.  
The existence of $z'$ then follows by applying the matrix $-w$ as in the proof of Proposition \ref{gamma1delta}.
\end{proof}

\begin{proof}[Proof of Proposition \ref{existsc}]
Let $z$ and $z'$ be the elements, independent of $a_p$, in $\Lambda^0$, from Lemma \ref{descrelement}.
We set $c = \mu(\psi(pz'-pz)) \in \bar{\F}_p$. By Lemma~\ref{muindepap}, 
we have $c = \mu^0(\psi^0(pz'-pz)) \in \F_p$.
Since $u(a_p) = 1+tp$, as in the proof of Proposition \ref{gamma1delta}, we see that
$\psi(\gamma_1e'-e') = \bar{t}\psi([\Id,p\theta X^s]-[\Id,p\theta Y^s]) + \psi(pz'-pz)$.
Hence
$\mu(\gamma_1e-e) = 2\bar{t}+c$. 
\end{proof}


\begin{prop}
\label{taup}
If $\delta = p$, then the extension 
$0 \rightarrow N \rightarrow M \rightarrow \omega \rightarrow 0$
is non-split and $[\tau]$ is of the form $(1:x)$, for some $x\in k_E$. 
Moreover, for any fixed $r \:(\not\equiv 1, 2/3 \mod p)$, as $a_p$ varies, all values of $x \in
\bar{\F}_p$ can occur.
\end{prop}

\begin{proof} The first statement follows from Proposition
\ref{gamma0e-e}, since $3r-2\not\equiv 0\mod p$.  

Now let us fix $r\not\equiv 1, 2/3\mod p$. Given any $x
\in\bar\F_p$, let 
$t$ be a lift of
$\frac{(\overline{3r-2})x-c}{2}\in\bar\F_p$ in some unramified extension of $\Z_p$.  Then the equation $u(a_p) =
1 + tp$ gives rise to a quadratic polynomial in $a_p/p$, 
one of whose roots reduces to $\overline{(1-r)}$, hence is a unit. Note that if $p\mid 2-r$, then this root lies in an unramified extension of $\Z_p$.
 This gives us a $p$-adic integer $a_p$ with
$v(a_p)=1$ such that $a_p$ is close to $\eps p(1-r)$ for $\eps=1$, and moreover 
$E=\Q_p(a_p)$ is unramified when $p\mid 2-r$. Therefore the corresponding $\tau$ satisfies $[\tau]=(3r-2:2\bar t+c)=(1: x)$,
as desired.
\end{proof}

This completes the proof of Theorem~\ref{a2ramifiethm}. \\


\subsection{Trivial semisimplification}
\label{sectiontrivialnonsplit}

In this section we make some remarks about the case when $\bar{V}_{k,a_p}^{ss}$
is the trivial representation, up to a twist. 

Assume that the  reduction $\bar{V}_{k,a_p}$ of a given $G_{\Q_p}$-stable lattice in $V_{k,a_p}$ is a non-split
extension of the trivial representation $1$ by $1$, up to a twist. Then one can
ask whether, after the same twist, the reduction is unramified or ramified. More precisely, the cocycle
describing the (twist of the) reduction is just a non-zero map from $G_{\Q_p}$ to $\bar\F_p$ lying
in the cohomology group $\mathrm{H}^1(G_{\Q_p}, \bar\F_p) \cong \bar\F_p^2$, and is again well-defined 
up to a constant. We say that (the twist of) $\bar{V}_{k,a_p}$ is unramified if this map factors through the 
Galois group of the maximal unramified extension of $\Q_p$, and is ramified otherwise.
In the context of Theorem~\ref{maintheoslopeone}, this question arises in exactly one case, 
namely when 
\begin{enumerate}
\item [(3)] $a=1$, $r \geq 3p-2$, $p$ divides $r$ and $\lambda = \pm \overline{1}$, so
          $\overline{\frac{a_p}{p}} = \lambda \pm \overline{\sqrt{r/p}}$.
\end{enumerate}

It is known, see \cite[Prop. 11]{BC14} (we thank F. Herzig for providing this reference), that the  $G_{\Q_p}$-stable lattices in this setting 
(when the semisimplification of 
the reduction is a direct sum of  equal characters) form a convex, bounded subset of the tree, whose
interior vertices have full valency and correspond to lattices whose reduction is split, and 
whose extremal vertices have valency one and correspond to lattices whose reduction is non-split. 
In particular, since
there may be a large number of extremal vertices, there may be a large number of non-homothetic 
$G_{\Q_p}$-stable lattices to study. 

We shall answer this question now, but only for the lattice in $V_{k,a_p}$ corresponding to the
standard lattice $\Theta = \Theta_{k,a_p}$ on the automorphic side.

We let $p \geq 3$ (allowing $p = 3$ in this subsection; for some related  work for $p = 2$, but in the global setting, see \cite{AC17}).
Assume that $r > 2p$, with $r \equiv 1 \mod (p-1)$, so $a = 1$, and that $p \mid r$. 
Assume, for simplicity, that $\lambda = 1$. We showed (Propositions~\ref{F_2=0}, and \ref{elimination and} with $c = 2$) that $\bar\Theta$, 
the reduction of the standard lattice, is isomorphic to $\ind_{KZ}^G (V_{p-2} \otimes D)/(T-1)^2$. 
The following theorem shows that the corresponding (twist of the) reduction $\bar{V}_{k,a_p} \otimes \omega^{-1}$ 
is an unramified non-split extension of $G_{\Q_p}$, as mentioned in the Introduction.

\begin{theorem}
\label{unramified}
Let $\Pi = (\ind_{KZ}^G V_{p-2}) /(T-1)^2$. Then $V(\Pi)$
is a non-split, unramified extension of the trivial mod $p$ representation by
itself.
\end{theorem}

\begin{proof}
Let $W = V(\Pi)$. By the semisimple mod $p$ local Langlands correspondence,
$W$ is a two-dimensional mod $p$ representation of
$G_{\Q_p}$ with trivial semisimplification. By \cite[Section
1.5.1 and Lemma 1.5.2]{Kis}, we know that there is an operator $S$ acting
on $W$ such that $W$ is a free $\bar\F_p[[S]]/S^2$-module of rank $1$ and
$S$ commutes with the action of $G_{\Q_p}$.
By \cite[Lemma 1.5.9]{Kis}, we know that $G_{\Q_p}$ acts on $W$ through
the character $\mu_{1+S}$ (with values in $(\bar\F_p[[S]]/S^2)^\times$, noting that the central character 
$\psi$ in the statement of the lemma is $\omega^{p-2}$). So we see that
the action of $G_{\Q_p}$ on $W$ is unramified. Moreover $S$ acts
non-trivially on $W$, so $W$ is not the trivial representation of
dimension $2$ of $G_{\Q_p}$.
\end{proof}

\section{Examples}
\label{sectionexamples}

In this final section we  compare our general theorems with some 
specific examples in the literature, and also 
provide some new information about these examples. 

Consider the Delta function $\Delta = \sum_{n=1}^\infty \tau(n) q^n$, the unique
normalized cusp form of level 1 and weight  $k=12$. 
It is known that the only primes $p < 10^6$ for which $\Delta$ has 
positive slope (i.e., $p \mid \tau(p) = a_p)$ are $p = 2$, $3$, $5$, $7$ and $2,411$, though recently it
was discovered that the 10 digit prime $p = 7,758,337,633$ is also of positive slope \cite{LR10}.
Serre and Swinnerton-Dyer \cite{Serre73} computed the global reduction of the Galois representations
attached to cusp forms of level 1 and small weight,  for small primes $p$ in order to explain 
congruences going back to Ramanujan. In particular, they proved that the shape of 
$\bar\rho_\Delta^\mathrm{ss} : G_\Q \rightarrow \mathrm{GL}_2(\F_p)$, for $p \leq 7$, is 
as follows:
\begin{itemize}
   \item If $p = 2$, so $\tau(p) = -24$  and $v(a_p) = 3$, then $\bar\rho_\Delta^{ss}$ is trivial. 
   \item If $p = 3$, so $\tau(p) = 252$ and $v(a_p) = 2$,  then $\bar\rho_\Delta^{ss} \cong
                                           \left( \begin{matrix} 
                                                        \omega & 0  \\ 0 & 1
                                           \end{matrix} \right)$.  
  \item If $p = 5$, so $\tau(p) = 4,830$ and $v(a_p) = 1$, then
            $\bar\rho_\Delta^{ss} \cong
                                          \left( \begin{matrix} 
                                                        \omega & 0  \\ 0 & 1
                                           \end{matrix} \right) \otimes \omega$. 
  \item If $p = 7$, so $\tau(p) = -16,744$, then $v(a_p) = 1$ and 
            $\bar\rho_\Delta^{ss} \cong
                                          \left( \begin{matrix} 
                                                        \omega^{3} & 0  \\ 0 & 1
                                           \end{matrix} \right) \otimes \omega$. 
\end{itemize}

For primes $p$ of positive slope, the restriction of 
$\bar{\rho}_\Delta^{ss}$ to $G_{\Q_p}$ is isomorphic to $\bar{V}_{12, \tau(p)}^{ss}$, 
at least if $\tau(p)^2 \neq 4p^{11}$,
so one may compare the results above for slope 1 with the results in this paper.
The primes $p \geq 11$ are in the well understood Fontaine-Lafaille range $2 \leq k \leq p+1$,  covered by
the work of Edixhoven \cite{Edixhoven92} (for all positive slopes), so we do not comment further on these
primes here. Also, the restriction of $\bar\rho_\Delta^\mathrm{ss}$ above to $G_{\Q_p}$ when $p = 7$ 
matches with what was computed by Breuil \cite{Br03} who treated weights $k \leq 2p+1$ for all positive slopes 
(though the case $k = 2p+1$ was only later stated in \cite{Berger11}).

We now consider the case $p = 5$. In this case, the shape of the global representation $\bar\rho_\Delta^{ss}$ above 
explains the congruence $\tau(n) \equiv n \sigma_1(n) \mod 5$, for all $n \geq 1$, which follows, e.g., from \cite{Bam46}.  We show
that the restriction of $\bar\rho_\Delta^{ss}$ to $G_{\Q_p}$ 
matches with what is predicted in general by Theorem~\ref{maintheoslopeone}.
In the notation of that theorem, we have $r = 10 \equiv 2 \mod (p-1)$, so $b =2$. Moreover 
$$v \left(\frac{a_p}{p} - {r \choose 2}\frac{p}{a_p} \right)= 0 = v(2-r),$$
so we are in the middle case of the  trichotomy there. Since 
$$\lambda \quad \equiv \quad \frac{2}{2-r} \left( \frac{a_p}{p} - {r \choose 2} \frac{p}{a_p} \right)
 \quad \equiv \quad
                          -\frac{2}{8} (966) \quad \equiv \quad  1 \mod p, $$ 
the unramified characters $\mu_{\lambda^{\pm1}}$ are trivial, and we recover that 
$(\bar \rho_\Delta |_{G_{\Q_p}})^{ss} \cong \omega^2 \oplus \omega$. 

But we can say more. Since $\frac{a_p}{p} \equiv 1 \equiv \varepsilon(1-r)  \mod p$ with $\eps = 1$,  and
$v(u  - \eps)$ is necessarily $\geq 1$ being a positive integer, we 
deduce by Theorem~\ref{maintheopeuvstres} that when $p = 5$,
    $$\bar\rho_\Delta |_{G_{\Q_p}} \quad \text{is a tr\`es ramifi\'ee extension}, $$
whenever it is a non-split extension of $\omega$ by $\omega^2$.

We remark that  the formula for $\lambda$ above simplifies 
to the more na\"ive one occurring for $3 \leq b \leq p-1$ in Theorem~\ref{maintheoslopeone},
whenever ${r \choose 2} \equiv 0 \mod p$, and in particular when $p \> | \> r$, as was the case 
above. We now give an example to show that the more complicated formula above for $\lambda$ when
$b = 2$ is indeed required. Consider instead $\Delta_{16}$, the unique cusp form of level 1 and weight $k= 16$, and
again take $p = 5$. Then $p \> \| \> a_p = 52,110$, so $v(a_p) = 1$.  Also $r = 14$ so
$b = 2$ as before, and
$$\lambda \quad \equiv \quad 
                          -\frac{2}{12} \left( 10,422 - \frac{ 7 \cdot 13} {10,422} \right) \quad \equiv \quad 1 \mod p,$$ compared to the more na\"ive (and incorrect) $\lambda = 3$,  and we again recover the
result that  $(\bar \rho_{\Delta_{16}} |_{G_{\Q_p}})^{ss} \cong \omega^2 \oplus \omega$, for $p = 5$,
from \cite{Serre73}. We also remark that Theorem~\ref{maintheopeuvstres} provides no extra information in this example, since $r \equiv 2/3 \mod p$.

For other examples, for small values of $k$ and $p$, we refer the reader to \cite{Roz16}, where an 
algorithm to compute $\bar{V}_{k,a_p}^{ss}$ is described and implemented for all positive slopes.

\vspace{.3cm}
\noindent {\bf Acknowledgements:} 
The first author was supported by grant 1246/2014 from the Germany-Israel
Foundation for Scientific Research and Development.
This work was partly supported by the NSF under Grant No. 0932078 000,
while the second and third authors were in residence at MSRI during the
Fall 2014 semester. 
The second author would like to thank Princeton University, University
of British Columbia and \'ENS de Lyon, for their hospitality during the 
periods April-June 2015, October 2015 and June 2016. 
The third author was partially supported by the ANR grant PerCoLaTor (ANR 14-CE25).
Finally, we thank the referee for a careful reading of the paper 
and for many helpful suggestions to simplify our arguments in Section 7.

\end{document}